\documentclass[letter, 11 pt]{amsart}

\usepackage[hyphens]{url} 
\usepackage{amsthm}
\usepackage{graphicx}
\usepackage{bbm}
\usepackage[english]{babel}
\usepackage{fancyhdr}
\usepackage[T1]{fontenc}
\usepackage[utf8]{inputenc}
\usepackage{mathtools}
\usepackage{amsmath,amsfonts}
\usepackage{enumitem}
\usepackage{mathrsfs}
\usepackage{textcomp}
\usepackage{tikz-cd}
\usepackage{polski}
\usepackage{array}
\usepackage{amssymb}
\usepackage[new]{old-arrows}
\usepackage[bindingoffset=0.2in,%
            left=1.3in,right=1.3in,top=1.2in,bottom=1.2in,%
            footskip=.25in]{geometry}

\graphicspath{ {./images/} }

\theoremstyle{plain}
\newtheorem{thm}{Theorem}[section]
\newtheorem{prop}[thm]{Proposition}
\newtheorem{hypo}[thm]{Hypothesis}
\newtheorem{lem}[thm]{Lemma}
\newtheorem{cor}[thm]{Corollary}

\newtheorem{dfn}[thm]{Definition}
\newtheorem{rem}[thm]{Remark}
\numberwithin{equation}{section}
\newcommand{\R}{\mathbb{R}}
\newcommand{\C}{\mathbb{C}}
\newcommand{\Z}{\mathbb{Z}}

\newcommand{\M}{\mathcal{M}}
\newcommand{\g}{\mathfrak{g}}
\newcommand{\s}{\mathfrak{s}}

\newcommand{\G}{\mathcal{G}}
\newcommand{\A}{A^{\text{t}}}
\newcommand{\B}{B^{\text{t}}}

\newcommand{\Y}{\mathcal{Y}_{\mathcal{T}}}

\DeclareMathOperator{\im}{im}

\DeclareMathOperator{\Ima}{Im}
\DeclareMathOperator{\supp}{supp}
\DeclareMathOperator{\Hom}{Hom}
\DeclareMathOperator{\Map}{Map}

\DeclareMathOperator{\Met}{Met}
\DeclareMathOperator{\id}{id}

\DeclareMathOperator{\Spinc}{Spin^c}
\DeclareMathOperator{\spin}{spin}
\DeclareMathOperator{\ind}{ind}

\DeclareMathOperator{\sign}{sign}

\DeclareMathOperator{\tr}{tr}
\DeclareMathOperator{\Span}{Span}

\DeclareMathOperator{\End}{End}
\DeclareMathOperator{\Ad}{Ad}

\DeclareMathOperator{\Hess}{Hess}

\DeclareMathOperator{\grad}{grad}

\DeclareMathOperator{\vol}{vol}
\DeclareMathOperator{\cl}{cl}
\DeclareMathOperator{\Red}{red}

\DeclareMathOperator{\PDL}{PD}
\DeclareMathOperator{\Stab}{Stab}
\DeclareMathOperator{\SW}{SW}

\DeclareMathOperator{\Sf}{\widetilde{Sf}}

\mathchardef\mhyphen="2D

\begin{document}
\title[A Surgery Formula of $\lambda_{SW}$]{A Surgery Formula for the Casson-Seiberg-Witten Invariant of Integral Homology $S^1 \times S^3$}
\author{Langte Ma}
\address{MS 050 Department of Mathematics, Brandeis University, 415 South St., Waltham MA 02453}
\email{ltmafixer@brandeis.edu}

\begin{abstract}
We prove a surgery formula of the Casson-Seiberg-Witten invariant of integral homology $S^1 \times S^3$ along an embedded torus, which could either be regarded as an extension of the product formula for Seiberg-Witten invariants or a manifestation of the surgery exact triangle in $4$-dimensional Seiberg-Witten theory of homology $S^1 \times S^3$. As an application, we compute this invariant for mapping tori of $3$-manifolds under diffeomorphisms of finite order and fixed-point set being a simple closed curve. This computation generalizes the result of Lin-Ruberman-Saveliev in \cite{LRS1}. 
\end{abstract}

\maketitle

\section{Introduction}
In \cite{MRS} Mrowka-Ruberman-Saveliev defined the Casson-Seiberg-Witten invariant $\lambda_{SW}(X)$ for a $4$-manifold $X$ with $H_*(X; \Z) \cong H_*(S^1 \times S^3; \Z)$ by introducing an index-theoretical correction term compensating for the jump of the count of irreducible monopoles along a generic path of metrics and perturbations. A key feature of $\lambda_{SW}(X)$ is that when $X=S^1 \times Y$ is given by the product of a circle with an integral homology sphere $Y$, one has 
\begin{equation}
\lambda_{SW}(S^1 \times Y) = - \lambda(Y),
\end{equation}
where $\lambda(Y)$ is the Casson invariant. Many applications of the Casson invariant in low-dimensional topology are derived from the surgery formula. Thus one natural question to ask is that whether there are some surgery formula for $\lambda_{SW}(X)$. The main purpose of this paper is to provide an answer for this question. 

Here we give a brief description, and a detailed one later in Section \ref{setup}. Let $\mathcal{T} \hookrightarrow X$ be an embedded torus with a primitive class in $H_1(\mathcal{T}; \Z)$ generating $H_1(X; \Z)$. Denote by $\nu(\mathcal{T})$ a regular neighborhood of $\mathcal{T}$ in $X$, and $M=X \backslash \nu(\mathcal{T})$ its complement. After fixing a framing $\nu(\mathcal{T}) \cong D^2 \times T^2$, we get a preferred choice of basis for $H_1(\partial \nu(\mathcal{T}); \Z)$ which we denote by $\{ \mu, \lambda, \gamma\}$. Given a relatively prime pair $(p, q)$ the analogue of $4$-dimensional Dehn surgery with coefficients $(p, q)$ along $\mathcal{T}$ results in a manifold 
\begin{equation}
X_{p, q}= M \cup_{\varphi_{p, q}} D^2 \times T^2,
\end{equation}
where the diffeomorphism on the boundary is 
\begin{equation}
\varphi_{p, q}=
\begin{pmatrix}
p & r & 0 \\
q & s & 0 \\
0 & 0 & 1
\end{pmatrix}
\in SL(3, \Z)
\end{equation}
under the preferred basis. Note that $H_*(X_{0, 1}; \Z) =H_*(T^2 \times S^2; \Z)$, thus $b^+(X_{0,1})=1$. We write 
\begin{equation}
\mathcal{SW}(X_{0, 1})=\sum_{\s_0 \in \Spinc(X_{0, 1})} \SW(X_{0, 1}, \s_0),
\end{equation}
where for each spin$^c$ structure $\s_0 \in \Spinc(X_{0, 1})$ the Seiberg-Witten invariant $\SW(X_{0,1}, \s_0)$ is computed using the chamber specified by small perturbations.

\begin{thm}\label{Main1}
After fixing appropriate homology orientations, for any $q \in \Z$ one has
\begin{equation}\label{LAT}
\lambda_{SW}(X_{1, q+1})  - \lambda_{SW}(X_{1, q}) = \mathcal{SW}(X_{0, 1}).
\end{equation}
\end{thm}

From (\ref{LAT}) we deduce that for any $q \in \Z$
\begin{equation}
\lambda_{SW}(X_{1, q}) = \lambda_{SW}(X)+ q \mathcal{SW}(X_{0, 1}).
\end{equation}
When $X=S^1 \times Y$ with $Y$ an integral homology sphere, we can take the embedded torus to be $S^1 \times K$ with $K \subset Y$ an embedded knot. Then (\ref{LAT}) recovers the surgery formula for the Casson invariants: 
\begin{equation}
\lambda(Y_{1 \over {q+1}}(K)) - \lambda(Y_{1 \over q}(K)) = {1 \over 2}\Delta''_{K \subset Y}(1),
\end{equation}
where $\Delta_{K \subset Y}(t)$ is the symmetric Alexander polynomial of $K$ in $Y$, $Y_{1 \over q}(K)$ is obtained by performing ${1\over q}$ Dehn surgery along $K$. On the other hand $(\ref{Main1})$ can be regarded as a generalization of the product formula of the Seiberg-Witten invariant in \cite{MMS} in the case when $b^+(X)=0$. 

As an application, the surgery formula (\ref{LAT}) can be used to deduce formulae for $\lambda_{SW}$ of homology $S^1 \times S^3$'s arisen from mapping tori of $3$-manifolds under finite order maps whose fixed-point set is a simple closed curve. 

To be more precise, let $Y$ be an integral homology sphere, $K \subset Y$ an embedded knot. Denote by $Y^n$ the $n$-fold branched cover of $Y$ along $K$, and $\tau_n: Y^n \to Y^n$ the covering transformation of order $n$. Let $X^n$ be the mapping torus of $Y^n$ under $\tau_n$. It's straightforward to check that $X^n$ is an integral homology $S^1 \times S^3$, see for instance \cite[Chapter 8]{BZ}. 

\begin{prop}\label{lmp}
Let $Y, K, X^n$ be as above. Then
\begin{equation}\label{lmp1}
\lambda_{SW}(X^n)=-n \lambda(Y) - {1 \over 8}\sum_{m=0}^{n-1} \sign^{m/n}(K),
\end{equation}
where $\sign^{m/n}(K)$ is the Tristram–Levine signature of $K$, $\lambda(Y)$ is the Casson invariant of $Y$.
\end{prop}

The formula (\ref{lmp1}) is conjectured by Lin-Ruberman-Saveliev \cite{LRS1} in the case when $Y^n$ is a rational homology sphere, and proved in the case when $n=2$. Combining with their splitting formula \cite{LRS}, one is able to obtain the Lefschetz number on the reduced monopole Floer homolgy $HM^{\Red}(Y^n)$ when $Y^n$ is a rational homology spehre. On the other hand, when $Y^n$ is not a rational homology sphere, (\ref{lmp1}) provides us with the computation for $\lambda_{SW}$ with the generating hypersurface having nonvanishing first betti number, which are the first known ones with nonvanishing $\lambda_{SW}$. 

\begin{cor}\label{6bc}
Let $K \subset S^3$ be either the right-handed, left-handed trefoil, or the figure-eight knot, $\Sigma^6(K)$ the $6$-fold branched cover of $S^3$ along $K$, and $X^6(K)$ the corresponding mapping torus of $\Sigma^6(K)$. Then 

\[
\lambda_{SW}(X^6(K)) =\left \{ \begin{array}{ll}
1 & \mbox{$K$ is the right-handed trefoil} \\
0 & \mbox{$K$ is the figure-eight knot} \\
-1 & \mbox{$K$ is the left-handed trefoil}
  	\end{array}
\right.
\]
\end{cor}

\begin{rem}
In \cite{MRS}, Mrowka-Ruberman-Saveliev proved that $\lambda_{SW}(X) \mod 2$ is the Rohlin invariant of the generating hypersurface $Y$ of $H^1(X; \Z)$ corresponding to the spin structure induced from that on $X$. The computation in Corollary \ref{6bc} is also consistent with this fact. 
\end{rem}

The proof of Proposition \ref{lmp} and Corollary \ref{6bc} is given in Section \ref{app}. The knots considered in Corollary \ref{6bc} are the three genus-$1$ fibered knots in $S^3$, whose monodromies all have order $6$ as a self-diffeomorphism on a punctured torus. It follows that $\Sigma^6(K)$ has the same homology as that of $S^1 \times S^2 \# S^1 \times S^2$.  In general one can consider a fibered knot $K \subset S^3$ of genus $g>0$ whose monodromy has order $n$. Then $\Sigma^n(K)$ has the same homology as that of $\#^{2g} S^1 \times S^2$. One then computes the Tristram-Levine signature using the Seifert matrix of $K$, thus $\lambda_{SW}(X^n(K))$. 

\subsection{Outline} Here we outline the contents of this paper. Section \ref{Pre} briefly reviews  Seiberg-Witten theory, definition of the Casson-Seiberg-Witten invariant $\lambda_{SW}(X)$, and gives a detailed exposition of the surgery construction along a torus. Section \ref{2} derives the structure of Seiberg-Witten moduli space for manifolds with cylindrical ends with $b^+=0$. The main result is Theorem \ref{KUR}, which we believe is essentially known to experts, but we couldn't find results documented in the literature applicable to our situation. The main issue in our case is that all manifolds of our interests have $b^+=0$ which means that the appearance of reducible locus is inevitable. We need to make sure one can choose nice and generic perturbations so that the structure of moduli space and gluing achieve one's expectation. Section \ref{Peri} develops the technique of periodic spectral flow, which was originally defined in \cite{MRS} to compute the difference of indices of two Dirac operators over a manifold with periodic end. We extend the notion to manifolds with cylindrical end, and prove a splitting formula in Theorem \ref{sf4.14}. Section \ref{Exci} proves the excision principle of elliptic operator over manifolds with periodic ends. The main point is to stretch the excision region, and use Taubes' Fredholmness criterion \cite[Lemma 4.3]{T1} for operators over end-periodic manifolds to generalize the standard argument. Section \ref{SurFor} incorporates all techniques developed in the previous sections to prove the surgery formula. Section \ref{app} applies the surgery formula to prove Proposition \ref{lmp}. In the end we also include an appendix on the gluing theorem on our set-up of moduli spaces for integral homology $S^1 \times S^3$.

\subsection*{Acknowledgement} The author would like to express gratitude to his advisor Daniel Ruberman for suggestion on this work and constant support during the preparation for the draft. He also wants to thank Jianfeng Lin for a lot of enlightening discussions and joyful moments.

\section{Preliminaries}\label{Pre}

\subsection{The Seiberg-Witten Moduli Space}

Let $(M, g)$ be a smooth oriented Riemannian $4$-manifold with with boundary $\partial M=Y$, where $(Y, h)$ is a smooth oriented closed Riemannian $3$-manifold with metric $h$. We assume that  in a collar neighborhood $(-1, 0] \times Y$ of $Y$, the metric $g$ has the form $g|_{(-1, 0] \times Y} =dt^2 +h$, where $t$ is the coordinate on $(-1, 0]$. We form a Riemannian manifold $(M_{\infty}, g_{\infty})$ with cylindrical end by attaching a cylindrical end $(-\infty, 0] \times Y$ to the manifold with boundary $M$, i.e. $M_{\infty}= M \cup [0, \infty) \times Y$ with 
\[
g_{\infty}|_M=g, \; \; g_{\infty}|_{(-\infty, 0] \times Y} = dt^2+h.
\]
By an abuse of notation we write $g_{\infty}=g$. We write $\pi: (-\infty, 0] \times Y \to Y$ for the projection map. The coordinate of the first factor $(-\infty, 0]$ is denoted by $t$. We write $Y_t=\{t \} \times Y$ for the $t$-slice of the cylindrical end. 

Let $\mathfrak{s}=(W, \rho_M)$ be a spin$^c$ structure over $M$, where $W=W^+\oplus W^-$ is a $U(4)$-bundle over $M$, $\rho_M: T^*M \to \End(W)$ is the Clifford multiplication. Write $\mathfrak{t}=\mathfrak{s}_M|_Y$ for the restricted spin$^c$ structure over $Y$. Explicitly the restricted spin$^c$ structure over $Y$ can be expressed in terms of a pair $\mathfrak{t}=(S, \rho)$ by identifying $S=W^+|_{\{0\}\times Y}$ and $\rho=-\rho_M(dt) \cdot \rho_M$. Then the spin$^c$ structure $\s$ extends to a spin$^c$ structure, still denoted by $\s$, over $M_{\infty}$ of the form 
\[
W^+|_{[0, \infty) \times Y} = W^-|_{[0,\infty) \times Y} =\pi^*S,
\]
and 
\[
\rho_M(dt)=
\begin{pmatrix}
0 & -1 \\
1 & 0 
\end{pmatrix}, 
\rho_M(e_t)=
\begin{pmatrix}
0 & -\rho^*(e_t) \\
\rho(e_t) & 0
\end{pmatrix},
\]
where $e_t \in T^*Y_t$. 

Fix a positive integer $k \geq 2$. We denote by $\mathcal{A}_k(\s)$ the space of  $L^2_{k, loc}$ spin$^c$-connections on $W^+$, $\Gamma_k(W^+)$ the space of $L^2_{k, loc}$ sections which we refer as spinors. We write $\mathcal{C}_k(\s)=\mathcal{A}_k(\s) \times \Gamma_k(W^+)$ for the configuration space. When there is no ambiguity we omit $\s$ in the notations. The perturbation that we shall use in this paper comes from the space $\mathcal{P}:=L^2_{k, c}(T^*M_{\infty} \otimes i\R)$ consisting of purely imaginary-valued $L^2_k$ $1$-forms on $M_{\infty}$ with compact support. Given $\beta \in \mathcal{P}$, $d^+\beta \in  L^2_{k-1}(\Lambda^+M_{\infty}\otimes i\R)$. Consider the perturbed Seiberg-Witten map 
\begin{align*}
\mathfrak{F}_{\beta}: \mathcal{A}_k \times \Gamma_k(W^+) & \to L^2_{k-1, loc}(\Lambda^+M_{\infty}\otimes i\R) \oplus \Gamma_{k-1} (W^-) \\
(A ,\Phi) & \mapsto ({1 \over 2} (F^+_{\A}-4d^+\beta)- \rho_M^{-1} (\Phi \Phi^*)_0, D^+_A \Phi), 
\end{align*}
where $(\Phi\Phi^*)_0= \Phi \otimes \Phi^* - {1 \over 2}\tr(\Phi \otimes \Phi^*) \in i\mathfrak{su}(W^+)$, and $\A$ is the connection on $\det W^+$ induced by $A$. 
The perturbed Seiberg-Witten equation is
\begin{equation}\label{4SW}
\mathfrak{F}_{\beta}(A,\Phi)=0.
\end{equation}
The gauge group for $(M_{\infty}, \s)$ is denoted by $\G_{k+1}(\s)$, which consists of $L^2_{k+1, loc}$ maps $M_{\infty} \to S^1$ with the gauge action given by 
\[
u\cdot (A, \Phi)= ( A - u^{-1} du \otimes 1_{W^+}, u \cdot \Phi).
\]
Note that on the determinant line bundle $\det W^+$ the action is even: $ (u\cdot A)^{\text{t}}=\A - 2u^{-1}du$. Given a configuration $(A, \Phi) \in \mathcal{A}_k \times \Gamma_k(W^+)$, following \cite[(4.16)]{KM1} we define its energy to be 
\[
\mathcal{E}_{\beta}(A, \Phi)={1 \over 4}\int_{M_{\infty}} |F_{\A}-4d\beta|^2 +\int_{M_{\infty}} |\nabla_A \Phi|^2 + {1 \over 4} \int_{M_{\infty}} (|\Phi|^4 + s^2 |\Phi|^2), 
\]
where $s$ stands for the scalar curvature of $({M_{\infty}}, g)$. Now we define the perturbed finite energy Seiberg-Witten moduli space to be 
\begin{equation}
\M_{g, \beta}(M_{\infty}, \s)=\{(A , \Phi) \in \mathcal{C}_k: \mathfrak{F}_{\beta}(A, \Phi)=0, \mathcal{E}_{\beta}(A, \Phi) < \infty \}/\G_{k+1}.
\end{equation}
Since gauge transformations preserve the energy $\mathcal{E}(A, \Phi)$, the moduli space is well-defined. It is well-known that the homeomorphism type of the moduli space is independent of $k$ once $k$ is sufficiently large, say $k \geq 2$. Thus we have omitted $k$ in the notation. Any element in the moduli space is referred as a monopole. We call a monopole of the form $[A, 0]$ reducible, otherwise irreducible. Then the moduli space decomposes into two parts accordingly 
\[
\M_{g, \beta}(M_{\infty}, \s)=\M^*_{g, \beta}(M_{\infty}, \s) \cup \M^{\Red}_{g, \beta}(M_{\infty}, \s).
\]

Now let $l=k-{1 \over 2}$. Over the end $[0, \infty) \times Y$, a configuration $(A, \Phi)$ has the form $A=B(t)+ c(t) dt$, where $B(t)$ is a time-dependent spin$^c$ connection on $S$, and $c(t) \in L^2_{l}(Y, i\R)$ is a time-dependent imaginary-valued function on $Y$. The spinor $\Phi$ gives rise to a time-dependent section on $Y$ which we denote by $\Psi(t)=\Phi|_{Y_t} \in \Gamma_l(S)$. Under these identifications we have 
\begin{equation}\label{eq1.2}
\begin{split}
F^+_{\A} & ={1 \over 2}(F_{\B}+*(\dot{B}^{\text{t}}-2dc) +dt \wedge (*F_{\B} + \dot{B}^{\text{t}} -2dc)) \\
D^+_A & = D_B  + {d \over {dt}} +c.
\end{split}
\end{equation}
Moreover
\[
\rho_M(F_{\A}^+)= -\rho(*F_{\B} + \dot{B}^{\text{t}}-2dc).
\]
Thus the unperturbed 4-dimensional Seiberg-Witten equation on the cylindrical end reads as
\begin{equation}\label{CSW}
\begin{split}
\dot{B} & = -({1 \over 2} *F_{\B} + \rho^{-1}(\Psi\Psi^*)_0) \otimes 1_{S} + dc \otimes 1_{S} \\
\dot{\Psi} & = -D_B \Psi - c\Psi.
\end{split}
\end{equation}

Now let's shift attention to $(Y, \mathfrak{t})$. Suppose $\mathfrak{t}$ is torsion, i.e. $c_1(S) \in H^2(Y ;\Z)$ is torsion. Fix a flat connection $B_0$ on $S$. Denote by $\mathcal{C}_l(\mathfrak{t}):= \mathcal{A}_l(\mathfrak{t}) \times \Gamma_l(S)$ the configuration space. We follow \cite{KM1} to define the Chern-Simons-Dirac functional on $\mathcal{C}_l(\mathfrak{t})$ as 
\begin{equation}
\mathcal{L}(B, \Psi)=-{1 \over 4}\int_{Y} b \wedge F_{B^{\text{t}}}+ {1 \over 2} \int_{Y} \langle D_B \Psi, \Psi \rangle d \vol, 
\end{equation}
where $B-B_0=b \otimes 1_{S}$, $\Psi \in \Gamma_l(S)$. It's straightforward to compute that the gradient of the Chern-Simons-Dirac functional $\grad \mathcal{L}: \mathcal{C}_l(\mathfrak{t}) \to T\mathcal{C}_l( \mathfrak{t})$ is given by 
\[
\grad \mathcal{L}|_{(B, \Psi)}=(({1 \over 2}*F_{B^{\text {t}}}+ \rho^{-1}(\Psi \Psi^*)_0) \otimes 1_{S}, D_B \Psi). 
\] Thus the critical points of the Chern-Simons-Dirac functional are given by solutions of the 3-dimensional Seiberg-Witten equations:
\begin{equation}\label{3SW}
\begin{split}
{1 \over 2} F_{B^\text{t}}-\rho^{-1}( \Psi \Psi^*)_0 & =0, \\
D_B \Psi & =0.
\end{split}
\end{equation}
For simplicity we write this equation as $\mathcal{F}(B, \Psi)=0$, where 
\[
\mathcal{F}: \mathcal{A}_l(\mathfrak{t}) \times \Gamma_{l}(S) \to L^2_{l-1}(i\mathfrak{su}(S)) \oplus \Gamma_{l-1}(S)
\]
is defined by the left hand side of (\ref{3SW}). Using parallel transport of $A$ to trivialize $W^+$ in the $t$-direction, we get a path $\gamma: [0, \infty) \to \mathcal{C}_l(\mathfrak{t})$ written as $\gamma(t)=(B(t), \Psi(t))$, given by $A=B(t), \Psi(t)=\Phi|_{Y_t}$.  We refer to such a connection $A$ as in temporal gauge. From (\ref{CSW}) the $4$-dimensional unperturbed Seiberg-Witten equation on the cylinder now is equivalent to the downward gradient equation 
\[
{d \gamma \over {dt}}=-\grad \mathcal{L}(\gamma(t)).
\]
The energy of a monopole $(A, \Phi)$ interacts nicely with the Chern-Simons-Dirac functional in the following sense. Over a finite part of the cylinder $([t_0, t_1] \times Y, \s)$, from \cite[(4.18)]{KM1} the drop of Chern-Simons-Dirac functional is proportional to the energy of $(A, \Phi)$: 
\begin{equation}
\mathcal{L}(\gamma(t_0)) - \mathcal{L}(\gamma(t_1)) = {1 \over 2} \mathcal{E}(A, \Phi). 
\end{equation}
Let $\mathcal{G}_{l+1}(\mathfrak{t})=\Map(Y, S^1)$ be the gauge group with action
\[
u\cdot (B, \Psi)  = (B- u^{-1}du \otimes 1_S, u \cdot \Psi),
\]
and $\mathcal{B}_l(\mathfrak{t})=\mathcal{C}_l(\mathfrak{t}) / \G_{l+1}(Y)$ the quotient configuration space. Note that since $c_1(S)$ is torsion, we have 
\[
\mathcal{L}(u \cdot (B, \Psi)) - \mathcal{L}(B, \Psi)=2\pi^2 ([u] \smallsmile c_1(S))[Y]=0.
\]
Thus the Chern-Simons-Dirac functional descends to the quotient space 
\[\mathcal{L}: \mathcal{B}(Y, \mathfrak{t}) \to \R.\]
The Seiberg-Witten moduli space of $(Y, \mathfrak{t})$ is 
\[
\mathcal{M}_h(Y, \mathfrak{t})=\{ (B, \Psi) \in \mathcal{C}_l : \mathcal{F}(B, \Psi)=0 \} / \mathcal{G}_{l+1}.
\]
Analogous to the $4$-dimensional case, we can decompose the moduli space into two parts:
\[
\mathcal{M}_h(Y, \mathfrak{t})=\mathcal{M}^*_h(Y, \mathfrak{t}) \cup \mathcal{M}^{\Red}_h(Y, \mathfrak{t})
\]
Recall that the spin$^c$ structure $\mathfrak{t}$ is torsion, thus the reducible part $\M^{\Red}_h(Y, \mathfrak{t})$ consists of gauge equivalence classes of  flat connections on $\det S$. Note that the gauge action is even: $u \cdot B^{\text{t}}= B^{\text{t}}-2u^{-1}du$. Then $\mathcal{M}_h(Y, \mathfrak{t})$ is identified with the character variety $\chi(Y):= \Hom(\pi_1(Y), U(1))/ \Ad$. 

\subsection{Casson-Seiberg-Witten Invariant $\lambda_{SW}$}
We briefly review the definition of the Casson-Seiberg-Witten invariant $\lambda_{SW}$ which was originally defined in \cite{MRS}. Let $(X, g)$ be a closed Riemannian oriented smooth $4$-manifold with $H_*(X; \Z) \cong H_*(S^1 \times S^3; \Z)$. Equip $X$ with one of the spin structures $\s$, whose induced spin$^c$ structure is denoted by $\s$ as well. Let's fix a generator $1_X \in H^1(X; \Z)$ given by $1_X=[df]$ for some smooth function $f: X \to S^1$. If $d\theta \in H^1(S^1; \Z)$ is the fundamental class, then $[df]=[f^*d\theta]$  We define the Seiberg-Witten moduli space $\M_{g, \beta}(X, \s)$ the same as in the case of manifolds with cylindrical end as above. 

\begin{dfn}\label{R1}
 Choosing a branch to represent the value of $\ln z$, we call a pair $(g, \beta)$ regular if the family of Dirac operators 
\begin{equation}\label{1.6.1}
D^+_{z, \beta}(X, \s)=D^+(X, \s) +\rho_X(\beta- \ln z \cdot df), \; |z|=1,
\end{equation}
have trivial kernel. 
\end{dfn}

For different choice of branches, the above operators in (\ref{1.6.1}) differ by the conjugation of $e^{2\pi k f}$. Thus $(g, \beta)$ being regular is well-defined. It is proved in \cite[Proposition 2.2]{MRS} that a generic pair $(g, \beta)$ is regular. Here being generic means that such choices of $(g, \beta)$ are form a set of countable intersections of open dense subsets in the space of $\Met(X) \times \Omega^1(X; i\R)$. The first part of $\lambda_{SW}$ is given by  counting irreducible monopoles $\# \M^*_{g, \beta}(X, \s)$ for a generic pair $(g, \beta)$. 

The second part of $\lambda_{SW}$ involves a correction term $\omega(X, g, \beta)$ defined as follows. Let $Y_X \subset X$ be a embedded hypersurface given by a regular value of $f$. Thus as a homology class $[Y_X]=\PDL 1_X$. We refer to $Y_X$ as a generating hypersurface of $X$. Since $Y_X$ is primitive, we can choose $Y_X$ to be connected and nonseparating. $Y_X$ inherits a spin$^c$ structure $\mathfrak{t}$ and metric $h$ from those of $X$. Note that $\s$ comes from a $\spin$ structure over $X$, thus $\mathfrak{t}$ also comes from a spin structure on $Y_X$. Cutting $X$ along $Y_X$ results in a spin cobordism $W: Y_X \to Y_X$ with $\partial W=-Y_X \cup Y_X$. Taking an arbitrary spin $4$-manifold $Z$ with spin boundary $\partial Z =Y_X$, we form the end-periodic spin manifold 
\[
Z_+(X)=Z \cup W_0 \cup W_1 \cup ...,
\]
where $W_i$ is a copy of $W$. Lifting the metric and perturbation pair $(g, \beta)$ to the periodic end $W_+=\cup_{i \geq 0} W_i$ and extending arbitrarily over $Z_+(X)$, we define the twisted Dirac operator to be 
\[
D^+_{\beta}(Z_+, g) = D^+(Z_+, \s_Z, g) + \rho_Z(\beta) :L^2_1(Z_+, W^+) \to L^2(Z_+, W^-). 
\]
\cite[Theorem 3.1]{MRS} says that for any regular pair $(g, \beta)$ the twisted Dirac operator $D_{\beta}^+(Z_+, g)$ is Fredholm. Thus we are able to take its index and define the correction term to be
\[
\omega(X, g ,\beta) = \ind_{\C} D^+_{\beta}(Z_+, g) + {\sigma(Z) \over 8},
\]
where $\sigma(Z)$ is the signature of $Z$. The Casson-Seiberg-Witten invariant is 
\[
\lambda_{SW}(X) : = \# \M_{g,\beta}(X, \s) - \omega(X, g, \beta), \; (g, \beta) \text{ a regular pair}.
\]
It is proved in \cite{MRS} that $\lambda_{SW}(X)$ is independent of the choice of the regular pair $(g, \beta)$. 

\subsection{Surgery Description}\label{setup}
Now we give a detailed exposition of the surgery construction we are interested in. Let $X$ be a homology $S^1 \times S^3$ as above. Let $\iota: T^2 \hookrightarrow X$ be an embedding of an oriented torus, which induces a surjection on the first homology, i.e. $\iota_*: H_1(T^2; \Z) \to H_1(X; \Z)$ is surjective. We denote its image by $\mathcal{T}:= \im \iota$, and $\nu(\mathcal{T})$ a tubular neighborhood of $\mathcal{T}$ in $X$.Due to the vanishing of the intersection form on $X$, $\nu(\mathcal{T})$ is the trivial disk bundle on $T^2$. Write $M=\cl(X \backslash \nu(\mathcal{T}))$ to be the closure of the complement of $\nu(\mathcal{T})$. The assumption on the homology of $X$ gives us that $H_*(M; \Z) = H_*(D^2 \times T^2; \Z)$. Let's pick a framing $\nu(\mathcal{T}) \cong D^2 \times T^2$. Under this identification we pick up representatives of $H_1(\partial \nu(\mathcal{T});\Z)$ as 
\[
\mu=\partial D^2 \times \{pt. \} \times \{ pt.\}, \lambda=\{pt. \} \times S^1 \times \{pt.\}, \gamma=\{pt\} \times \{pt.\} \times S^1.
\] 
We call $\mu$ the meridian, $\lambda$ the longitude, and $\gamma$ the latitude. Moreover we require the framing is chosen so that the longitude $[\lambda]$ generates the kernel of the inclusion map $i_*: H_1(\partial M; \Z) \to H_1(M; \Z)$. However there are still $\Z$-worth  ambiguities for the choices of the framing, since one can always add to $[\gamma]$ by a multiple of $[\lambda]$. Note that the latitude $\gamma$ generates $H_1(X; \Z)$ under the inclusion $\nu(\mathcal{T}) \hookrightarrow X$. Let $n$ be the outward normal vector field along $\nu(\mathcal{T})$. Since the torus $T$ is oriented, we get orientations on $\mu$ and $\gamma$ automatically. We orient $\lambda$ so that the ordered basis $\langle n,  \mu,  \lambda ,\gamma \rangle$ is the orientation of $\nu(\mathcal{T})$ induced from that of $X$. The boundary orientation of $\partial M$ is given by an ordered basis $\langle \mu, -\lambda, \gamma\rangle$. 

Now take another copy of $D^2 \times T^2$ with boundary orientation given by the basis as above $\{\mu_0, \lambda_0, \gamma_0\}$. 
\begin{dfn}\label{RST}
Given a relatively prime pair $(p ,q) \in \Z \oplus \Z$, the $(p,q)$ Dehn surgery along an embedded $T^2$ in $X$ as above results in a manifold $X_{p, q}$ given by
\begin{equation}
X_{p, q} = M \cup_{\varphi_{p, q}} D^2 \times T^2,
\end{equation}
where $\varphi_{p, q}: \partial D^2 \times T^2 \to \partial M$ is an orientation-reversing diffeomorphism whose isotopy class in $-SL(3; \Z)$ is given by the matrix 
\begin{equation}
\varphi_{p, q}=
\begin{pmatrix}
p & r & 0 \\
q & s & 0 \\
0 & 0 & 1
\end{pmatrix}
\end{equation}
under the basis $\{\mu_0, \lambda_0, \gamma_0\}$ of $\partial D^2 \times T^2$ and $\{\mu, -\lambda, \gamma\}$ of $\partial M$ respectively. 
\end{dfn}

\begin{rem}\hfill
\begin{enumerate}
\item[\upshape (i)] Note that $H_*(X_{1, q}; \Z) \cong H_*( S^1 \times S^3; \Z)$ and $H_*(X_{0, 1}; \Z) \cong H_*(T^2 \times S^2; \Z)$. We write $X_q=X_{1, q}$ for $q \neq 0$, and $X_0 =X_{0,1 }$. It's clear that $X=X_{1, 0}$.
\item[\upshape (ii)] Since the manifold $M \cup_{\varphi} D^2 \times T^2$ is determined by the image of $[\varphi({\mu})] \in H_1(\partial M; \Z)$ and the choices for $[\mu]$ and $[\lambda]$ are canonical, we see that $X_{p, q}$ is independent of the choice of the framing caused by the ambiguity from $\gamma$. 
\end{enumerate}
\end{rem}

\begin{lem}There exists a generating hypersurface $Y_X \subset X$ so that $Y_X$ intersects $\mathcal{T}$ transversely into a knot $K$.
\end{lem}

\begin{proof}
The choice of the basis above identifies $\mathcal{T} \cong S^1 \times S^1$ so that the algebraic intersection between $Y_X$ and $\{pt.\} \times S^1$ is $1$. By attaching $1$-handles on the alternating intersection points to $Y_X$ we can arrange that their geometric intersection is also $1$. Due to the choice of the longitude $\lambda$, it follows that $Y_X \cap \mathcal{T}$ is the union of a copy of $S^1 \times \{ pt,\} \subset \mathcal{T}$ and several null-homotopic loops. We then attach $2$-handles on $Y_X$ along those loops to reduce the number of components of $Y_X \cap \mathcal{T}$ to $1$. 
\end{proof}
Chosen $Y_X$ as above, we see that performing an $n$-surgery along $\mathcal{T} \subset X$ amounts to performing an $n$-surgery along $K \subset Y_X$ with respect to the framing induced from that of $\mathcal{T}$.

\subsection{The Neck-Stretching Set-Up}\label{nssu}
In the proof of our main result Theorem \ref{Main1}, a neck-stretching construction  will be applied in several places, i.e. the splitting formula Theorem \ref{sf4.14} for the periodic spectral flow, the periodic excision principle Theorem \ref{EP}, and the compactness result Theorem \ref{G1}. So we would like to fix the notations that will be used in the neck-stretching process. 

Let $X$ be a homology $S^1 \times S^3$ decomposed along a hypersurface $Y$ (a generic seperating embedded $3$-manifold, not the generating hypersurface $Y_X$) as 
\[
X=M \cup N
\] 
with $\partial M=-\partial N= \cap N=Y$. We identify a neighborhood of $Y$ in $X$ as $(-1, 1) \times Y$. In this way neighborhoods of $Y$ in $M$ and $N$ are identified as $(-1, 0] \times Y$ and $[0, 1) \times Y$ respectively. Given a metric $h$ on $Y$, let's denote by $\Met(X, h)$ the space of smooth metrics $g$ on $X$  whose restriction to the neighborhood of $Y$ have the form
\[
g|_{(-1, 1) \times Y} = dt^2 +h. 
\]
Given $T >0$ we write 
\[
\begin{split}
&M_T=M \cup [0, T] \times Y,  \; N_T=[-T, 0] \times Y \cup N, \\
&X_T=M \cup [-T, T] \times Y \cup N, \; I_T=[-T, T] \times Y.
\end{split}
\]
For the corresponding geometric limits we write 
\[
\begin{split}
& M_{\infty}=M \cup [0, \infty) \times Y, \; N_{\infty}=(- \infty, 0] \times Y \cup N,  \\
&X_{\infty}=M_{\infty} \sqcup N_{\infty}, \; I_{\infty}=[0, \infty) \times Y \cup (-\infty, 0] \times Y.
\end{split}
\]
We embed $M_T \hookrightarrow X_T$ by identifying $[0, T] \times Y \subset M_T$ with $[-T, 0] \times Y \subset X_T$. Similarly we have $N_T \hookrightarrow X_T$. Any metric $g \in \Met(X, h)$ naturally extends to metrics on all the neck-stretched manifolds above, which we still denote by $g$ when there is no confusion.  As for perturbations, we let $\beta_T=\beta_{M, T} +\beta_{N, T} \in \Omega^1(X_T; i\R)$ with $\supp \beta_{M, T} \subset M$ and $\supp \beta_{N, T} \subset N$ be a family of perturbations for $T \in [0, \infty)$ such that $\beta_{M, T}$ converges to some $\beta_M \in \Omega_c^1(M_{\infty}; i\R)$ in $L^2_{k}$-topology. Similarly $\beta_{N, T}$ converges to $\beta_N \in \Omega_c^1(N_{\infty}; i\R)$ in $L^2_{k}$-topology. Moreover the fixed generator $1_X \in H^1(X; \Z)$ gives rise to generators $1_{X_T} \in H^1(X_T; \Z)$. We choose functions $f_T: X_T \to S^1$ so that $[df_T]=1_{X_T}$. 

Let $\s$ be a spin$^c$ structure over $X$ coming from a fixed spin structure, whose restriction on $Y$ is denoted by $\s|_{\{0 \} \times Y}=\mathfrak{t}$. The spin$^c$ structure also extends to various manifolds with stretched neck. By an abuse of notation, we still write $\s$ for the extended spin$^c$ structures on $X_T, X_{\infty}$, $\s_M$ on $M_T, M_{\infty}$, and $\s_N$ on $N_T, N_{\infty}$. The following result from \cite{LRS} will be useful in several places in this paper. 

\begin{lem}\label{3.8.8}
(\cite[Proposition 7.3]{LRS})
Suppose the spin Dirac operators \[
D^+(M_{\infty}, \s_M, g): L^2_1(M_{\infty}, W^+) \to L^2(M_{\infty}, W^-)\] 
and 
\[
D^+(N_{\infty}, \s_M, g): L^2_1(N_{\infty}, W^+) \to L^2(N_{\infty}, W^-)
\]
are both invertible. Then there exists $T_1 >0$ and $\epsilon_1 >0$ such that for all $T \geq T_1$ the family of operators
\[
D^-_z(X_T) D^+_z(X_T): L^2_2(X_T, W^+) \to L^2(X_T, W^+), \; |z|=1
\]
has no eigenvalues in $[0, \epsilon_1^2)$. 
\end{lem}

This result motivates us to introduce the following notion. 
\begin{dfn}\label{dfad}
We say a metric $g \in \Met(X, h)$ admissible with respect to the decomposition $X=M \cup N$ if the spin Dirac operator
\begin{equation}
D^+(X_{\infty}, \s, g): L^2_1(X_{\infty}, W^+) \longrightarrow L^2(X_{\infty}, W^-)
\end{equation}
is an isomorphism. 
\end{dfn}
 
\begin{rem}\label{read}\hfill
\begin{enumerate}
\item[\upshape (i)] If we choose the perturbation $1$-form $\beta_T$ with small $L^2_k$-norm, then $D^+_{z, \beta}(X_T)$ is also invertible for all $z$ with $|z|=1$ when $g$ is admissible.
\item[\upshape (ii)] From Theorem 10.3 in \cite{LRS}, one can find admissible metrics $g \in \Met(X, h)$ if $X$ is spin cobordant to the empty set, and the associated spin Dirac operator $D_B(Y, h)$ is invertible.
\end{enumerate}
\end{rem}

\section{The Moduli Space over Manifolds with Cylindrical End}\label{2}
In this section we analyze the structure of moduli space over $4$-manifolds with cylindrical end in the following case. Let $(M, g, \s)$ be a Riemannian spin$^c$ manifold with compatible boundary $(T^3, h, \mathfrak{t})$, where $h$ is a flat metric, and $\mathfrak{t}$ the spin$^c$ structure on $T^3$ with $c_1(\mathfrak{t})=0$. We also assume that 
\begin{equation}
H_*(M; \Z) \cong H_*(D^2 \times T^2; \Z).
\end{equation}
To avoid redundancy of notations, in this section we denote by $Z=M \cup [0, \infty) \times Y$ the end-cylindrical manifold instead of $M_{\infty}$ in this section. The same as before we have an induced metric $g$ and a spin$^c$ structure $\s$ on $Z$. 

When analyzing the structure of moduli spaces of $(Z, g, \s)$, two issues arise naturally. First the perturbations $\beta$ we are using are compactly supported, which means the Seiberg-Witten equation on the end is not perturbed. Thus the unperturbed moduli space of $(Y, \mathfrak{t})$ forms a manifold of critical points for the Chern-Simons-Dirac functional instead of isolated points. Moreover there is a point where the Morse-Bott condition does not hold, i.e. the Hessian $\Hess \mathcal{L}$ is degenerate in the normal direction. Thus the local picture around monopoles in $\M_{g, \beta}(Z, \s)$ asymptotic to this singular point needs understanding. Second the condition $b^+(Z)=0$ makes the appearance of the reducible locus $\M^{\Red}_{g, \beta}(Z, \s)$ inevitable no matter how we  choose the perturbation $\beta \in \mathcal{P}$. Thus when there is a sequence of irreducibles approaching a reducible monopole, we need to analyze the local picture around that point. We resolve the first issue via examining the asymptotic map, and the second via examining the Kuranishi obstruction map. 

\subsection{The Asymptotic Map} 

We begin by reviewing the asymptotic map constructed in \cite[Chapter 4]{MMR}, which relates the moduli spaces $\M_{g, \beta}(Z, \s)$ and $\M_h(T^3, \mathfrak{t})$. Originally the asymptotic map is derived in the Yang-Mills set-up (c.f. \cite[Chapter 4]{MMR}). In Seiberg-Witten case it's derived in Nicolaescu's book \cite[Chapter 4]{N}, and the case for cylindrical end modeled on $[0,\infty) \times T^3$ by Morgan-Mrowka-Szabó in \cite{MMS}. The keypoint for the existence of the asymptotic map relies on that the monopoles have finite energy, and the energy of a gradient flowline is proportional to the drop of the Chern-Simons-Dirac functional. 

\begin{dfn}
We say a reducible monopole $\mathfrak{b}=[B, 0] \in \M_h(T^3, \mathfrak{t})$ is a singular point if $\ker D_B \neq 0$, otherwise a smooth point. 
\end{dfn}

The singular set of $\M_h(T^3, \mathfrak{t})$ is well-known. 
\begin{lem}\label{ST0}
There is a unique flat spin$^c$ connection $\Theta$ of $\mathfrak{t}$, up to gauge transformation, such that $\ker D_{\Theta} \neq 0$. Moreover $\ker D_{\Theta}=\C^2$.
\end{lem}

The group of components of $\mathcal{G}(\mathfrak{t})$ is identified as 
\[\pi_0(\mathcal{G}(\mathfrak{t}), 1) = H^1(T^3; \Z).\]
 There is a subgroup $\mathcal{G}_M(\mathfrak{t}) \subset \mathcal{G}(\mathfrak{t})$ consisting of maps $u: T^3 \to S^1$ that extend to $\tilde{u}: M \to S^1$. Let $d\theta \in H^1(S^1; \Z)$ be the fundamental class. The primary obstruction to extending $u$ over $M$ is given by $\delta ( u^* d\theta) \in H^2(M, T^3; \Z)$, where $\delta: H^1(T^3;\Z) \to H^2(M, T^3;\Z)$ is the connecting map. Thus $u \in \mathcal{G}_M(\mathfrak{t})$ if and only if $\delta ( u^* d\theta)=0$. Thus the quotient $\mathcal{G}(\mathfrak{t}) / \mathcal{G}_M(\mathfrak{t})=H^1(T^3;\Z) / \im i^*$ with $i^*: H^1(M;\Z) \to H^1(T^3;\Z)$ the restriction map. Denote by $\mathcal{M}_{M, h}(T^3, \mathfrak{t})=\mathcal{C}_l(\mathfrak{t})/ \mathcal{G}_M(\mathfrak{t})$, which we identify as a copy of $\chi_M(Y)$ given as the $\Z$-cover $p:\chi_M(T^3) \to \chi(T^3)$. 
\begin{prop}\label{2.1}(\cite[Theorem 2.2]{MMS})
The moduli space $\M_{g, \beta}(Z, \s)$ is compact, moreover there exists a continuous map 
\[
\partial_+: \M_{g, \beta}(Z, \s) \longrightarrow \M_{M, h}(T^3, \mathfrak{t}).
\]
\end{prop}

Following the notation in \cite{MMS} we denote $\bar{\partial}_{+}=p \circ \partial_{+}$. We write $\theta=[\Theta, 0]$, and $U_{\theta}$ a small open neighborhood of $\theta$ in $\M_h(Y, \mathfrak{t})$. The structure theorem \cite[Theorem 2.4]{MMS} tells us the deleted irreducible moduli space $\M^*_{g, \beta}(Z, \s) \backslash \bar{\partial}^{-1}_{+}(U_{\theta})$ is an oriented smooth manifold of dimension 
\begin{equation}\label{Dimfor}
d(\s)={1 \over 4} (c_1(\mathfrak{s})^2[M, \partial M]-2 e(M) -3\sigma(M)) +{b_1(T^3)-b_0(T^3) \over 2} +{\rho_B(T^3) \over 2},
\end{equation}
where $\rho_B(T^3)$ is the invariant given by the odd signature operator twisted by the flat connection $B$ defined in \cite{APS2}. Since both $T^3$ admits orientation-reversing diffeomorphisms, we know that $\rho_B(T^3)=0$. 

Injectivity of $i^*: H^2(M;\Z) \to H^2(T^3;\Z)$ is injective implies that $\s$ is torsion. Then $c_1(\s)^2[M, \partial M]=0$. Note that $b^{\pm}(M)=0$, thus $\sigma(M)=0$. The dimension formula (\ref{Dimfor}) now reads as 
\begin{equation}
d(\s)={1 \over 2}( b_1(M)-b_2(M)+1)=1.
\end{equation}
F
Taking into account of perturbations on $Z$, the reducible locus $\M^{\Red}_{g, \beta}(Z, \s)$ of the moduli space is given by 
\[
\{
A \in \mathcal{A}_k(\s): F^+_{\A}=2d^+ \beta, \int_Z |F_{\A}-4d\beta|^2 < \infty
\}/ \G_{k+1}(\s)
\]
The following lemma identifies the reducible locus with the Picard torus consisting of gauge equivalence classes of flat connections on $(Z, \s)$. 

\begin{lem}\label{2.4}
The reducible locus $\M^{\Red}_{g, \beta}(Z, \s)$ is identified as the Picard torus $\mathbb{T}(\s)$ consisting of gauge classes of flat connections on $\det W^+$. 
\end{lem}

\begin{proof}
Let $A_0$ be a spin$^c$ connection satisfying $F^+_{\A_0}=2d^+\beta$. Such a connection always exists and can be chosen to be flat outside $\supp \beta$, thus of finite energy. Let $A$ be any other spin$^c$ connection satisfying 
\[
F^+_{\A}=2d^+\beta, \text{ and } \int_Z |F_{\A}-4d\beta|^2 < \infty.
\]
Write $a \otimes 1_{W^+} =A- A_0$ with $a \in L^2(T^*Z \otimes i\R)$. Noting that $\A-\A_0=2a$, we get $d^+a=0$. Since $b^+(M)=b^-(M)=0$, we conclude 
\[
0=\int_{Z} da \wedge da =\int_Z |d^+a|^2 - \int_Z |d^-a|^2 \Longrightarrow da=0.
\]
Recall that $u \cdot \A -\A_0=2(a- u^{-1}du)$ for $u \in \G_{k+1}(\s)$. Since the component group of $\G_{k+1}(\s)$ is $H^1(Z; \Z)$, and any element in the identify component has the form $u=e^{\xi}$, $\xi \in L^2_{k+1}(i\R)$, acting on $a$ as $u \cdot a=a-d\xi$, thus the equivalence classes of the difference $[a]$ are parametrized by 
\[
H^1(Z; i\R) / H^1(Z; i\Z) = \mathbb{T}(\s). 
\]
\end{proof}

It turns out the image of the reducible locus $\M^{\Red}(Z)$ under the map $\bar{\partial}_+$ misses the singular point $\theta \in \chi(T^3)$. 
\begin{lem}\label{2.5} 
Given $(Z, \s, g)$ as above, then for generic small perturbation $\beta$, 
\[
\theta \notin \bar{\partial}_+(\M^{\Red}_{g, \beta}(Z, \s)).
\]
\end{lem}

\begin{proof}
Since $\bar{\partial}_+$ is continuous, from Lemma \ref{2.4} it suffices to prove the result in the case when the perturbation $\omega=0$. Let $\tilde{\mathfrak{t}}_1$ be the product spin structure on $T^3$ inducing $\Theta$ as the spin connection. It's a well-known fact that the Rohlin invariant corresponding to this spin structure $\mu(T^3, \tilde{\mathfrak{t}}_1)=8 \mod 16$. One way to see it is to run Kaplan's argument for the characteristic sublink given by Borromean rings with coefficient 0 on all three components (c.f. \cite[Exercise 5.7.17. (c)]{GS}). \\
Since $H_1(M; \Z)$ has no $2$-torsion, the universal coefficient theorem tells us that the second Stiefel-Whitney class $w_2(M) \in H^2(M; \Z/2)$ is determined by its evaluation on $H_2(M; \Z)$. Wu's formula says that $w_2 (x) = x \cdot x =0 \mod 2$ due to the fact that $b^{\pm}(M)=0$. We conclude that $w_2(M)=0$, thus $M$ is spin. \\
The spin$^c$ structure $\s$ over $M$ comes from a spin structure twisted by some line bundle. Fix a flat spin connection $A_0$ on $M$ for some spin structure $\tilde{\s}$ with $\tilde{\s}|_{T^3}=\tilde{\mathfrak{t}}_0$. Suppose there is a flat spin$^c$-connection $A_{\theta}$ on $M$ such that $A_{\theta}|_{T^3}=\Theta$. Write $B_0=A_0|_{T^3}$. Regarding both $\Theta$ and $B_0$ as spin$^c$-connections on $\mathfrak{t}$, $[B_0]-[\Theta] \in H^1(M; i\R) / H^1(M; i\Z)$ is $2$-torsion. Since the restriction map $i^*: H^1(M; \Z/2) \to H^1(T^3; \Z/2)$ is injective, we get that $[A_0]-[A_{\theta}]$ is $2$-torsion as well. Since $A_0$ is a $\spin$ connection, we conclude $A_{\theta}$ is a spin connection for some spin structure on $X$. Thus $(T^3, \tilde{\mathfrak{t}}_1)$ is the spin boundary of $(M, \tilde{\s}')$. However the signature $\sigma(M)=0$, which means $\mu(T^3, \tilde{\mathfrak{t}}_1) =0 \mod 16$. We get a contradiction.
\end{proof}

In general one can only expect continuity of the asymptotic map $\bar{\partial}_+$. However in our case, one can get the smoothness and transversality of $\bar{\partial}_+$ away from the singular points in $\chi(T^3)$ using the technique of center manifolds developed in \cite[Chapter 5]{MMR}. 

Recall that the deformation complex of the 3-dimensional Seiberg-Witten equations at a reducible monopole $\mathfrak{b}:=(B, 0) \in \mathcal{C}_l(T^3)$ is a $U(1)$-equivariant complex
\begin{equation}\label{3cx}
L^2_{l+1}(i\R) \xrightarrow{\delta_1} L^2_l(T^*T^3 \otimes i \R) \oplus L^2_l(S) \xrightarrow{\delta_2} L_{l-1}^2(T^*T^3 \otimes i \R) \oplus L_{l-1}^2(S),
\end{equation}
where 
\[
\delta_1(\xi)=(-d\xi, 0), \delta_2(b, \psi)=(*db, D_B \psi).
\]
Their formal $L^2$-adjoints are given by 
\[
\delta_1^*(b, \psi)=-d^*b, \delta_2^*(b, \psi)=(*db, D_B\psi).
\]
We write $\mathcal{K}_{l, \mathfrak{b}}=\ker \delta^*_1$, and $\mathcal{S}_{l, \mathfrak{b}}=\mathfrak{b} + \mathcal{K}_{l, \mathfrak{b}}$ for the slice at $\mathfrak{b}$. Any other monopole $(B', 0)$ can be gauge transformed into this slice by $e^{\xi}$ solving the equation
\begin{equation}
d^*d\xi=d^*(B' -B).
\end{equation}
Due to Hodge decomposition, such a solution always exists. The gradient of the Chern-Simons-Dirac functional restricted to this slice is 
\begin{equation}
\grad \mathcal{L}|_{\mathcal{S}_{l, \mathfrak{b}}}(b, \psi) = (*db+\Pi_{\mathcal{K}} \rho^{-1}(\psi \psi^*)_0, D_B \psi + \rho(b) \psi),
\end{equation}
where $\Pi_{\mathcal{K}}: L^2_l(T^*T^3\otimes i \R) \to \mathcal{N}_{l,\mathfrak{b}}$ is the $L^2$-orthogonal projection. We further compute its Hessian at $\mathfrak{b}'= \mathfrak{b} +(b, \psi)$ to be 
\begin{equation}\label{eq2.4}
\Hess \mathcal{L}|_{\mathcal{S}_{l, \mathfrak{b}'}} (b_1, \psi_1) = (*db_1+\Pi_{\mathcal{K}} \rho^{-1}(\psi \psi_1^*+ \psi_1\psi^*)_0, D_{B+b} \psi_1+ \rho(b_1) \psi). 
\end{equation}
Note that when $\mathfrak{b}'=(B' ,0)$ is a reducible monopole, the second component of $\Hess \mathcal{L}|_{\mathcal{S}_{l, \mathfrak{b}'}}$ is just $D_{B'}$, which has trivial kernel when $B' \neq \Theta$, the singular point. Since the restricted Hessian is self-adjoint, we conclude that it's invertible at reducible monopoles. We denote the smallest positive eigenvalue of $\Hess \mathcal{L}|_{\mathcal{S}_{l, \mathfrak{b}}}$ by $\mu_{\mathfrak{b}}$.  \\

Let's write $H^1_{\mathfrak{b}}$ for the first homology of the deformation complex (\ref{3cx}), and $H^{\perp}_{\mathfrak{b}}$ for the $L^2$-orthogonal complement of $H^1_{\mathfrak{b}}$ in $\mathcal{K}_{l, \mathfrak{b}}$. Let $\tilde{U}_{\mathfrak{b}} \subset \mathcal{K}_{l, \mathfrak{b}}$ be a $\Stab(\mathfrak{b})$-invariant neighborhood of $0$, $U_{\mathfrak{b}} = \tilde{U}_{\mathfrak{b}} \cap H^1_{\mathfrak{b}}$, and $V_{\mathfrak{b}}=\mathfrak{b}+\tilde{U}_{\mathfrak{b}} \subset \mathcal{S}_{l, \mathfrak{b}}$. In short a center manifold $\mathcal{C}_{\mathfrak{b}}$ for the pair $(U_{\mathfrak{b}}, \grad \mathcal{L}|_{V_{\mathfrak{b}}})$ is the graph of a smooth map from $U_{\mathfrak{b}}$ to $H^{\perp}_{\mathfrak{b}}$, which is preserved by the flow of $\grad \mathcal{L}|_{V_{\mathfrak{b}}}$, and contains all critical points of $\grad \mathcal{L}|_{V_{\mathfrak{b}}}$. For more details one may consult \cite[Definition 5.1.2]{MMR}. 

The advantage of the center manifolds is that any gradient flowline of the Chern-Simons-Dirac functional is exponentially close to a flowline on the center manifold which has finite dimension. The Yang-Mills case was proved in \cite[Theorem 5.2.2]{MMR}. For the Seiberg-Witten case, compare with \cite[Theorem 4.2.33]{N}. Due to this reason we will study the gradient flowline on the center manifold. Let's write $W^s_{\mathfrak{b}} \subset \mathcal{C}_{\mathfrak{b}}$ for the stable set in the center manifold, i.e. those points with limit existing in $\mathcal{C}_{\mathfrak{b}}$, and $W^{s, 0}_{\mathfrak{b}}$ for the stable set consisting of points that converge to $\mathfrak{b}$. For a smooth point $[\mathfrak{b}]$, one can find a trivial center manifold. 

\begin{lem}
Let $[\mathfrak{b}] \in \chi(T^3)$ be a smooth point with a neighborhood $V_{\mathfrak{b}}=\mathfrak{b}+\tilde{U}_{\mathfrak{b}}$ away from singular points. Then $U_{\mathfrak{b}}$ is a smooth center manifold for the pair $(U_{\mathfrak{b}}, \grad \mathcal{L}|_{V_{\mathfrak{b}}})$, i.e. the center manifold given by the graph of the zero map. 
\end{lem}

\begin{proof}
We regard $U_{\mathfrak{b}}$ as the graph of the zero map. For any $\mathfrak{b}'=(b', 0) \in U_{\mathfrak{b}}$, $\grad (\mathcal{L}|_{V_{\mathfrak{b}}})|_{{\mathfrak{b}'}}= (*db', 0) \in U_{\mathfrak{b}}$. Away from $\theta$ all critical points of $\grad \mathcal{L}$ has the form $(b, 0) \in U_{\mathfrak{b}}$. This verifies that $U_{\mathfrak{b}}$ is a center manifold. 
\end{proof}

Note that every point in $V_{\mathfrak{b}}$ is a critical point of the $\grad \mathcal{L}|_{V_{\mathfrak{b}}}$, we conclude that $W^s_{\mathfrak{b}}=V_{\mathfrak{b}}$, and $W^{s, 0}_{\mathfrak{b}}=\mathfrak{b}$. 

One can also construct a center manifold for the singular point $\theta$. 

\begin{lem}(\cite[Lemma 2.5]{MMS})
$H^1_{\theta} \cong \mathcal{H}^1(T^3; i\R) \oplus \C^2$ is a smooth center manifold around the singular point $\theta$ for the pair $(\mathcal{K}_{l, \theta}, \grad \mathcal{L}|_{\mathcal{S}_{l, \theta}})$, where $\mathcal{H}^1(T^3; i\R)$ consists of imaginary-valued harmonic $1$-forms on $T^3$, and $\C^2=\ker D_{\Theta}$. 
\end{lem} 

In this case the stable sets are explicitly written out under some identification in \cite[Section 2]{MMS}. We rephrase their identification below. For $(b, \psi) \in H^1_{\theta}$ the gradient of the Chern-Simons-Dirac functional is 
\begin{equation}
\grad \mathcal{L}|_{\mathcal{S}_{l, \theta}}(b, \psi) = (\rho^{-1}(\psi \psi^*)_0,  \rho(b) \psi).
\end{equation}
We identify the center manifold $H^1_{\theta}$ with $\Ima \mathbb{H} \oplus \mathbb{H}$ as follows. Note that $H^1_{\theta}= \mathcal{H}^1(Y; i\R) \oplus \ker D_{\theta}$, where $\ker D_{\theta}$ consists of constant spinors, which is identified with $\C^2$. We identify $\C^2 \cong \mathbb{H}$ via $(\alpha, \beta) \mapsto \alpha+ j \beta$. Let ${e_1, e_2, e_3}$ be an oriented orthonormal framing of $T^*T^3$ such that $\{e_1 \otimes i, e_2 \otimes i, e_3 \otimes i\}$ spans the harmonic $1$-forms $\mathcal{H}^1(T^3; i\R)$. Identify $\mathcal{H}^1(T^3; i\R) \cong \Ima \mathbb{H}$ via 
\[
e_1 \otimes i \mapsto i, e_2 \otimes i \mapsto j, -e_3 \otimes i \mapsto k.
\]
Let $\psi=(\alpha, \beta)$. We have 
\[
(\psi \psi^*)_0= 
\begin{pmatrix}
{1 \over 2}(|\alpha|^2 - |\beta|^2) & \alpha \overline{\beta} \\
\overline{\alpha} \beta & {1 \over 2}(|\beta|^2 - |\alpha|^2)
\end{pmatrix}.
\]
Note that the Clifford multiplication on $T^*T^3$ is given by Pauli matrices:
\[
\rho(e_1)=
\begin{pmatrix}
i & 0 \\
0 & -i
\end{pmatrix},
\rho(e_2)=
\begin{pmatrix}
0 & -1 \\
1 & 0
\end{pmatrix},
\rho(e_3)=
\begin{pmatrix}
0 & i \\
i & 0
\end{pmatrix}.
\]
Under this identification
\[
\rho(b)\psi \mapsto b\psi i \in \mathbb{H}, \; \; \rho^{-1}(\psi \psi^*)_0 \mapsto -{1 \over 2} \psi i \overline{\psi} \in \Ima \mathbb{H}, 
\]
where the latter part is given by quaternion multiplication. Thus the downward gradient flow equation in the center manifold $H^1_{\theta}$ is 
\begin{equation}\label{dgr}
\begin{split}
\dot{b}(t) & = {1 \over 2}\psi(t) i \overline{\psi}(t), \\
\dot{\psi}(t) & = -b(t)\psi(t)i .
\end{split}
\end{equation}

Denote by $\langle b_1, b_2 \rangle$ the real inner product, $b_1 \times b_2 =\Ima (b_1 b_2)$ the real cross product for $b_1, b_2 \in \Ima \mathbb{H}$. Lemma 2.7 in \cite{MMS} asserts that inside $H^1_{\theta}$ 
\begin{equation}
W^{s}_{\theta}=\{ (b, \psi) : 2|b|^2 \geq |\psi|^2, b \times \psi i\overline{\psi}=0, \langle b, \psi i \overline{\psi} \rangle \leq 0\},
\end{equation}
and 
\begin{equation}
W^{s, 0}_{\theta}=\{ (b, \psi) : 2|b|^2 = |\psi|^2, b \times \psi i\overline{\psi}=0, \langle b, \psi i \overline{\psi} \rangle \leq 0\}.
\end{equation}
Thus $W^{s}_{\theta} \backslash (0,0)$ is a smooth $6$-dimensional manifold with boundary $W^{s, 0}_{\theta} \backslash (0,0)$. The singular point $(0,0)$ has codimension $6$ in $W^{s}_{\theta}$. 

In order to get the transversality of the asymptotic map $\bar{\partial}_+$, we consider the based moduli spaces. Now fix a point $x_0=(0, y_0) \in \{0 \} \times T^3 \subset Z$, there are two ways to think of the based moduli space. One way is to consider the based gauge group $\tilde{\G}(\s)=\{ u \in \G(\s) : u(x_0)= \id\}$ and then define the based moduli space consisting of monopoles up to based gauge transformation. Another view point is to enlarge the configuration space to $\mathcal{C}_k(\s) \times \mathbb{S}(W^+|_{x_0})$, where $\mathbb{S}(W^+|_{x_0})$ is the unit circle in the fiber of $W^+$ over $x_0$ of the determinant line bundle. Then the based moduli space consists of framed monopoles $(A, \Phi, v)$ up to full gauge transformation. A similar construction can be carried out over the $3$-manifold $T^3$. 

Denote by $\tilde{\M}^*_{g, \beta}(Z, \s)$ the based irreducible moduli space of $(Z, \s, g, \beta)$, which is a $U(1)$-bundle over $\M^*_{g, \beta}(Z, \s)$. Recall that $U_{\mathfrak{b}} \subset H^1_{\mathfrak{b}}$ is $\Stab(\mathfrak{b})$-invariant, thus can be identified as a neighborhood of $[\mathfrak{b}]$ in $\chi(T^3)$. Denote by $\tilde{\M}^*(Z, U_{\mathfrak{b}})$ the part of the bundle over $\bar{\partial}_+^{-1}(U_{\mathfrak{b}})$. Now we prove the transversality for the asymptotic map.

\begin{prop} \label{2.12}
Denote $\M^*(Z, \theta^c)=\M^*_{g, \beta}(Z, \s) \backslash \bar{\partial}_+^{-1}(\theta)$. Then 
\item[\upshape (i)] the asymptotic map $\bar{\partial}_+: \M^*(Z, \theta^c) \to \M_h(T^3, \mathfrak{t})$ is smooth;
\item[\upshape (ii)] given any finite subcomplex $C \subset \M_h(T^3, \mathfrak{t})$, there is a Baire set of second category of small perturbations $\omega$ so that $\bar{\partial}_+|_{\M^*(Z, \theta^c) }$ is transverse to $C$. 
\end{prop}

\begin{proof}
As pointed out in the paragraph above Theorem 2.8 in \cite{MMS}, the Seiberg-Witten analogue of \cite[Theorem 9.0.1]{MMR} assures that there is a smooth map 
\[
\tilde{\partial}_+ : \tilde{\M}^*(Z, U_{\mathfrak{b}}) \longrightarrow H^1_{\mathfrak{b}} \times_{U(1)} \mathbb{S}(S|_{y_0})
\]
transverse to any finite subcomplex for generic perturbation $\beta$. The map $\tilde{\partial}_+$ has the property that for any $[\Gamma, v] \in  \tilde{\M}^*(Z, U_{\mathfrak{b}})$, the corresponding path $\gamma(t)$ on $[T_0, \infty)$ is exponentially close to the flowline in $H^1_{\mathfrak{b}}$ starting at $\tilde{\partial}_+(\Gamma)$. In this way we can regard $\tilde{\partial}_+$ as a refinement of $\bar{\partial}_+$. When $[\mathfrak{b}] \neq \theta$, all gradient flowlines in $U_{\mathfrak{b}}$ are constant. Note that the $U(1)$ action on $H^1_{\mathfrak{b}}$ is trivial, thus we get the following commutative diagram:
\[
\begin{tikzcd}
\tilde{\M}^*(Z, U_{\mathfrak{b}}) \ar[r, "\tilde{\partial}_+"] \ar[d, "p_1"] &  H^1_{\mathfrak{b}} \times_{U(1)} \mathbb{S}(W^+|_{x_0}) \ar[d, "p_2"] \\
\M^*(Z, U_{\mathfrak{b}}) \ar[r, "{\bar{\partial}_+}"] & H^1_{\mathfrak{b}}
\end{tikzcd}
\]
Since the map $p_1$ is a submersion, $p_2$ is a diffeomorphism, smoothness and transversality of the map $\tilde{\partial}_+$ carries over to $\bar{\partial}_+$. 
\end{proof}

The transversality result gives us a little bit more understanding of the structure for the moduli space.

\begin{prop}\label{4.8}
After fixing a homology orientation, for generic perturbation $\beta$ the irreducible moduli space  $\M^*_{g, \beta}(Z, \s)$ is a (possibly noncompact) $1$-manifold with boundary of finite components. Moreover $\partial \M^*_{g, \beta}(Z, \s) \subset \bar{\partial}^{-1}_{\infty}(\theta)$. 
\end{prop}

\begin{proof}
The formal dimension of the based space $ \tilde{\M}^*_{g, \beta}(Z, \s)$ is $2$. Since $(0,0) \in H^1_{\theta}$ has codimension 7, transversality implies that $(0, 0) \notin \im \tilde{\partial}_{+}$. Note that $W^{s}_{\theta} \backslash (0,0)$ is a manifold with boundary $W^{s, 0}_{\theta} \backslash (0,0)$ of codimension $1$, and $\tilde{\M}^*_{g,\beta}(Z, \s)$ $\to \M^*_{g, \beta}(Z, \s)$ is a principal $U(1)$-bundle, then the conclusion follows.
\end{proof}

\subsection{Kuranishi Picture at Reducible Monopoles}
In this section we use the Kuranishi obstruction map to study the local picture at a reducible monopole over the cylindrical end $4$-manifold $(Z, \s, g)$ where a sequence of irreducibles converges to this reducible monopole. 

\subsubsection{Reformulation of the Moduli Space}
Since $Z$ is noncompact, in order to get a local deformation theory of the moduli space, we make use of the weighted Sobolev space. Let's extend the coordinate map $t: [0, \infty) \times T^3 \to [0, \infty)$ arbitrarily to a smooth map $\tau: Z \to [-1, \infty)$. Given any positive number $\delta >0$,  the weighted Sobolev space $L^2_{k,\delta}(W^+)$ is the completion of the space $C^{\infty}_c(W^+)$ via the norm 
\[
\|\Phi\|_{L^2_{k, \delta}} = \|e^{\tau \delta} \Phi\|_{L^2_k}. 
\]
In particular multiplication by $e^{\tau\delta}$ gives us an isometry 
\[
e^{\tau \delta} \cdot : L^2_{k, \delta}(W^+)  \longrightarrow L^2_k(W^+). 
\]
Differentiation on $C^{\infty}_c(W^+)$ is given by the connection $A_{\mathfrak{b}}$. Similarly one can define the weighted Sobolev spaces $L^2_k(\Lambda^* T^*Z \otimes i\R)$.

Fix $\epsilon_0 >0$. Let $U_{\theta}$ be an open neighborhood of $\theta \in\chi(T^3)$ so that 
\[
\M^{\Red}_{g, \beta}(Z, \s) \cap \bar{\partial}_+^{-1}(U_{\theta}) =\emptyset \text{ for all } \|\beta\|_{L^2_{k+1}} < \epsilon_0.
\]
Let $\mu_0=\min \{\mu_{\mathfrak{b}} : \mathfrak{b} \in \chi(T^3) \backslash U_{\theta} \}$. Choose $0< \delta < {\mu_0 \over 2}$. Denote by $U^c_{\theta}=\chi(T^3) \backslash U_{\theta}$ the complement of $U_{\theta}$. From \cite[Theorem 4.2.33]{N} we know that any finite energy monopole $\Gamma$ over $Z$ with $\bar{\partial}_{\infty}([\Gamma]) \in U^c_{\theta}$ has exponential decay over the end with exponent at least $\delta$:
\begin{equation}\label{expd}
|\Psi(t)| \leq C e^{-{ \delta} (t-T_0)},
\end{equation}
where $\Gamma=(A , \Phi), \Psi(t)=\Phi|_{T^3_t}$, $T_0$ is some fixed positive number. To each monopole $\mathfrak{b} \in \mathcal{C}_l(T^3)$, we associate the constant flowline $\gamma_{\mathfrak{b}}$ and the 4-dimensional monopole $\Gamma_{\mathfrak{b}}$ over $[T_0, \infty) \times T^3$. Extend it arbitrarily over $Z$ to an element $\Gamma_{\mathfrak{b}}=(A_{\mathfrak{b}}, 0) \in \mathcal{C}_k(Z)$. We define the $\mathfrak{b}$-asymptotic configuration space to be 
\begin{equation}
\mathcal{C}_{k, \delta}(Z, \mathfrak{b}):=\{ \Gamma \in \mathcal{C}_k(Z) : \Gamma- \Gamma_{\mathfrak{b}} \in L^2_{k, \delta}(T^*Z \otimes i\R \oplus W^+) \}.
\end{equation}
Denote by $\mathcal{G}_{k+1, \delta}(Z, \mathfrak{b}) \subset \mathcal{G}_{k+1}(Z)$ the subgroup consisting of elements preserving $\mathcal{C}_{k, \delta}(Z, \mathfrak{b})$. 
\begin{lem}\label{2.7}
$\mathcal{G}_{k+1, \delta}(Z, \mathfrak{b})$ is independent of $\mathfrak{b}$. Moreover 
\begin{align*}
\mathcal{G}_{k+1,\delta}(Z) = \{&  u \in \mathcal{G}_{k+1}(Z) : u|_{[T, \infty) \times T^3} =u_0 \cdot e^{\xi},  \text{ where } u_0 \in S^1, \\ 
&\xi \in L^2_{k+1, \delta}([T, \infty) \times T^3, i\R) \text{ for some } T >0\}.
\end{align*}
\end{lem}

\begin{proof}
It's clear that the second statement implies the first one. Now suppose $u$ has the form as above. Let $\Gamma-\Gamma_{\mathfrak{b}} = (a , \phi) \in L^2_{k,\delta}(Z)$. Over the end $[T, \infty) \times T^3,$ we have $e^{\xi} \cdot (a, \phi) = (a - d\xi, e^{\xi} \phi).$ The Sobolev multiplication theorem implies that $\|e^{\delta \tau} \cdot e^{\xi} \phi\|_{L^2_k} < \infty$. By assumption $\|e^{\delta \tau}(a- d\xi) \|_{L^2_k} \leq \|a\|_{L^2_{k, \delta}} + \|d\xi \|_{L^2_{k, \delta}} < \infty$. Thus $u \in \mathcal{G}_{k+1, \delta}(Z)$. \\
The argument for the converse direction was inspired by that of \cite[Lemma 13.3.1]{KM1}. Suppose $u$ preserves $\mathcal{C}_{k, \delta}(Z, \mathfrak{b})$. Let $u_n=u|_{[n-1, n]\times T^3}$  which is regarded as a gauge transformation on $[0,1] \times T^3$. The component group of the gauge transformations over $[0,1] \times T^3$ is $\pi_0( \Map([0,1] \times T^3, S^1))$, which is identified as $H^1([0,1] \times T^3; \Z)$ via the assignment $[u] \mapsto (1/2\pi i) u^{-1}du$. Thus the component is determined by 
\[
\int_{[0, 1] \times T^3} \alpha_i \wedge {1 \over {2\pi i}} u^{-1}du,
\]
for a basis $\{\alpha_i\}$ of $H^3([0,1] \times T^3, \{0, 1\} \times T^3; \Z)$. Note that $\|u_n^{-1} du_n\|_{L^2_k} \to 0$ from the assumption. So when $n \geq T$ for some $T$ large enough , $u_n=e^{\xi_n}$ lies on the identity component, where $\xi_n \in L^2_{k+1}([0,1] \times T^3, i \R)$. Now we decompose $\xi_n=\xi_n^0+\xi_n^{\perp}$, where $\xi_n^0$ is constant, and $\int_{[0,1] \times T^3} \xi_n^{\perp} d \vol = 0$. Continuity of $\xi$ implies that $\xi_n^0$ are equal for all $n$. Thus we get a global decomposition $\xi = \xi^0+ \xi^{\perp}$ with $\xi^0$ constant and $\int_{[T, \infty) \times T^3} \xi^{\perp} d \vol =0$. Then we have 
\[
\| d\xi \|_{L^2_{k, \delta}} \leq \| A - A_{\mathfrak{b}} \|_{L^2_{k , \delta}} +\|A - d\xi -A_{\mathfrak{b}}\|_{L^2_{k, \delta}} < \infty.
\]
Moreover the Poincaré inequality implies 
\[
\| \xi_n^{\perp} \|^2_{L^2_{\delta}([0,1] \times T^3)} \leq C \| d\xi^{\perp}_n \|^2_{L^2_{\delta}([0,1] \times T^3)}.
\]
Thus 
\[
\| \xi^{\perp}\|^2_{L^2_{\delta}([T, \infty) \times T^3)}= \sum_n \|\xi^{\perp}_n \|^2_{L^2_{\delta}([0,1] \times T^3)} \leq C \| d\xi^{\perp} \|^2_{L^2_{\delta}( [T, \infty) \times T^3)} < \infty.\]
Thus $\xi^{\perp} \in L^2_{k+1, \delta}([T, \infty) \times T^3, i\R)$. Let $u_0=e^{\xi^0}$. This finishes the proof.
\end{proof}

We denote the quotient by $\mathcal{B}_{k, \delta}(Z, \mathfrak{b}) = \mathcal{C}_{k, \delta}(Z, \mathfrak{b}) / \mathcal{G}_{k+1, \delta}(Z).$ The $\mathfrak{b}$-asymptotic moduli space is defined to be 
\[
\mathcal{M}(Z, \mathfrak{b})=:\{ (A, \Phi) \in \mathcal{C}_{k, \delta}(Z, \mathfrak{b}) : \mathfrak{F}_{\beta}(A, \Phi)=0 \}/ \mathcal{G}_{k+1, \delta}(Z). 
\]

The following lemma identifies this moduli space with the preimage of $[\mathfrak{b}]$ under the asymptotic map $\bar{\partial}_+$. We write $[\tilde{\mathfrak{b}}] \in \chi_M(T^3)$ to represent the class. 
\begin{lem}\label{2.8}
Suppose $[\mathfrak{b}] \neq \theta \in \chi(T^3)$. Then there is a homeomorphism
\[ \mathcal{M}(Z, \mathfrak{b}) \longrightarrow \partial_+^{-1}([\tilde{\mathfrak{b}}]) \subset \M_{g, \beta}(Z, \s).\]
\end{lem}

\begin{proof}
Let $\tilde{\Gamma} \in \mathcal{C}_{k ,\delta}(Z, \mathfrak{b})$ be a lift of $[\Gamma] \in \M(Z, \mathfrak{b})$. Then we assign the class $[\tilde{\Gamma}] \in \partial_+^{-1}([\tilde{\mathfrak{b}}])$ to $[\Gamma]$. Since the gauge group $\mathcal{G}_{k+1, \delta}(Z) \subset \mathcal{G}_{k+1}(Z)$, it follows that the assignment is well-defined. 

To see injectivity, suppose we have $\tilde{\Gamma}_1, \tilde{\Gamma}_2$ of lifts $[\Gamma_1]$ and $[\Gamma_2]$ respectively satisfying $[\tilde{\Gamma}_1] =[\tilde{\Gamma}_2]$ in $\partial_+^{-1}([\tilde{\mathfrak{b}}])$. Then there is a gauge transformation $u \in \mathcal{G}_{k+1}(Z)$ such that $u \cdot \tilde{\Gamma}_1 = \tilde{\Gamma}_2$. By Lemma \ref{2.7} we know $u \in \mathcal{G}_{k+1,\delta}(Z)$. Thus $[\Gamma_1]=[\Gamma_2] \in \mathcal{M}(Z, \mathfrak{b})$. 

For surjectivity, let's take $\Gamma \in \M_{g, \beta}(Z, \s)$ such that $\partial_+(\Gamma) = \mathfrak{b}'$ and $[\tilde{\mathfrak{b}}']=[\tilde{\mathfrak{b}}]$. Thus $\exists u \in \mathcal{G}_{l+1}(T^3)$ that extends to $\tilde{u} \in \mathcal{G}_{k+1}(Z)$ with $u=\tilde{u}|_{[T_0, \infty) \times T^3}$ satisfying $u \cdot \mathfrak{b}'= \mathfrak{b}$. It follows from (\ref{expd}) that $\Gamma - \Gamma_{\mathfrak{b}'} \in L^2_{k, \delta}(Z)$. Then $\tilde{u} \cdot \Gamma - \Gamma_{\mathfrak{b}} \in L^2_{k, \delta}(Z)$. Thus $[\tilde{u} \cdot \Gamma] \in \M(Z, \mathfrak{b})$, which is mapped to $[\Gamma] \in \M_{g, \beta}(Z, \s)$. 

The continuity of the map and its converse is clear from the nature of these spaces. 
\end{proof}

From Lemma \ref{2.8} we conclude there is a homeomorphism
\begin{equation}\label{imu}
\M(Z, U_{\theta}^c):= \bigcup_{ [\tilde{\mathfrak{b}}] \in \chi_M(T^3) \backslash p^{-1}(U_{\theta})} \M(Z, [\tilde{\mathfrak{b}}]) \cong \M_{g, \beta}(Z, \s) \backslash \bar{\partial}_+^{-1}(U_{\theta}). 
\end{equation}
It's clear that this homeomorphism also identifies those two sets as smooth stratified spaces (see for example \cite[Chapter 8]{MMR}).

\subsubsection{Deformation Complex}
Having identified the moduli spaces, we are now ready to study the deformation theory at a reducible monopole $[\Gamma]=[A, 0] \in \M^{\Red}_{g, \beta}(Z, \s)$. The deformation complex at $(A, 0)$ is given by 
\begin{equation}\label{4cx}
\hat{L}^2_{k+1, \delta}(i \R) \xrightarrow{\delta_{1, \Gamma}} \hat{L}^2_{k, \delta}(T^*Z \otimes i \R) \oplus L^2_{k, \delta}(W^+) \xrightarrow{\delta_{2, \Gamma}} L^2_{k-1, \delta} (\Lambda^+  \otimes i\R \oplus W^-),
\end{equation}
where 
\[
\delta_{1, \Gamma}(\xi) =(-d\xi, 0), \; \delta_{2, \Gamma}(a, \phi)=(d^+a, D^+_A \phi).
\]
We write $\hat{L}^2_{k+1, \delta}(i \R)$ for the Lie algebra of $\mathcal{G}_{k+1, \delta}(Z)$ identified as 
\[
 \{\xi \in L^2_{k+1, loc}(Z, i\R) : d\xi \in L^2_{k, \delta}(Z, i\R) \}, 
\]
$\hat{L}^2_{k, \delta}(T^*Z \otimes i\R)$ for the tangent space of $\mathcal{A}_{k}(Z, U_{\theta}^c)$  at $A$ identified as 
\[
 \{a + \varphi b : a \in L^2_{k, \delta}(T^*Z \otimes i\R), b \in i^* \mathcal{H}^1(M; i\R) \},
\]
where $\varphi: Z \to [0,1]$ is a cut-off function such that $\varphi|_{M}=0, \varphi|_{[1, \infty) \times T^3}=1$, and  $i^*:\mathcal{H}^1(M; i\R) \to \mathcal{H}^1(T^3; i\R)$ is the restriction map on harmonic $1$-forms. Note that the stabilizer of $\Gamma$ is $U(1)$ consisting of constant gauge transformations, thus the complex (\ref{4cx}) is $U(1)$-equivariant. 
Sitting inside of the complex (\ref{4cx}) is a subcomplex 
\begin{equation}\label{sub4cx}
L^2_{k+1, \delta}(i \R) \xrightarrow{\delta_{1, \Gamma}} L^2_{k, \delta}(T^*Z \otimes i \R \oplus W^+) \xrightarrow{\delta_{2, \Gamma}} L^2_{k-1, \delta} (\Lambda^+ T^*Z \otimes i\R \oplus W^-),
\end{equation}
whose quotient is a finite dimensional complex 
\begin{equation}
i\R \to i^* \mathcal{H}^1(M; i\R)  \to 0,
\end{equation}
where $\partial_+(A, 0) = \mathfrak{b}=(b,0)$. The arguments in \cite{APS1} implies that the complex (\ref{sub4cx}) is Fredholm for $\delta$ small enough. Thus the total deformation complex (\ref{4cx}) is Fredholm as well. Since we are dealing with weighted spaces, the formal adjoints are slightly different: 
\[
\delta^*_{1, \Gamma}(a, \phi) = -e^{-\tau \delta} d^* e^{\tau \delta} a, \; \delta^*_{2, \Gamma}(\omega, \phi)=e^{-\tau \delta}(d^* e^{\tau \delta} \omega, D_A^- e^{\tau \delta} \phi). 
\]
The homology of the subcomplex (\ref{sub4cx}) was computed by Taubes in \cite[Proposition 5.1]{T1}, though his set-up is slightly different from ours. A similar computation for the homology of (\ref{4cx}) in the case of $3$-dimensional cylindrical-ended manifolds was carried out by Lim in \cite[Proposition 6.1]{Lim}.
\begin{prop}\label{2.9}
We denote the complex (\ref{4cx}) by $E_{\delta}(\Gamma)$. Its homology is given by: 
\begin{enumerate}
\item[\upshape (i)] $H^0(E_{\delta}(\Gamma)) = H^0(Z ; i\R)$,
\item[\upshape (ii)] $H^1(E_{\delta}(\Gamma))= H^1(Z; i\R) \oplus \ker D_A^+$, 
\item[\upshape (iii)] $H^2(E_{\delta}(\Gamma))= \hat{H}^+_c(Z; i\R) \oplus \ker D^{+, *}_{A, \delta} = \ker D^-_{A, \delta}$, where  $\hat{H}^+_c(Z; i\R)$ is the image of $H^+_c(Z; i\R)$ in $H^2(Z; i\R)$ under the inclusion map, and 
\[
D^{+, *}_{A, \delta} =e^{-\tau \delta} D_A^- e^{\tau \delta}
\]
is the $L^2_{\delta}$-adjoint of $D^+_{A, \delta}$.  
\end{enumerate}
\end{prop}

\begin{proof}
$(i)$ As in the proof of Lemma \ref{2.7}, for any $\xi \in \hat{L}^2_{k+1, \delta}(i \R)$ we have a decomposition $\xi =\xi^0 +\xi^{\perp}$ with $\xi^0$ constant. $d \xi =0$ implies that $d \xi^{\perp}=0$. Since $\xi^{\perp} \in L^2_{k+1, \delta}(i\R)$, we get $\xi^{\perp}=0$. Thus $\xi=\xi^0$, and $H^0(E_{\delta}(\Gamma)) = i\R.$

$(ii)$ The homology of $E_{\delta}(\Gamma)$ splits into the connection part $H^1_1(E_{\delta}(\Gamma))$ and the spinor part $H^1_2(E_{\delta}(\Gamma))$. Suppose $\delta_{2, \gamma}(a ,\phi)=(d^+a, D^+_A \phi) =0$. It's clear that $H^1_2(E_{\delta}(\Gamma))$ is identified as $\ker D_A^+$. To identify the connection part, note that $A$ is a flat connection, the argument in Lemma \ref{2.4} implies that $A+a$ is flat, thus $da=0$. Define the map $r: H^1_1(E_{\delta}(\Gamma)) \to H^1_{dR}(Z; i\R)$ sending $[a] \mapsto [a]$. We claim $r$ is an isomorphism. To see injectivity, suppose $a=d\xi$ for $\xi \in \Omega^0(Z; i\R)$. we write $a=a^0+\beta b$. Since $b$ is harmonic, it cannot be exact, thus $b=0$. Note that $a^0=d\xi \in L^2_{k,\delta}$. The argument in the proof of Lemma \ref{2.7} implies that $\xi \in L^2_{k+1, \delta}$. Thus $[a]=0 \in H^1_1(E_{\delta}(\Gamma))$. To see surjectivity, from the following portion of an exact sequence: 
\[
H^1_c(Z ; i\R) \longrightarrow H^1(Z; i\R) \longrightarrow H^1(T^3; i\R)
\]
we derive that given $[a] \in H^1(Z; i\R)$ with a representative $a$, one can choose $b \in H^1(T^3; i\R)$ such that $a- \varphi b \in H^1_c(Z ; i\R)$. In particular $a- b \in L^2_{k, \delta}(T^*Z; i\R)$, which proves the surjectivity. Lastly the topology of $Z$ tells us that $H^1(Z; i\R) = i\R \oplus i\R$. 

$(iii)$ Again we get a decomposition of the homology $H^2(E_{\delta}(\Gamma)) = H^2_1(E_{\delta}(\Gamma)) \oplus H^2_2(E_{\delta}(\Gamma))$ corresponding to the connection part and spinor part respectively. Since the complex is Fredholm, it's clear that the spinor part is identified as $\ker D^{+, *}_{A, \delta} = \ker D^-_{A, \delta}$. Now we identify the connection part $H^2_1(E_{\delta}(\Gamma))=\ker d^*_{\delta} \cap \ker d$, where $d^*_{\delta} = e^{-\tau \delta} d^* e^{\tau \delta}$. For any $\omega \in L^2_{k-1, \delta} (\Lambda^+ T^*Z \otimes i\R)$, $d^*_{\delta} \omega=0$ implies that $d e^{\tau \delta} \omega=0$. Let $e^{\tau \delta} \omega = \eta \in L^2_{k-1} (\Lambda^+ T^*Z \otimes i\R)$. Then $d \eta=0, d^* \eta=0.$ Proposition 4.9 in \cite{APS1} gives us an injection 
\begin{equation}
\begin{split}
\iota: H^2_1(E_{\delta}(\Gamma)) & \to \hat{H}^+_c(Z; i\R) \\
[\omega] & \mapsto [e^{\tau \delta} \omega]=[\eta].
\end{split}
\end{equation}
 To see $\iota$ is also surjective, note that Proposition 4.9 in \cite{APS1} only identifies $\hat{H}^+_c(Z; i\R)$ with harmonic forms $\eta \in L^2(\Lambda^+T^*Z \otimes i\R)$. To see $\eta \in L^2_{k-1}$, we write $\eta|_M=\eta_0, \eta|_{[n-1, n] \times T^3}=\eta_n, \forall n \geq 1$. Ellipticity of $d^*\oplus d$ implies that 
\[
\| \nabla \eta_n \|^2_{L^2_m} \leq C (\|(d^* \oplus d) \eta_n \|^2_{L^2_{m-1}} + \|\eta_n\|^2_{L^2}).
\]
Since $\eta$ is harmonic, we get
\[
\| \nabla \eta \|^2_{L^2_m(Z)} \leq \sum_{ m \geq 0}  \|  \eta_n \|^2_{L^2_m} \leq C \|\eta\|^2_{L^2(Z)} < \infty,
\] 
which implies that $\eta \in L^2_{k-1}$. Let $\omega=e^{-\tau\delta} \eta$. Then $\iota[\omega]=[\eta]$ gives us the surjectivity. Finally the intersection form on $\hat{H}^2_c(Z; i\R)$ is trivial, thus $\hat{H}^+_c(Z; i\R)=0$. We conclude that $H^2(E_{\delta}(\Gamma))=\ker D^-_{A, \delta}$. 
\end{proof}

\subsubsection{Approximation of Obstruction Map} Recall that our choice of perturbation is $d^+ \beta \in \Omega^+_c(Z; i\R)$. Fix a flat spin$^c$ connection $A_0$ on $W^+$. The reducible locus 
\[
\M^{\Red}_{g, \beta}(Z, \s)=\{ [A, 0] : F^+_{\A} = 2 d^+ \beta\} 
\]
is identified as $A_0 + \beta + H^1(Z; i\R) / H^1(Z; i\Z) =A_0+\beta+\mathbb{T}(\s)$. Thus for each perturbation $\beta$, we get a $T^2$-family of Dirac operators 
\[
D^+_{\alpha, \beta}:=D^+_{A_0}+\rho_Z(\alpha)+\rho_Z(\beta), \alpha \in \mathbb{T}(\s).
\]
The part still missing in the picture of the moduli space $\M_{g, \beta}(Z, \s)$ locates at the reducibles $[A, 0]$ to which there is a sequence of irreducibles converge. A necessary condition for the occurrence of such case is when $\ker D^+_A \neq 0$. 
\begin{lem}\label{4.4}
Suppose there is a sequence of irreducible monopoles $[A_i ,\Phi_i]$ in $\M^*_{g, \beta}(Z, \s)$ converging to a reducible monopole $[A, 0]$. Then $\ker D^+_A \neq 0$.
\end{lem}

\begin{proof}
Since $\theta \notin \bar{\partial}_+(\M^{\Red}_{g, \beta}(Z, \s))$, we may assume $\theta \neq \bar{\partial}_{\infty}([A_i, \Phi_i]), \forall i$. Then the pointwise norm of the spinor has exponential decay in the $t$-direction. Thus we can normalize them by 
\[
\hat{\Phi}_i = {\Phi_i \over {\|\Phi_i\|_{L^2}}},
\]
so that $\|\hat{\Phi}_i\|_{L^2(Z)}=1$, and $D^+_{A_i} \hat{\Phi}_i=0.$ Passing to a subsequence and use ellipticity of Dirac operator, we can find a spinor $\Phi$ of unit $L^2$-norm so that $\Phi_i \to \Phi$ in $C^{\infty}$-topology over any compact subset of $Z$. With the help of Sobolev multiplication $L^2_1 \times L^2_1 \hookrightarrow L^2$ and the convergence on the connection part, we conclude that $D^+_A \Phi=0$. 
\end{proof}

Now over the reducible locus we focus on connections $A$ whose Dirac operator $D^+_A$ has nontrivial kernel. It turns out for generic small perturbation the complex dimension of kernels of this $T^2$-family of Dirac operators $D^+_{\alpha, \beta}$ is at most $1$.

\begin{prop}\label{4.5}
For generic small perturbation $\beta$, There are only finitely many points $\alpha_i \in \mathbb{T}(\s)$ such that $\ker D^+_{\alpha_i, \beta} \neq 0$. Moreover $\ker D^+_{\alpha_i, \beta}= \C$ for each $i$. 
\end{prop}

\begin{proof}
Note that the space of complex Fredholm operators with index $0$ is stratified by the dimension of the kernel (e.g. \cite{K}). The top stratum consists of operators with trivial kernel. The stratum consisting of operators with $1$-dimensional kernel has complex codimension $1$, and the rest strata have higher complex codimension. The proof is an application of the implicit function theorem and this observation. 

Take $(\alpha_0, 0) \in \mathbb{T}(\s) \times \mathcal{P}$. Denote by $\mathcal{H}^1 = \ker D^+_{\alpha_0, 0}, \mathcal{H}^2= \ker D^-_{\alpha_0, 0}$. Let $\Pi: L^2(W^-) \to \im D^+_{\alpha_0, 0} $ be the $L^2$-orthogonal projection. Consider the operator
\begin{align*}
F: \mathbb{T}(\s) \times \mathcal{P} \times L^2_1(W^+) & \longrightarrow \im D^+_{\alpha_0, 0} \\
(\alpha, \beta, \Phi) & \longmapsto \Pi  (D^+_{(\alpha, \beta)} \Phi).
\end{align*}
The differential of this operator on the third component is 
\[
dF|_{(\alpha_0, 0 ,\Phi)} (0, 0 ,\phi)=\Pi  (D^+_{\alpha_0, 0} \phi),
\]
which is surjective. From the implicit function theorem on Banach spaces there exists neighborhood $U_0  \times W_0  \subset \mathbb{T}(\s) \times \mathcal{P}$ of $(\alpha_0, 0)$ and a map $\varphi: U_0  \times W_0 \times \mathcal{H}^1 \to L^2_1(W^+)$ such that $\forall (\alpha, \beta) \in U_0 \times W_0$ one has 
\[
F(\alpha, \beta, \Phi+\varphi(\alpha, \beta, \Phi) )=0, \forall \Phi \in \mathcal{H}^1.
\]
In particular $D^+_{\alpha, \beta}(\Phi+\varphi_0(\alpha, \beta, \Phi)) \in \mathcal{H}^2$. Thus we get a map
\begin{align*}
\eta: U_0 \times W_0 & \longrightarrow \Hom_{\C}(\mathcal{H}^1, \mathcal{H}^2) \\
(\alpha, \beta) & \longmapsto \{ \Phi \mapsto D^+_{\alpha, \beta}(\Phi+\varphi_i(\alpha, \beta, \Phi) \}.
\end{align*}
Note that for $(\alpha, \beta) \in U_0 \times W_0$, $\ker D^+_{\alpha, \beta}\cong \ker \eta(\alpha, \beta)$.  We now turn to study $\eta$. The differential of this operator on the second component is given by 
\[
d_2\eta|_{(\alpha_0, 0)}(0, b) \longmapsto \{\phi \mapsto \Pi (\rho_Z(b) \phi)\}.
\] Denote by $\mathcal{K}_0=\im d_2 \eta|_{(\alpha_0, 0)} \subset \Hom_{\C}(\mathcal{H}^1, \mathcal{H}^2)$ the image of this map. We claim that $\dim_{\C} \mathcal{K}_0 \geq \dim_{\C} \mathcal{H}^1$. \\
To prove the claim, when $\dim_{\C} \mathcal{H}^1=0$, there is nothing to do. We may assume $\dim_{\C} \mathcal{H}^1\geq 1$. Note that $\ind D^+_{\alpha, \beta}=0$ for any $(\alpha, \beta) \in \mathbb{T}(\s) \times \mathcal{P}$, thus $\dim_{\C} \mathcal{H}^1 =\dim_{\C} \mathcal{H}^2=m$. Let $\{\phi_k\}_{k=1}^{m}, \{\psi_l \}_{l=1}^{m}$ be $L^2$-orthonormal basis of $\mathcal{H}^1$ and $\mathcal{H}^2$ respectively. Unique continuation for the Dirac operator $D^{\pm}_{\alpha_0, 0}$ implies that outside a nowhere dense subset of $Z$, $\phi_k$ and $\psi_l$ are nonvanishing. Note that $\langle \psi_l, \psi_k \rangle= \delta_{k}^l$, for small $\epsilon_0 >0$, one can then find $b_l \in \mathcal{P}, l=1, ..., m$ satisfying
\[
|\langle \rho_Z(b_l) \phi_1, \psi_l \rangle | > 1- \epsilon_0, \; |\langle \rho_Z(b_l) \phi_1, \psi_k \rangle | < \epsilon_0, k \neq l.
\]
By choosing $\epsilon_0$ small enough we can make  $d_2 \eta|_{(\alpha_0, 0)} \rho_Z(b_l) \phi_1, l=1, ..., m, $ linearly independent. This proves the claim. \\
Now we pick up a finite dimensional subspace $V_0=\Span \{b_l\}$ of $\mathcal{P}$ whose intersection with $W_0$ is denoted by $W'_0= V_0 \cap W_0$. Let $\Pi^0: \Hom_{\C}(\mathcal{H}^1, \mathcal{H}^2) \to \mathcal{K}_0$ be an orthogonal projection. We take a smaller neighborhood $U_0' \subset U_0$ such that the map 
\begin{align*}
\eta': U'_0\times W'_0 & \longrightarrow \mathcal{K}_i \\
(\alpha, \beta) & \longmapsto \Pi^0( \eta (\alpha, \beta)),
\end{align*}
has surjective differential at all $(\alpha, \beta) \in U_0' \times W_0'$. Then $\eta'^{-1}(0)$ is a smooth submanifold in $U_0' \times W_0'$. Let $\pi: \mathbb{T}(\s) \times \mathcal{P} \to \mathcal{P}$. Sard theorem tells us for generic $\beta \in W_0'$, $\pi^{-1}(\beta) \cap \eta'^{-1}(0)$ is a smooth submanifold of codimension at least $2 \dim_{\C} \mathcal{H}^1$. Since $\dim U'_0=2$, the set $\pi^{-1}(\beta) \cap \eta'^{-1}(0)$ is nonempty only if $\dim_{\C} \mathcal{H}^1 \leq 1$. Thus we conclude that if $\dim_{\C} \mathcal{H}^1 =m \geq 2$, after apply a generic perturbation $\beta \in W_0'$, one gets 
\[
\dim_{\C} D^+_{\alpha, \beta} \leq m-1, \forall \alpha \in U_0'.
\]\ \\
Note that $\mathbb{T}(\s)$ is compact, we can run the argument above for a finite open cover $\{U_i'\}$ of $\mathbb{T}(\s)$ with $\alpha_i \subset U_i'$ and enlarge $V_0$ to a larger finite dimensional subspace so that the property for $\eta'$ holds over all $\mathbb{T}(\s)$. Thus by induction on the dimension of $\ker D^+_{\alpha, \beta}$,  we have proved that for a generic small perturbation $\beta \in \mathcal{P}$ 
\[
\dim_{\C} \ker D^+_{\alpha, \beta} \leq 1, \forall \alpha \in \mathbb{T}(\s). 
\]
Finiteness of $\alpha_j$ with $\ker D^+_{\alpha, \beta}$ follows from the compactness of $\mathbb{T}(\s)$. 
\end{proof}

Finally we arrive at the stage to discuss the Kuranishi picture near a reducible monopole $[\Gamma]=[A,0]$ approached by a sequence of irreducibles. A standard reference for the Kuranishi picture is \cite[Chapter 4]{FU}. Here we shall follow the construction given in \cite[Chapter 12]{MMR}. As their result is stated in the context of instantons, we rephrase it in terms of monopoles. 

\begin{thm}(\cite[Theorem 12.1.1]{MMR}) 
Let $\Gamma=(A, 0) \in \mathcal{C}_{k}(Z)$ be a reducible monopole of finite energy, $E_{\delta}(\Gamma)$ the deformation complex at $\Gamma$ given by (\ref{4cx}). Then there is a $U(1)$-equivariant neighborhood $V_{\Gamma}$ of $0$ in $H^1(E_{\delta}(\Gamma))$ and a $\mathcal{G}_{k+1, \delta}$-invariant neighborhood $U_{\Gamma}$ of $\Gamma$ in $\mathcal{C}_{k}(Z)$ together with two $U(1)$-equivariant smooth maps:
\begin{enumerate}
\item[(i)] $f_{\Gamma}: V_{\Gamma} \longrightarrow \ker \delta^*_{1, \Gamma} \cap U_{\Gamma}$,
\item[(ii)] $\mathfrak{o}_{\Gamma}: V_{\Gamma} \longrightarrow H^2(E_{\delta}(\Gamma))$
\end{enumerate}
such that the differential $Df_{\Gamma}|_0$ is the inclusion $H^1(E_{\delta}(\Gamma)) \hookrightarrow \ker \delta^*_{1, \Gamma}$, and the $U(1)$-quotient of $f_{\Gamma} (\mathfrak{o}_{\Gamma}^{-1}(0)) \subset \delta^*_{1, \Gamma}$ is isomorphic to a neighborhood of $[\Gamma] \in \M_{g, \beta}(Z, \s)$ as a stratified space. 
\end{thm} 

We call $\mathfrak{o}_{\Gamma}: V_{\Gamma} \to H^2(E_{\delta}(\Gamma))$ the Kuranishi obstruction map. By the construction, the constant term and linear term of $\mathfrak{o}_{\Gamma}$ all vanish. Thus the local structure of $[\Gamma] \in \M_{g, \beta}(Z, \s)$ is given by $(D^2\mathfrak{o}_{\Gamma})^{-1}(0)$.\\  
Now we take $\Gamma=(A, 0)$ to be one which is approached by a sequence of irreducibles $(A_i, \Phi_i)$. Lemma \ref{4.4} and Proposition \ref{4.5} tell us that $\ker D^+_{A} =\C$. From the description of the homology of the complex $E_{\delta}(\Gamma)$ in Proposition \ref{2.9}, we can choose orthonormal basis of $H^1(E_{\delta}(\Gamma))$ as $\{a_1, a_2, \phi_1\}$, of $H^2(E_{\delta}(\Gamma))$ as $\{\phi_2\}$, where $\{a_1, a_2\}$ is a real basis of $H^1(Z; i\R)$ and $\{\phi_1\}$, $\{\phi_2\}$ are complex bases of $\ker D_A^+$, $\ker D_{A, \delta}^-$ respectively. The stabilizer $U(1)$ acts trivially on the connection part, and as multiplication on the spinor part. Thus the quotient by the $U(1)$ action identifies each copy of $\C$ as $\R^+$. The construction of $\mathfrak{o}_{\Gamma}$ is given as follows. We decompose $\ker \delta_{1, \Gamma}= H^1(E_{\delta}(\Gamma)) \oplus \im \delta^*_{2, \Gamma}$. The implicit function theorem gives us a $U(1)$-equivariant map 
\[
g_{\Gamma}: V_{\Gamma} \longrightarrow \im \delta^*_{2, \Gamma}
\]
such that the constant and linear terms of $\g_{\Gamma}$ all vanish, and $\forall \mathfrak{r}=(a, \phi) \in V_{\Gamma}$
\[
\Pi_2 \mathfrak{F}_{\omega}(\Gamma + \mathfrak{r}+ g_{\Gamma}(\mathfrak{r})) =0,
\]
where $\Pi_2: L^2_{k-1}(\Lambda^+T^*Z \otimes i\R) \oplus L^2_{k-1}(S^-) \to \im \delta_{2, \delta}$ is the $L^2_{\delta}$-orthogonal projection. Then $f_{\Gamma}(\mathfrak{r})=\Gamma+ \mathfrak{r}+ g_{\Gamma}(\mathfrak{r})$, and 
\[
\mathfrak{o}_{\Gamma}(\mathfrak{r}) = \Pi_1 \mathfrak{F}_{\omega}(f_{\Gamma}(\mathfrak{r})),
\]
where $\Pi_1: L^2_{k-1}(\Lambda^+T^*Z \otimes i\R) \oplus L^2_{k-1}(S^-)  \to H^2(E_{\delta}(\Gamma))$ is the $L^2_{\delta}$-orthogonal projection. We write $\mathfrak{r}+g_{\Gamma}(\mathfrak{r}) =(a+g_1(\mathfrak{r}), \phi+g_2(\mathfrak{r})) = (a', \phi')$. Recall that 
\[
\mathfrak{F_{\omega}}(A+a', \phi') = (d^+a' -2d^+\beta- \rho^{-1}(\phi' \phi'^*)_0, D_A^+\phi'+ \rho(a)\phi'), 
\]
thus 
\[
\mathfrak{o}_{\Gamma}(\mathfrak{r})=\Pi_1 \mathfrak{F_{\omega}}(A+a', \phi') =(1 - \mathcal{D}(\mathcal{D}^*\mathcal{D})^{-1} \mathcal{D}^*)(D_A^+ \phi'+ \rho(a) \phi'),
\]
 where $\mathcal{D}=D_A^+$, $\mathcal{D}^*=D^-_{A, \delta}$. As mentioned above that $g_{\Gamma}(\mathfrak{r})$ vanishes at least to the second order, and Lemma 12.1.2 in \cite{MMR} asserts that the projection $\Pi_1$ in a small neighborhood is one-to-one. Thus after applying a suitable diffeomorphism of $V_{\Gamma}$, the quadratic term $D^2 \mathfrak{o}_{\Gamma}$ is given by the map 
\[
\tilde{\mathfrak{o}}_{\Gamma}(x_1 a_1, x_2 a_2, z_1 \phi_1)=D^+_A g_2(\mathfrak{r}) +x_1z_1 \rho(a_1)\phi_1+ x_2z_1 \rho(a_2) \phi_1, 
\]
where $\mathfrak{r}=(x_1 a_1, x_2 a_2, z_1 \phi_1), x_1, x_2 \in \R, z_1 \in \C.$ Recall that $\ker D^+_A=\C$, then the implicit function theorem gives rise to a smooth map $\nu: \R^2 \to \C$ invertible near $0$ such that 
\[
\tilde{\mathfrak{o}}_{\Gamma}(x_1, x_2, \nu(x_1, x_2)) =0. 
\]
Precompose with $\nu$, we can write $z_1= x_1 + ix_2$ with $\tilde{\mathfrak{o}}_{\Gamma}(x_1, x_2, z_1)=0$. Recall taking the quotient by $U(1)$ identifies $\C$ with $\R^+$. Thus the local structure $[\Gamma] \in \M_{g, \beta}(Z, \s)$ is given by the $U(1)$ quotient of the zero of the map 
\begin{equation}
\begin{split}
\tilde{\mathfrak{o}}_{\Gamma}: \R^2 \times \C & \longrightarrow \C \\
(x_1, x_2, z)  & \longmapsto (x_1 + ix_2) \cdot z. 
\end{split}
\end{equation}

We summarize the above discussion as follows.

\begin{prop}\label{5.7}
Let $[\Gamma] =[A, 0] \in \M^{\Red}_{g, \beta}(Z, \s)$ be a reducible monopole approached by a sequence of irreducible monopoles. For generic metric and perturbation pair $(g, \beta)$, an open neighborhood of $U_{[\Gamma]}$ of $[\Gamma]$ in the total moduli space $\M_{g, \beta}(Z, \s)$ is modeled on the zero set of the following map
\begin{align*}
\tilde{\mathfrak{o}}_{[\Gamma]}: \R^2 \times \R_+ & \longrightarrow \C \\
(x_1, x_2, z) & \longmapsto (x_1 + i x_2) \cdot z, 
\end{align*}
where the $\R_+$ component of $\tilde{\mathfrak{o}}_{[\Gamma]}^{-1}(0)$ represents the irreducibles $U_{[\Gamma]} \cap \M^*_{g, \beta}(Z, \s)$, and the $\R^2$ component represents the reducible part $U_{[\Gamma]} \cap M^{\Red}_{g, \beta}(Z, \s)$. 
\end{prop}

Combining Lemma \ref{2.5}, Proposition \ref{4.8}, and Proposition \ref{5.7}, we get a complete description of the moduli space $\M_{g, \beta}(Z, \s)$.

\begin{thm}\label{KUR}
Let $(Z, g, \s)$ and $(T^3, \mathfrak{t}, h)$ be given as above. After fixing a homology orientation, for a generic choice of metric and small perturbation $(g, \beta)$ we have 
\begin{enumerate}
\item[\upshape (i)] The Seiberg-Witten moduli space $\M_{g, \beta}(Z, \s)$ is an oriented compact smooth stratified space.
\item[\upshape (ii)] The reducible locus $\M^{\Red}_{g, \beta}(Z, \s)$ is diffeomorphic to $T^2$.
\item[\upshape (iii)] The irreducibles $\M^*_{g, \beta}(Z, \s)$ is an oriented smooth $1$-manifold whose components are diffeomorphic to either a circle or an arc such that \\
{\noindent $\bullet$} The closed ends of the arcs are contained in $\bar{\partial}_+^{-1}(\theta)$. \\
{\noindent $\bullet$} The open ends of the arcs lie on $\M_{g, \beta}^{\Red}(Z,\s)$. Moreover near an open end $[\Gamma] \in \M^{\Red}_{g ,\beta}(Z, \s)$ the moduli space is modeled on the zero set of a map $\tilde{\mathfrak{o}}_{[\Gamma]}$ given by 
\begin{align*}
\tilde{\mathfrak{o}}_{[\Gamma]}: \R^2 \times \R^+ & \longrightarrow \C \\
(x_1, x_2, z) & \longmapsto (x_1 + i x_2) \cdot z, 
\end{align*}
where the $\R^+$ component of $\tilde{\mathfrak{o}}_{[\Gamma]}^{-1}(0)$ represents the irreducible part  ,and the $\R^2$ component represents the reducible part.
\item[\upshape (iv)]There is no reducible monopole in $\M_{g, \beta}(Z, \s)$ asymptotic to the singular point, i.e. $\theta \notin \bar{\partial}_+(\M^{\Red}_{g, \beta}(Z, \s))$. 
\end{enumerate}
\end{thm}

Lastly we fix our convention for orientation, and explain the relation among the orientations of the irreducible locus, reducible locus, and the moduli space $\chi(T^3)$.

As discussed in Lemma \ref{2.4} the reducible locus $\M^{\Red}_{g, \beta}(Z, \s)$ is an affine space of $H^1(M; i\R) / H^1(M; i \Z)$, and $\chi(T^3)$ is an affine space of $H^1(T^3; i\R)/ H^1(T^3; i\Z)$. Thus the they are oriented by an orientation of $H^1(M; \Z)$ and $H^1(T^3; \Z)$ respectively. Since the asymptotic map is given by the inclusion and a shifting $[\mathfrak{b}] + i^*$ which is an embedding of $T^2 \hookrightarrow T^3$, we get an orientation of $\im i^*$ from that of $\chi(M)$.  We follow the "fiber-first" convention to assign the orientations so that the orientation of $\im i^*$ followed by a unit vector perpendicular to it in $\chi(T^3)$ agrees with the orientation on $\chi(T^3)$. 

The orientation on the irreducible locus $\M^*_{g, \beta}(Z, \s)$ is well-known as early as the introduction of Seiberg-Witten invariants (e.g. see \cite{M}). Here since we are working with manifolds with cylindrical end, the local deformation theory only works on weighted Sobolev spaces. Thus we work on $\M^*(Z, U^c_{\theta})$. Given $[\Gamma] \in \M^*(Z, U^c_{\theta})$ with a representative $\Gamma=(A, \Phi)$, the deformation complex can be homotoped (via ignoring the $0$-th order terms) to the complex $E_{\delta}(\Gamma)$ as in (\ref{4cx}). Since $\ker D_A^+$ is complex, the spinor part is canonically oriented. $H^0(Z; i\R)$ is canonically oriented as $Z$ is oriented. Thus to orient the interior of the irreducible locus it suffices to fix a choice on $H^1(Z; i\R)$, which we choose to be the same as orienting the reducible locus. The boundary points $\bar{\partial}_+^{-1}(\theta)$ are oriented as the boundary orientation of the irreducible locus. 

\section{Periodic Spectral Flow}\label{Peri}

The notion of periodic spectral flow for a family of Dirac operators over a homology $S^1 \times S^3$ is originally introduced in \cite{MRS}. In this section we first generalize the notion to closed $4$-manifolds with $b_1 >0$, $\sigma=0$ and noncompact $4$-manifolds with appropriate cylindrical end. Then we derive a gluing formula for the periodic spectral flows on cylindrical-ended manifolds with compatible asymptotic behavior, which can be thought of as a generalization to the gluing formula for usual spectral flows as in \cite{CLM1} and \cite{CLM2}. 

\subsection{Periodic Spectral Flow over Closed Manifolds} Let $(X, \s, g)$ be a closed smooth Riemannian spin $4$-manifold with $b_1(X) >0$ and $\sigma(X)=0$. We write $\s=(W^{\pm}, \rho)$ for the spinor bundles and Clifford multiplication. Let $\alpha \in H^1(X; \Z)$ be a primitive class. We choose a smooth function $f: X \to S^1$ so that $[df]=\alpha$. We also regard $\s$ as a spin$^c$ structure and write $A_0$ for the spin connection induced by the Levi-Civita connection. Let $\beta \in \Omega^1(X; i\R)$ be a imaginary valued $1$-form, which we think of as a perturbation as before. For any spin$^c$ connection $A$ we consider the following family of twisted Dirac operators from $L^2_1(X, W^{\pm}) \to L^2(X, W^{\mp})$: 
\begin{equation}
D^{\pm}_{A, z}(X, \beta):= D_A^{\pm}(X)  + \rho(\beta - \ln z \cdot df), \; z \in \C^*.
\end{equation}
\begin{rem}
We write the spin Dirac operator $D^{\pm}_{A_0}=D^{\pm}$ for simplicity. When $A-A_0=a \otimes 1_{W}$, we have $D^{\pm}_{A}=D^{\pm}+\rho(a).$ We choose not to include the class $\alpha$ in the notation unless the dependence on the choice of $\alpha$  becomes relevant in the argument.
\end{rem}
The spectral set of the family of operators $D^+_{A, z}(X, \beta)$ is defined to be 
\begin{equation}
\Sigma(A, \beta) = \{ z \in \C^* : D^+_{A, z}(X, \beta) \text{ is not invertible } \}. 
\end{equation}
Note that for different choices of $f_1, f_2: X \to S^1$ representing $\alpha$, there is a real-valued function $u: X \to \R$ such that $du= df_1 - df_2$. Then one computes that 
\begin{equation}
D^+_{A, z}(X, \beta, f_1)= e^{-zu} D^+_{A, z}(X, \beta, f_2) e^{zu}.
\end{equation}
Thus the spectral set $\Sigma(A, \beta)$ is independent of the choice of $f$. A fundamental property of the spectral set is the following result (see \cite[Theorem 3.1]{T1} and \cite[Theorem 4.6]{MRS}).

\begin{dfn}\label{reg4.3}
We call a path of spin$^c$ connections $A_t$, $t \in [0,1]$, a regular path with respect to $(\s, \beta, \alpha)$ if 
\begin{enumerate}
\item[\upshape (i)] For $i=0,1$, the family of operators $D^+_{A_i, z}(X, \beta)$ has trivial kernel for all $z$ with unit length $|z|=1$. 
\item[\upshape (ii)] The spectral set 
\[
\Sigma_I(A_t, \beta): =\{ (t, z) \in [0,1] \times S^1 : D^+_{A_t, z}(X, \beta) \text{  is not invertible }\}
\]
is discrete and 
\[
\ker D^+_{A_t, z}(X, \beta) = \C, \; \forall (t, z) \in \Sigma_I(A_t, \beta).
\]
\end{enumerate}
\end{dfn}

Since $[0,1] \times S^1$ is a compact manifold of real dimension $2$, the argument in Proposition \ref{4.5} implies that any path $A_t$ is regular with respect to a generic small perturbation $\beta$. We write $\Sigma_I(A_t, \beta)=\{(t_j, z_j)\}_{j=1}^m$. The following result gives us a description of the spectrum of $D^+_{A_t, z}(X, \beta)$ in a neighborhood of $[0,1] \times S^1$ in $[0,1] \times \C^*$. 

\begin{prop}\label{4.3.1}(\cite[Theorem 4.8]{MRS})
Let $(A_t)$ be a regular path. Then there exist $\delta >0$ and neighborhoods $U_{z_j} \subset \C^*$ of $z_j$, $j=1, ..., m$, such that 
\[
\bigcup_{|t-t_j| <\delta} \{ t\} \times (\Sigma(A_t, \beta) \cap U_{z_j} ) \subset [0,1] \times \C^*, \; j=1, ..., m,
\]
is a smoothly embedded curve. 
\end{prop} 

\begin{rem}
The spectral curves are given by the intersection between the family of Dirac operators $D^+_{A_t, z}$ parametrized by the thickened cylinder $\{1- \epsilon <|z| < 1+\epsilon\}$ and the stratum consisting of operators with $1$-dimensional kernel in the space of complex Fredholm maps of index $0$. After a generic perturbation this intersection is transverse, thus gives us the smoothness of the spectral curves. The original proof in \cite{MRS} is formulated in terms of operator algebra. 
\end{rem}
Given $r >0$, $a <b$ we write 
\[
C^r_{[a, b]}:= \{ (t, z) \in [a, b] \times \C^* : |z|=r \}
\]
for the cylinder from $a$ to $b$ with radius $r$. Note that if $t$ satisfies $|t-t_j| \geq \delta$ for all $j=1, ..., m$, $D^+_{A_t, z}(X, \beta)$ is invertible. Thus we can choose $\epsilon_1 >0$ so that 
\begin{enumerate}
\item[\upshape (i)] $B_{\epsilon_1}(z_j) \subset U_{z_j}, \; j=1, ..., m.$ 
\item[\upshape (ii)] $\Sigma(A_t, \beta) \cap \{ z : 1- \epsilon_1 \leq |z| \leq 1+\epsilon_1 \} =\emptyset$ for all $t$ satisfying $|t-t_j| \geq \delta$, $j=1, ..., m$. 
\end{enumerate}
In other words, the $\epsilon_1$-neighborhood of $C^r_{[0,1]}$ only intersects the spectral set in the curves appearing in Proposition \ref{4.3.1}, which we denote by $S_{\epsilon_1}$. We orient $S_{\epsilon_1}$ by the orientation of $[0,1]$. Since $\ker D^+_{A_t, z}(X, \beta)=\C$, $\forall (t, z) \in \Sigma_I(A_t, \beta)$, the $t$-directional derivative of spectral curves has to be positive. Otherwise the vanishing point has kernel of dimension greater than $1$. 

\begin{dfn}
We say $\lambda >0$ is an excluded value for the operator $D^+_{A}(X, \beta)$ if 
\[
D^+_{A, z}(X, \beta) \text{ is invertible for all } |z| = \lambda.
\]
\end{dfn} 

From the assumption $1$ is an excluded value for $D^+_{A_i}(X, \beta)$, $i=0, 1$, of a regular path $(A_t)_{[0,1]}$. Intuitively the periodic spectral flow of the family $D^+_{A_t}(X, \beta)$ is meant to be the signed count of intersections between the spectral curves $S_{\epsilon_1}$ and the cylinder $C^{1}_{[0,1]}$ in such a way that at point in the intersection we assign $-1$ if the curve is entering the cylinder and $+1$ if the curve is leaving. 

\begin{dfn}\label{sf}
Given a regular path $(A_t)_{[0,1]}$, a system of excluded values for $D^+_{A_t}(X, \beta)$ over is a finite sequence of pairs $\{( t_l, \lambda_l)\}_{l=1}^n$ where
\begin{enumerate}
\item[\upshape (i)] $0=t_0 < t_1 < ... < t_n = 1$ is a partition of $[0, 1]$, 
\item[\upshape (ii)] $\lambda_l \in (1 - \epsilon_1, 1+\epsilon_1)$ is an excluded value of $D^+_{A_t}(X, \beta)$ for all $t \in [t_{l-1} , t_l]$. Moreover $\lambda_1=\lambda_{n}=1$. 
\end{enumerate}
\end{dfn}

The continuity of $S_{\epsilon_1}$ implies that if $\lambda_0$ is an excluded value for $D^+_{A_{t_0}}(X, \beta)$, then for all $t$ satisfying $|t-t_0| \leq \delta_0$ with some $\delta_0$ small enough $\lambda_0$ is also an excluded value for $D^+_{A_t}(X, \beta)$. Compactness of $[0,1]$ implies that a system of excluded values always exists. Now we are in a position to give a formal definition of periodic spectral flow, which is inspired by that of the usual spectral flow as in \cite{CLM2}. 

\begin{dfn}
Let $\{ (t_l, \lambda_l)\}_{l=1}^n$ be a system of excluded values for a regular path $(A_t)_{[0,1]}$. For each $l$ in the range $1 \leq l \leq n-1$ we define 
\begin{equation}\label{sfeq}
a_l=\left \{ 
\begin{array}{lll} 
1  & \mbox{if $ \lambda_l > \lambda_{l+1}$} \\
0 & \mbox{if $ \lambda_l=\lambda_{l+1}$} \\
-1  & \mbox{if $ \lambda_l < \lambda_{l+1}$} 
\end{array}
\right.
\end{equation}
and $b_l$ to be the number of spectral points in $\Sigma(A_{t_l}, \beta)$ of length between $\lambda_l$ and $\lambda_{l+1}$. The periodic spectral flow of the family of operators $D^+_{A_t}(X, \beta)$ along the regular path $(A_t)_{t \in [0,1]}$ is 
\begin{equation}\label{sfe}
\Sf (D^+_{A_t}(X, \beta)) := \sum_{l=1}^{n-1} a_lb_l. 
\end{equation}
\end{dfn}

\begin{rem}
Originally the periodic spectral flow is defined in the following way in \cite{MRS}. We write $\Sigma_I(A_t, \beta)=\{(t_j, z_j) \}$ for the set of spectral points on $[0, 1]\times S^1$. Around each point $(t_j, z_j)$ there is a unique smooth spectral curve $\gamma_j: (-\epsilon, \epsilon) \to \C^*$ such that $\gamma_j(0)=z_j$. We write $\ln \gamma_j(t)=u_j(t)+ i v_j(t)$ for real-valued functions $u_j, v_j$. The condition 
\[
\ker D^+_{A_{t_j}, z_j}(X, \beta)= \C
\]
implies that $\dot{u}_j(0) \neq 0$. To each point $(t_j, z_j)$ we assign $-1$ if $\dot{u}_j(0) < 0$, and $+1$ if $\dot{u}_j(0) >0$. Then the periodic spectral flow is defined to be the signed count of the set of spectral points. By drawing pictures locally one see that Definition \ref{sf} is equivalent to this one. 
\end{rem}

\begin{lem}
The periodic spectral flow $\Sf (D^+_{A_t}(X, \beta))$ does not depend on the choice of the system of excluded values $\{(t_l, \lambda_l)\}_{l=1}^n$. In particular the periodic spectral flow is well-defined. 
\end{lem}

\begin{proof}
It's straightforward to see that the periodic $\lambda$-spectral flow is invariant under a refinement of the partition $0=t_0 < t_1 < ... < t_n =1$. Indeed let $0=t'_0< t'_1< ... < t'_{n'}=1$ be a refinement of partition. We let $\lambda_{l'}=\lambda_l$ if $t_{l'} \in  [t_{l-1}, t_l]$ Then by definition $a_lb_l=a_{l'}b_{l'}$ for $t_{l'} =t_l$, and $a_{l'}b_{l'}=0$ otherwise. Thus only the terms with $t_{l'}=t_l$ survive, which leads to the same sum in (\ref{sfe}). 

Given two systems of excluded values, after taking a subdivision of the partition of $[0,1]$, we may assume the partition parts are the same. For each $1 \leq l \leq n$, we have $\lambda_l$ and $\lambda'_l$ as excluded values of $D^+_{A_t}(X, \beta)$ for $t \in [t_{l-1}, t_l]$. By induction we can also assume that there is exactly one $l$ in the range $[1, n]$ such that $\lambda_l \neq \lambda'_{l}$. We may assume $\lambda_l < \lambda'_l$. Then one sees that the quantity 
\[
q(t)=|\Sigma(A_t, \beta) \cap \{ (t, z) : \lambda_l < |z| < \lambda'_l \} |
\]
is constant for all $t \in [t_{l-1}, t_l]$. Then $a'_{l-1}b'_{l-1}=a_{l-1}b_{l-1} \pm q(t)$ and $a'_lb'_l=a_lb_l \mp q(t)$ with cancelling signs depending the value of $\lambda_{l-1}$. Moreover $a'_jb'_j=a_jb_j$ for $j \notin \{ l, l-1\}$. This completes the proof. 
\end{proof}

Although the periodic spectral flow might depend on the perturbation $\beta$, it is still invariant under a small deformation of the perturbation we use to make the path $(A_t)$ regular. 

\begin{lem}\label{hi4.9}
Let $(A_t)_{[0,1]}$ be a regular path with respect to $\beta$. Then there exists a neighborhood $U_{\beta}$ of $\beta$ in $\Omega^1(X; i\R)$ so that $(A_t)$ is regular for a generic choice of perturbation $\beta' \in U_{\beta}$. Moreover 
\[
\Sf(D^+_{A_t}(X, \beta) )= \Sf(D^+_{A_t}(X, \beta')).
\]
\end{lem}

\begin{proof}
Proposition \ref{4.5} says that a generic $\beta'$ in a small neighborhood $U_{\beta}$ of $\beta$ makes $(A_t)$ a regular path. Since the set of such perturbations is of second Baire category, it's locally path connected. Thus we may assume the set of perturbations making $(A_t)$ regular in the neighborhood $U_{\beta}$ is connected. Now take a such path $(\beta_s)_{s \in [0,1]}$ connecting $\beta$ and $\beta'$. This path gives a cobordism 
\[
W=\bigcup_{s \in [0,1]} \Sigma_I((A_t, \beta_s)
\]
from the spectral set $\Sigma_I(A_t, \beta)$ to $\Sigma_I(A_t, \beta')$. By choosing the neighborhood small enough, discreteness of the spectral sets implies that this cobordism embeds as a curve in $[0,1] \times S^1 \times [0,1]$. Suppose a component $z(s)$ of the $1$-dimensional cobordism $W$ connects two points of opposite sign. Then continuity of the spectral curves imply that there is a point $s_0 \in (0, 1)$ with $z(s_0) \in \Sigma(A_{t_0},\beta_{s_0})$ such that the spectral curve is both entering and leaving the unit circle $S^1$ at $z(s_0)$, which is impossible. Note that the periodic spectral flow is the signed count of the spectral points, thus they are equal. 
\end{proof}

\subsection{Periodic Spectral Flow over Cylindrical-Ended Manifolds} Now we generalize the notion of periodic spectral flow over manifolds with cylindrical ends. Let $(M, \s, g)$ be a smooth compact Riemannian spin $4$-manifold with compatible spin boundary $(Y, \mathfrak{t}, h)$. Form $(M_{\infty}, \s)$ and $(M_T, \s)$ as before. Then $g$ naturally extends to a metric $g$ on $M_{\infty}$ and $M_T$. The space of perturbations we are considering here have compact support $\beta \in \Omega^1_c(M_{\infty}, i\R)$. As before we also write $\s, \mathfrak{t}$ for the induced spin$^c$ structures. We impose two topological assumptions on $M$: 
\begin{hypo}\label{hypo4.11}
\hfill
\begin{enumerate}
\item[\upshape (i)] The signature $\sigma(M)=0$. The first Betti number $b_1(M) >0$. 
\item[\upshape (ii)] The restriction map $i^*: H^1(M; \Z) \to H^1(Y; \Z)$ is injective. 
\end{enumerate}
\end{hypo}

Let $\alpha \in H^1(M_{\infty}; \Z)$ be a primitive class with $\alpha|_{Y_t}=\alpha_Y$ independent of $t$. Choose smooth functions $f: M_{\infty} \to S^1$ to represent $[df]=\alpha$ and $f_Y: Y \to S^1$ to represent $[df_Y]=\alpha_Y$ so that $df|_{Y_t} = df_Y$ for all $t \in [0, \infty)$. We write $D(Y): L^2_1(Y, S) \to L^2(Y, S)$ for the spin Dirac operator over $(Y, \mathfrak{t})$ and $B_0$ for the spin connection on $\mathfrak{t}$. 

\begin{dfn}\label{dfn4.12}
We call a spin$^c$ connection $A$ on $(M_{\infty}, \s)$ regularly asymptotically flat (RAF) if the following holds.
\begin{enumerate}
\item[\upshape (i)] $A$ is exponentially asymptotic to a flat spin$^c$ connection $B_o$ on $(Y, \mathfrak{t})$, i.e. over the cylindrical end one writes $A|_{[0, \infty) \times Y}=B(t)+c(t)dt$, then there are constants $T, \delta >0$ such that for all $t >T$
\[
|B(t) - B_o| < C_1 e^{-\delta t} \text{ and } |c(t)| < C_2 e^{-\delta t}.
\]
\item[\upshape (ii)] The twisted Dirac operators 
\begin{equation*}
D_{B_o, z}(Y):=D_{B_o}(Y)- \ln z \rho(df_Y): L^2_1(S) \to L^2(S)
\end{equation*}
are invertible for all $z$ with unit length $|z|=1$.
\item[\upshape (iii)] There exists $\epsilon_Y >0$ such that the twisted Dirac operators 
\begin{equation*}
D^+_{A, z}(M_{\infty}):=D_A^+(M_{\infty}) - \ln z \rho(df) : L^2_1(W^+) \to L^2(W^+)
\end{equation*}
have index $0$ for all $z$ in the range $1-\epsilon_Y < |z| < 1+ \epsilon_Y$.
\end{enumerate}
\end{dfn}

Recall from (\ref{eq1.2}) that the Dirac operator $D^+_{A, z}(M_{\infty})$ restricted to the cylindrical end has the form 
\[
D^+_{A, z}(M_{\infty})|_{[0, \infty) \times Y} = {d \over {dt}} +D_{B, z}(Y)+ c(t)+\rho(B(t)-B).
\] 
Due to the weighted Sobolev embedding theorem (c.f. \cite{L}), exponential decay implies that the zeroth order operator
\[
c(t)+\rho(B(t)-B)): L^2_1(W^+|_{[0, \infty) \times Y}) \to L^2(W^-|_{[0,\infty) \times Y})
\]
is compact. Invoking the classical argument by Atiyah-Patodi-Singer in \cite{APS1}, the invertibility of $D_{B, z}(Y)$ implies the Fredholmness of $D^+_{A, z}(M_{\infty}, \beta)$ for $\beta \in \Omega^1_c(M_{\infty}, i\R)$. Since $B$ is flat, (iii) in Definition \ref{dfn4.12} can be achieved assuming certain topological restraints (e.g. vanishing of signature) on $(M, Y)$ based on the results of Atiyah-Patodi-Singer in \cite{APS1}, \cite{APS2}. Analogous to the case of closed manifolds we have the notion of regular path as well. 
\begin{dfn}
We call a path of RAF connections $(A_t)$, $t \in [0,1]$, a regular path with respect to $(\s, \alpha, \beta)$ if 
\begin{enumerate}
\item[\upshape (i)] For $i=0,1$, the family of operators $D^+_{A_i, z}(M_{\infty}, \beta)$ has trivial kernel for all $|z|=1$. 
\item[\upshape (ii)] For some $\epsilon_1 < \epsilon_Y$ the spectral set 
\[
S_{\epsilon_1}: =\{ (t, z) : |z| \in [1- \epsilon_1, 1+\epsilon_1], D^+_{A_t, z}(X, \beta) \text{  is not invertible }\}
\]
consists of finitely many smoothly embedded curves transverse to the cylinder $[0,1] \times S^1$, i.e. $S_{\epsilon_1}\cap [0,1] \times S^1$ consists of finitely many points, $S_{\epsilon_1} \cap \{0, 1\} \times S^1 = \emptyset$, and 
\[
\ker D^+_{A_t, z}(X, \beta) = \C, \; \forall (t, z) \in S_{\epsilon_1}. 
\]
\end{enumerate}
\end{dfn}

\begin{prop}\label{4.14}
Any path of RAF connections $(A_t)_{[0,1]}$ is regular with respect to a generic small perturbation $\beta$. 
\end{prop}

\begin{proof}
Fix $\epsilon_1 < \epsilon_Y$. Take an arbitrary path $(A_t)_{[0,1]}$. Then we get a family of Dirac operators $D^+_{A_t, z}(M_{\infty})$ of index $0$ parametrized by the compact set $\{ (t, z) : t \in [0,1], |z| \in [1-\epsilon_1, 1+\epsilon_1]\}$. Then we can apply the argument in Proposition \ref{4.5} to this case to conclude that for generic small $\beta \in \Omega^1_c(M_{\infty}, i\R)$ the family of Dirac operators $D^+_{A_t, z}(M_{\infty}, \beta)$ is transverse to the space of Fredholm operators of index $0$ stratified by the dimension of kernels. The stratum consisting of kernels with complex dimension $1$ has complex codimension $1$. The other lower strata have higher complex codimension. 
\end{proof}

Given a regular path $(A_t)_{[0,1]}$ we define the periodic spectral flow 
\[
\Sf(D^+_{A_t}(M_{\infty}, \beta))=\sum_l a_l b_l
\] as in the closed case (Definition \ref{sf}) with respect to a system of excluded values $\{(t_l ,\lambda_l)\}$. 

\subsection{A Splitting Formula} 
In this section we derive a splitting formula relating the periodic spectral flows over a closed manifold and manifolds with cylindrical end as we run the neck-stretching process. Let $(X, \s, g)$ be a closed Riemannian spin $4$-manifold decomposed as $X=M \cup N$ with $M \cap N=Y$, which enables us to invoke the neck-stretching set-up in Section \ref{nssu}. We assume both $M$ and $N$ satisfies Hypothesis \ref{hypo4.11}. Let $\alpha \in H^1(X; \Z)$ be a primitive class. We write the restrictions as $\alpha_M=i_M^* \alpha, \alpha_N=i_N^* \alpha$, which are both primitive. We choose representative $f: X \to S^1$ so that $df|_{Y_t}=df_Y$, $t \in (-1,1)$, for some function $f_Y: Y \to S^1$. Since $f$ is parallel along the $t$-direction, we get representatives $f_M: M_{\infty} \to S^1$ and $f_N: N_{\infty} \to S^1$ by extending the restriction of $f$ over the cylindrical end.

Now suppose for each $T \in [0, \infty)$ we have a path $(A_{T, s})_{s \in [0,1]}$ of spin$^c$ connections over $(X_T, \s)$. Choose $\beta_T \in \Omega^1(X_T; i\R)$ to be a perturbation so that $\supp \beta_T \subset M$ and $(A_{T, s})$ is regular with respect to $\beta_T$. We denote by $\epsilon_T$ the "$\epsilon_1$" that appeared in Definition \ref{sf}. Then we get the periodic spectral flow $\Sf(D^+_{A_{T, s}}(X_T, \beta_T))$ for each $T$. We make the following assumptions on the convergence of $(A_{T, s})$:

\begin{hypo}\label{hypo3}
There are regular paths of RAF spin$^c$ connections $A_{o, s} , A'_{o, s}$ on $(M_{\infty}, \s)$ and $(N_{\infty}, \s)$ respectively asymptotic to the same path of connections $B_{o, s}$ on $(Y, \mathfrak{t})$ such that the family of path of connections $(A_{T,s })$ converges uniformly to $(A_{o, s}, A'_{o, s})$ in the following sense 
\begin{enumerate}
\item[$\bullet$] For each $s \in [0,1]$
\[
A_{T, s}|_{M_T} \xrightarrow{L^2_{k, loc}} A_{o, s} \text{ and } A_{T, s}|_{N_T} \xrightarrow{L^2_{k, loc}} A'_{o, s}.
\]
\item[$\bullet$] Given any $\epsilon>0$, there exists $T_o$ such that for any $T >T_o$ one has 
\[
\|A_{T, s} |_{I_{T-T_o}} - B_{o, s} \|_{L^2_k(I_{T-T_o})} < \epsilon. 
\]
\end{enumerate}
\end{hypo}

The splitting formula we are going to prove states as follows. 

\begin{thm}\label{sf4.14}
Assume that Hypothesis \ref{hypo3} holds. Suppose $D^+_{A'_{o,s}, z}(N_{\infty})$ is invertible for all $s \in [0,1]$ and $z$ with unit length $|z|=1$. Then there exists $T_o >0$ such that for all $T >T_o$ one has 
\begin{equation}
\Sf(D^+_{A_{T, s}}(X_T, \beta_T)) = \Sf(D^+_{A_{o, s}}(M_{\infty}, \beta_o)).
\end{equation}
\end{thm}

The idea of the proof is to show that the spectral curves of $D^+_{A_{T,s}}(X_T, \beta_T)$ will approach those of $D^+_{A_{o, s}}(M_{\infty}, \beta_o)$ as $T \to \infty$ due to the absence of spectral curves of $D^+_{A'_{o,s}}(N_{\infty})$. 

\begin{lem}\label{4.15.1}
Suppose $D^+_{A_{o,s_0}, z_0}(M_{\infty}, \beta_o)$ is invertible for some $(s_0, z_0) \in [0, 1] \times \C^*$. Then there are constants $\delta_1 >0, T_1 >0$ such that for all $T >T_1$, $z$ in the range $|z-z_0| < \delta_1$, $D^+_{A_{T, s_0}, z}(X_T, \beta_T)$ is also invertible. 
\end{lem}

\begin{proof}
Note that $\ind D^+_{A_{T, s},z}(X_T, \beta_T)=0$. Thus being invertible is equivalent to having trivial kernel. Suppose there is not true. Then one can find sequences $T_n \to \infty, z_n \to z_0$ such that there exists a sequence of spinors 
\[
\Phi_n \in \ker D^+_{A_{T_n, s_0}, z_n}(X_{T_n}, \beta_{T_n}) \text{ with } \|\Phi_n\|_{L^2(X_T)}=1.
\]Due to the invertibility of $D^+_{A'_{o, s_0}, z}(N_{\infty})$, a convergent subsequence gives us a non-zero element in $\ker D^+_{A_{o, s_0}, z_0}(M_{\infty}, \beta_o)$. This is a contradiction. 
\end{proof}

The next lemma asserts that a spectral curve of $D^+_{A_{T, s}}(X_T, \beta_T)$ will indeed show up in a neighborhood of a spectral curve of $D^+_{A_{o,s}}(M_{\infty}, \beta_o)$. 

\begin{lem}\label{4.16.1}
Suppose $\ker D^+_{A_{o, s_0}, z_0}(M_{\infty}, \beta_o) = \C$ for some $(s_0, z_0) \in [0,1] \times \C^*$. Then for any $\epsilon_2 >0$ there exists $T_2 >0$ such that for any $T >T_2$ one can find a neighborhood $U_{\beta_T}$ of $\beta_T$ in $\Omega^1(X_T; i\R)$ so that for generic $\beta'_T \in U_{\beta_T}$ there exists $z_T$ in the range $|z_T - z_0| <\epsilon_2$ satisfying 
\begin{equation}\label{4.18.1}
\ker D^+_{A_{T, s_0}, z_T}(X_T, \beta_T+\beta'_T) \neq 0.
\end{equation}
Moreover for each $T >T_2$, one can choose $\epsilon_T < \epsilon_2$ so that for a generic $\beta'_T \in U_{\beta_T}$ there is a unique $z_T$ depending on $\beta'_T$ in the range $|z_T - z_0| < \epsilon_T$ satisfying (\ref{4.18.1}). 
\end{lem}

\begin{proof}
Let $\Phi_o \in \ker D^+_{A_{o, s_0}, z_0}(M_{\infty}, \beta_o)$ have $\| \Phi_o\|_{L^2}=1$. Take for each $T >0$, we take a cut-off function $\eta(T):M_{\infty} \to [0, 1]$ with $\supp \eta(T) \subset M_T$. After appropriate rescaling $\Phi_T:=C_T \eta(T) \Phi_o$ is a spinor over $X_T$ extending as $0$ on $N_T \subset X_T$ with unit $L^2$-norm $\|\Phi_T\|_{L^2(X_T)}=1$. The convergence of $A_{T, s_0} \to A_{o, s_0}$ and $\beta_T \to \beta_o$ implies that there is a function $\delta: [0, \infty) \to \R_+$ depending on $A_T, \beta_T, C_T, \eta(T)$ such that 
\begin{equation}
\| D^+_{A_{T, s_0}, z_0}(X_T, \beta_T) \Phi_T \|_{L^2_{k-1}(X_T)} \leq \delta(T). 
\end{equation}
Moreover $\delta(T) \to 0$ as $T \to \infty$. 

To simplify the notations, we write $w_T=\ln z_T - \ln z_0$ and 
\[
D^+_{w, T}= D^+_{A_{T, s_0}, z_0}(X_T, \beta_T) - w \rho(df_T) = D^+_{A_{T, s_0}, z}(X_T, \beta_T).
\]
Now for $T$ large enough we want to find $\phi_T \in L^2_k(X_T, W^+)$ and $z_T$ in the range $|z_T - z_0| <\epsilon_2$ such that for generic $\beta'_T$ close to $\beta_T$ one has 
\begin{equation}
(D^+_{w_T, T} +\rho(\beta'_T))(\Phi_T +\phi_T)=0.
\end{equation}
To further simplify the notation, we drop $T$'s in the above equation as well. Then the equation can be rearranged as 
\begin{equation}\label{eq4.7}
D^+_{w, \beta'} \phi + \rho(\beta' -wdf)\Phi = -D^+\Phi.
\end{equation}
Let's write $\mathcal{Q}(\phi, w, \beta')=D^+_{w, \beta'} \phi + \rho(\beta' -wdf)\Phi $ for the left hand side of (\ref{eq4.7}). Its differential at $(0, 0, 0)$ is given by 
\begin{equation}
Q_0(\varphi, v, \eta) := D\mathcal{Q}|_{(0, 0, 0)}(\varphi, v, \eta) =D^+\varphi +\rho(\eta- vdf) \Phi.
\end{equation}
A standard argument (c.f. \cite[Lemma 27.1.1]{KM1}) shows that the differential $Q_0(\phi, 0 ,\eta)$ is surjective. Thus the implicit function theorem tells us given $w$ in a small neighborhood of the origin and $T$ sufficiently large so that $\delta(T)$ is small, there exists $(\phi(w), w, \beta'(w))$ in a small neighborhood of $(0, 0, 0)$ solving (\ref{eq4.7}). 

Now we know $\mathcal{Q}^{-1}(-D^+\Phi) \neq \emptyset$. Let $(\phi, w, \beta')$ be a solution. The differential is given by 
\begin{equation}
D\mathcal{Q}|_{(\phi, w, \beta')}(\varphi, v, \eta) =D^+_{w, \beta'}+\varphi +\rho(\eta- vdf) (\phi+\Phi),
\end{equation}
which is surjective by the same argument. Thus $-D^+\Phi$ is a regular value for $\mathcal{Q}$. Now consider the projection of $\mathcal{Q}^{-1}(-D^+\Phi)$ to its third component and apply Sard-Smale theorem to get the genericity of existence of $\beta'$. 

Since all operators have index $0$, and the operators corresponding to $z_T$ have nontrivial kernal. Dimension counting gives us that the set of $z_T$ is discrete. Now the uniqueness of $z_T \in B_{\epsilon_T}(z_0)$ for some $\epsilon_T$ follows from the discreteness of the spectral curves of $D^+_{A_{T,s}}(X_T, \beta_T)$. 
\end{proof}

Now we are ready to prove the splitting formula. 
\begin{proof}[Proof of Theorem \ref{sf4.14}]
Let $\{(s_l, \lambda_l)\}$ be a system of excluded values for the path $(A_{o, s})_{s \in [0,1]}$. Then $\lambda_l$ is a excluded value for all $s \in [s_{l-1}, s_l]$. Together with compactness of the circle $\{ |z|=\lambda_l\}$, Lemma \ref{4.15.1} implies that for each $l$ there exists $T_l$ such that for all $T >T_l$, $\lambda_l$ is an excluded value for $D^+_{A_{T, s}}(X_T, \beta_T)$ for $s \in [s_{l-1}, s_l]$. Since there are finitely many $s_l$, we conclude that there is $T_1$ such that for all $T > T_1$, $\{(s_l, \lambda_l)\}$ is also a system of excluded values for the path $(A_{T, s})$. 

For each $T >T_1$, we get $(a_l(T), b_l(T))$ as well as $(a_l(\infty), b_l(\infty))$ in the definition of periodic spectral flow in (\ref{sfeq}). From the definition we know that $a_l(T)$ is independent of the family of operators, thus independent of $T$. Let $\{z_j \}$, $j=1, ..., b_l(\infty)$ be the spectral points of the operator $D^+_{A_{o, s_l}}(M_{\infty}, \beta_o)$ with length in between $\lambda_l$ and $\lambda_{l+1}$. By Lemma \ref{4.16.1} there exists $T'_l >0$ and $\epsilon_l$ small enough such that for all $T >T'_l$, there is exactly one spectral point  $z_{j, T} \in B_{\epsilon_l}(z_j)$ for $D^+_{A_{T, s_l}}(X_T, \beta_T+\beta'_T)$. Note that there are finitely many $s_l$, thus genericity of $\beta'_T$ ensures us $\beta'_T$ can be chosen uniformly for all $s_l$ with respect to a fixed $T$. This captures all spectral points of $D^+_{A_{T, s_l}}(X_T, \beta_T)$ with length in between $\lambda_l$ and $\lambda_{l+1}$, for otherwise one gets a sequence of spectral sequence converging to a invertible point $z$ for $D^+_{A_o, s_l}(M_{\infty, \beta_o})$. Again finiteness of $s_l$ enables us to take $T'_l$ to be uniform as $T_2$. Thus we conclude that for any $T>T_2$ $b_l(T)=b_l(\infty)$ is independent of $T$ as well. Now invoke Lemma \ref{hi4.9} to get rid of the dependence on the small perturbation $\beta'_T$ we have chosen. This completes the proof. 
\end{proof}

\section{Excision Principle over End-Periodic Manifolds}\label{Exci}

Originally the excision principle for Fredholm operators is only proved in the case when the excision part has compact closure. It will be useful for us to extend the excision principle to periodic Fredholm operators over end-periodic manifolds where the closure of the excision part is periodic but noncompact. For simplicity we work in dimension $4$ with Dirac operators. The generalization to other first order elliptic operators are not hard with corresponding assumptions imposed.The argument we will use in this section is to modify the proof of the version from Charboneau's thesis \cite[Appendix B]{C}. 

Let $Z$ be a smooth Riemannian $4$-manifold with periodic end modeled on $W_+=W_0 \cup W_1 \cup ...$, i.e. one can decompose $Z=M \cup W_+$, where $M$ is a manifold with boundary $Y$, each $W_i$ is a copy of a compact cobordism $W: Y \to Y$. Let's write $X$ for the manifold obtained by identifying the two boundary components of $W$ using the identity map. We refer $X$ as the furled manifold of the end of $Z$. We denote by $\pi: W_+ \to X$ the projection map. 
\begin{dfn}
For the sake of narration we fix our terminology for saying end-periodic.
\begin{enumerate}
\item[\upshape (i)] An end-periodic vector bundle $E \to Z$ is a vector bundle restricting to the end of the form $E|_{W_+}= \pi^*E'$ for some bundle $E'$ over $X$. 
\item[\upshape (ii)] An end-periodic differential operator between two end-periodic bundles $E, F$ over $Z$ is a differential operator
\[
D: C^{\infty}(Z, E) \to C^{\infty}(Z, F)
\]
restricting to the end $D|_{W_+}=\pi^* D'$ for some differential operator $D': C^{\infty}(X, E') \to C^{\infty}(X, F')$.

\item[\upshape (iii)] An end-periodic open partition of $Z$ is a partition $Z=U \cup V$ into open subsets so that 
$U \cap W_+$ and $V \cap W_+$ are both invariant under the covering transformation. In other words $U \cap W_+ = \pi^{-1}(U')$ and $V \cap W_+=\pi^{-1}(V')$ for a open cover $X=U' \cup V'$.  
\end{enumerate}
\end{dfn}

In order to fit with our further applications, we want to refine our set-up a little bit more . Let 
\[
Z=U \cup V, D: C^{\infty}(E) \to C^{\infty}(F),
\] 
be a set of end-periodic data as above. We say the set of data is spin$^c$, and has a codimension-$1$ overlap if it has the following form
\begin{enumerate}
\item[\upshape (i)]The overlap is given by 
\[
U \cap V = (-1, 1) \times N, 
\]
where $N \hookrightarrow Z$ is an embedded end-periodic $3$-manifold. Moreover the metric restricts on the overlap has the form 
\[
g|_{(-1, 1) \times N} =dt^2+ h.
\]

\item[\upshape (ii)] $Z$ admits an end-periodic spin$^c$ structure $\s=(\rho, W^{\pm})$ with $E=W^+, F=W^-$. The restriction of the spin$^c$ structure on the overlap $\s|_{(-1, 1) \times N}$ is the pullback of a spin$^c$ structure $\mathfrak{t}$ over $N$. 
\item[\upshape (iii)] The differential operator is a Dirac operator
\[
D=D^+_A: C^{\infty}(E) \longrightarrow C^{\infty}(F),
\]
where $A$ is a spin$^c$ connection on $\s$. Moreover the operator extends as a Fredholm operator 
\[
D: L^2_1(Z, E) \longrightarrow L^2(Z, F).
\]
where $h$ is a metric on $N$. 
\end{enumerate}
Since $N$ is end-periodic, it's given by the pull-back of an embedded $3$-manifold $N' \hookrightarrow X$ in the folded manifold. We write $Z_T$ for the manifold obtained from $Z$ by inserting a cylinder $[-T, T] \times N$ into the overlap $U \cap V$. Then $Z_T = U_T \cup V_T$ with $U_T \cap V_T = (-T-1, T+1)$. Equivalently we can think of $Z_T$ as obtained from $Z$ by deforming the metric $g$ on the overlap. Thus we identify $(Z, g_T)$ with $(Z_T, g)$. We denote the operator over $Z_T$ by $D_T$. The folded manifold of $Z_T$ is denoted by $X_T$. We also write 
\[
X_{\infty}: = X \backslash V' \cup (-1, \infty) \times N' \bigcup (-\infty, 1) \times N' \cup X \backslash U'.
\]
We write the operator on $X_{\infty}$ as
\[
D'_{\infty}: C^{\infty}(X_{\infty}, E') \longrightarrow C^{\infty}(X_{\infty}, F').
\]

let $\tau: Z \to \R$ be a real-valued function such that $\tau|_M \in [-1, 0]$, and over the periodic end $\tau(x+1) = \tau(x)+1$, where $x \in W_+$, and $+1: W_i \to W_{i+1}$ is the covering transformation on $W_+$.  Note that $\tau$ descends to a map $\tau': X \to S^1$. Given $\delta \geq 0$, the weighted Sobolev space $L^2_{k, \delta}(Z, E)$ over end-periodic manifold  is  defined as the completion of $C_0^{\infty}(Z, E)$ with respect to the norm
\[
\|f\|_{L^2_{k, \delta}} := \| e^{\delta \tau} f \|_{L^2_k}.
\]
The Fredholmness of the extension $D: L^2_{k, \delta}(E) \to L^2_{k-1, \delta} (F)$ is characterized by the following result:

\begin{lem}\label{fh5.3}(\cite[Lemma 4.3]{T1})
The extended operator $D: L^2_{k, \delta}(E) \to L^2_{k-1, \delta} (F)$  is Fredholm if and only if the operators
\[
D'_z:= D' - \ln z \nabla \tau' * : L^2_k(X, E') \longrightarrow L^2_{k-1}(X, F'),
\]
are invertible for all $z \in \C^*$ with length $|z| = e^\delta$, where $\nabla \tau' *$ is the algebraic operation given by the symbol of $D'$ coupled with $\tau'$.
\end{lem}

\begin{rem}
As we are dealing with Dirac operators, $\nabla \tau' * =\rho(d\tau')$ is given by the Clifford multiplication. 
\end{rem}

Given two sets of periodic spin$^c$ codimension-$1$ data: 
\begin{equation}\label{2pd}
\begin{split}
Z_1=U_1 \cup V_1, & \; \; D_1: C^{\infty}(E_1) \to C^{\infty}(F_1), \\
Z_2=U_2 \cup V_2, & \; \; D_2: C^{\infty}(E_2) \to C^{\infty}(F_2),
\end{split}
\end{equation}
where for $i=1, 2$, $Z_i=U_i \cup V_i$ is an end-periodic open partition of an end-periodic manifold $Z_i$, $D_i: C^{\infty}(E_i) \to C^{\infty}(F_i)$ is an end-periodic first order elliptic differential operator. Suppose the overlaps are identified as $(-1, 1) \times N \cong U_1 \cap V_1 \cong U_2 \cap V_2$. Moreover on the overlap the operators 
\[
D_1|_{(-1, 1) \times N} \cong D_2|_{(-1, 1) \times N}\]
 and 
\[
E_1|_{(-1, 1) \times N} \cong E_2|_{(-1, 1) \times N}, F_1|_{(-1, 1) \times N} \cong F_2|_{(-1, 1) \times N}
\]
are identified as end-periodic differential operators and vector bundles. All the identification maps are smooth, periodic, and have uniformly bounded differential and inverse. Now we form another two sets of data by interchanging the decompositions:
\begin{equation}
\begin{split}
\tilde{Z}_1= U_1 \cup V_2, & \; \; \tilde{D}_1: C^{\infty}(\tilde{M}_1, \tilde{E}_1) \to C^{\infty}(\tilde{M}_1, \tilde{F}_1), \\
\tilde{Z}_2=V_1 \cup U_2, & \; \; \tilde{D}_2: C^{\infty}(\tilde{M}_2, \tilde{E}_2) \to C^{\infty}(\tilde{M}_2, \tilde{F}_2),
\end{split}
\end{equation}
where the new operators are defined by 
\begin{equation}
\begin{split}
\tilde{D}_1|_{U_1}=D_1|_{U_1}, \; \; & \tilde{D}_1|_{V_2}=D_2|_{V_2} \\
\tilde{D}_2|_{U_2}=D_2|_{U_2}, \; \; & \tilde{D}_2|_{V_1}=D_1|_{V_1}
\end{split}
\end{equation}
We say such two sets of data excisable if $\tilde{D}_i$, $i=1, 2$, extends as Fredholm operators
\[
\tilde{D}_i: L^2_1(\tilde{Z}_i, \tilde{E}_i)  \to L^2(\tilde{Z}_i, \tilde{F}_i). 
\]

\begin{figure}[h]\label{PeriEx}
\centering
\begin{picture}(200,110)
\put(-120, -15){\includegraphics[width=1.15\textwidth]{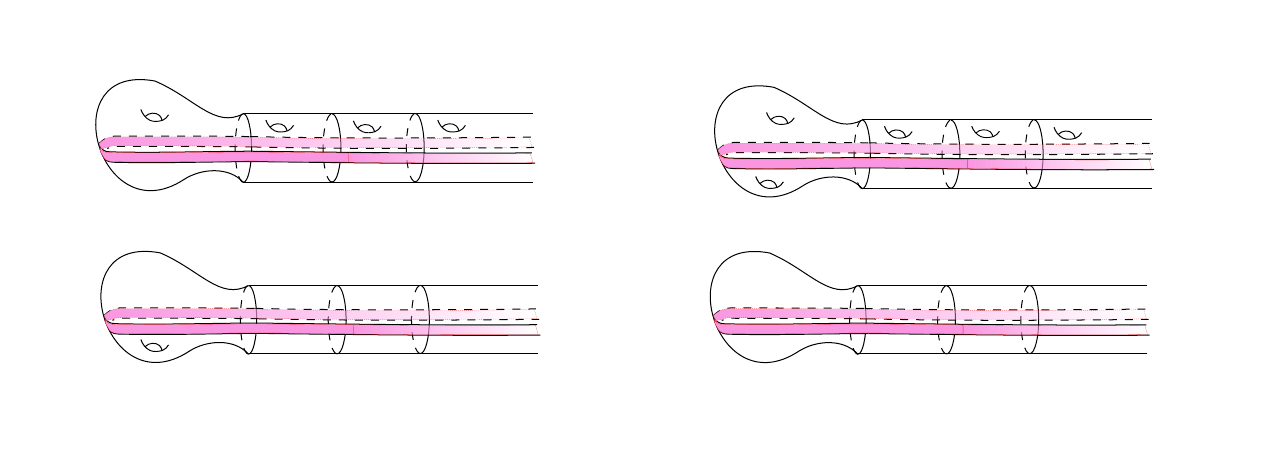}}
\put(-50, 60){$Z_1=U_1 \cup U_2$}
\put(90, 60){$\Longrightarrow$}
\put(165, 60){$\tilde{Z_1}=U_1 \cup V_2$}
\put(-50, 0){$Z_2=V_1 \cup V_2$}
\put(165, 0){$\tilde{Z}_2=V_1 \cup U_2$}
\end{picture}
	\caption{End-Periodic Excision}
\end{figure}
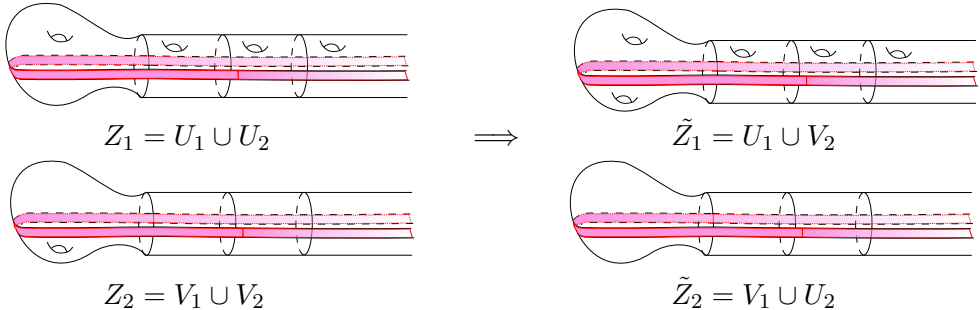

\begin{thm}\label{EP}(Excision Principle)
Given two sets of excisable spin$^c$ codimension-$1$ periodic data as above, suppose the operators
\begin{equation}\label{EPhypo}
\begin{split}
D'_{1, \infty}: L^2_1(X_{1, \infty}, E'_1) & \longrightarrow L^2(X_{1, \infty}, F'_1) \\
D'_{2, \infty}: L^2_1(X_{2, \infty}, E'_2) &\longrightarrow L^2(X_{2, \infty}, F'_2) 
\end{split}
\end{equation}
are both invertible. 
Then there exists $T_o >0$ such that for any $T > T_o$ one has 
\[
\ind D_{1, T} + \ind D_{2, T} = \ind \tilde{D}_{1, T} + \ind \tilde{D}_{2, T}. 
\]
\end{thm}

\begin{proof}
Let's omit the stretching parameter $T$ in the notation unless it becomes relevant. In the definition of the weighted Sobolev spaces we choose the weight functions $\tau_1: Z_1 \to [-1, \infty)$ and $\tau_2: Z_2 \to [-1, \infty)$ to satisfy 
\[
\tau_1|_{(-1,1) \times N}=\tau_2|_{(-1, 1) \times N}. 
\]
Thus we can form the obvious another two weight functions $\tilde{\tau}_1: \tilde{Z}_1 \to [-1, \infty)$ and $\tilde{\tau}_2: \tilde{Z}_2 \to [-1, \infty)$. We choose square roots of partitions of unity $(\phi_1, \psi_1)$ and $(\phi_2, \psi_2)$ subordinate to $(U_1, V_1)$ and $(U_2, V_2)$ respectively to satisfy 
\begin{enumerate}
\item[$\bullet$] $\phi_1^2+\psi_1^2=1$, and $\phi_2^2+\psi_2^2=1$,
\item[$\bullet$] $\phi_1|_{(-1, 1) \times N}=\phi_2|_{(-1, 1) \times N}$, and $\psi_1|_{(-1, 1) \times N}=\psi_2|_{(-1, 1) \times N}$,
\item[$\bullet$] $\| \nabla \phi_i\|_{L^{\infty}} < \epsilon(T)$, and $\|\nabla \psi_i\|_{L^{\infty}} < \epsilon(T)$, $i=1, 2$, where $\epsilon(T) \to 0$ as $T \to \infty$. 
\end{enumerate}
Now we consider maps on the space of smooth functions
\begin{equation}
\begin{split}
\Phi: C^{\infty}(Z_1) \oplus C^{\infty} (Z_2) & \to C^{\infty}(\tilde{Z}_1) \oplus C^{\infty}(\tilde{Z}_2) \\
\Psi:  C^{\infty}(\tilde{Z}_1) \oplus C^{\infty}(\tilde{Z}_2) & \to C^{\infty}(Z_1) \oplus C^{\infty} (Z_2) 
\end{split}
\end{equation}
represented by multiplication of matrices 
\[
\Phi=
\begin{pmatrix}
\phi_1 & \psi_2 \\
-\psi_1 & \phi_2
\end{pmatrix}, \;\;
\Psi=
\begin{pmatrix}
\phi_1 & - \psi_1 \\
\psi_2 & \phi_2
\end{pmatrix}.
\]
From the assumptions of $(\phi_1, \psi_1)$ and $(\phi_2, \psi_2)$, the maps are well-defined. Moreover it's straightforward to compute that 
\[
\Phi \circ \Psi=
\begin{pmatrix}
\phi_1^2 + \psi_2^2 & -\phi_1\psi_1 + \psi_2\phi_2 \\
-\psi_1\phi_1 + \phi_2\psi_2 & \psi_1^2+\phi_2^2
\end{pmatrix}
=1.
\]
Similarly $\Psi \circ \Phi=1$. Thus $\Phi$ and $\Psi$ are inverse to each other. 

Now we claim that both $\Phi$ and $\Psi$ extend as bounded linear maps in $L^2$- and $L^2_1$-norms. To see we get bounded maps between $L^2$ spaces, one computes
\begin{equation}
\begin{split}
\|\Psi (f_1, f_2)\|^2_{L^2} & =\int_{Z_1} \phi_1^2|f_1|^2+\psi_1^2|f_2|^2 + \int_{Z_2}\psi_2^2|f_1|^2 + \phi_2^2 |f_2|^2 \\
& = \int_{Z_1} |f_1|^2 + \int_{Z_2} |f_2|^2\\
& = \|f_1\|^2_{L^2} + \|f_2\|^2_{L^2}. 
\end{split}
\end{equation}
Similarly we conclude $\| \Phi(f_1, f_2)\|^2_{L^2} = \|f_1\|^2_{L^2} + \|f_2\|^2_{L^2}.$
As for boundedness between $L^2_1$ spaces, one has 
\begin{equation*}
\begin{split}
\|\Psi (f_1, f_2)\|^2_{L^2_1} 
 = & \|\phi_1f_1 -\psi_1f_2\|^2_{L^2_1} + \|\psi_2 f_1 + \phi_2 f_2\|^2_{L^2_1} \\
 = & \|f_1\|^2_{L^2} +  \|f_2\|^2_{L^2}  + \| \nabla (\phi_1f_1 - \psi_1 f_2)\| ^2_{L^2}\\
  +& \|\nabla (\psi_2f_1+ \phi_2f_2)\|^2_{L^2} 
\end{split}
\end{equation*}
Using the uniform bound on $\nabla \phi_i, \nabla \psi_i,$ $ \phi_i,$ and $\psi_i$, $i=1, 2$, we get 
\begin{align*}
\| \nabla (\phi_1f_1 - \psi_1 f_2)\|^2_{L^2} & \leq C_1 (\|f_1 \|^2_{L^2_1} + \|f_2\|^2_{L^2_1}) \\
 \|\nabla (\psi_2f_1+ \phi_2f_2)\|^2_{L^2} & \leq C_2(\|f_1 \|^2_{L^2_1} + \|f_2\|^2_{L^2_1}).
\end{align*}
Thus 
\[
\|\Psi (f_1, f_2)\|_{L^2} \leq C_3 (\|f_1 \|_{L^2_1} + \|f_2\|_{L^2_1}).
\]
A similar computation gives rise to 
\[
\| \Phi(f_1, f_2) \|_{L^2_{1}} \leq C_4  (\|f_1\|_{L^2_{1}} + \|f_2\|_{L^2_{1}}). 
\]

Now we let 
\begin{align*}
D:=D_1 \oplus D_2 : & L^2_{1}(E_1) \oplus L^2_{1}(E_2) \to L^2(F_1) \oplus L^2(F_2) \\
\tilde{D}:=\tilde{D}_1 \oplus \tilde{D}_2 : & L^2_{1}(E_1) \oplus L^2_{1}(E_2) \to L^2(F_1) \oplus L^2(F_2)
\end{align*}
so that 
\[
\ind D= \ind D_1 + \ind D_2, \; \; \ind \tilde{D}= \ind \tilde{D}_1 + \ind \tilde{D}_2.
\]
Now one computes that 
\begin{equation}
\begin{split}
\Psi \tilde{D} \Phi &=
\begin{pmatrix}
\phi_1 & -\psi_1 \\
\psi_2 & \phi_2 
\end{pmatrix}
\begin{pmatrix}
\tilde{D}_1 & \\
& \tilde{D}_2
\end{pmatrix}
\begin{pmatrix}
\phi_1 & \psi_2 \\
-\psi_1 & \phi_2
\end{pmatrix}
\\
&=
\begin{pmatrix}
\phi_1 D_1 \phi_1 + \psi_1 D_1 \psi_1 & \phi_1 D_2 \psi_2 - \psi_1 D_2 \phi_2 \\
\psi_2 D_1 \phi_1 -\phi_2 D_1 \psi_1 & \psi_2 D_2 \psi_2 + \phi_2 D_2 \phi_2 
\end{pmatrix}
\\
&= D+K, 
\end{split}
\end{equation}
where $K: L^2_1(E_1) \oplus L^2_1(E_2) \to L^2(F_1)\oplus L^2(F_2)$ is a bounded zeroth order periodic operator supported on $U_1 \cap V_1 \times U_2 \cap V_2$. Since $\tilde{D}$ is Fredholm, $\Phi$ and $\Psi$ are bounded invertible maps on both $L^2$ and $L^2_1$ spaces, we know that $\Psi \tilde{D} \Phi$ is Fredholm with 
\[
\ind D+K = \ind \tilde{D}.
\]
Let's write $K=(K_{ij})$. One computes that 
\begin{equation}
\begin{split}
K_{11} f_1 & = \phi_1 D \phi_1 f_1 +\psi_1 D_1 \psi_1 f_1 - D_1 (\phi_1^2 f_1 - \psi_1^2 f_1) \\
             & = -(\phi_1 \rho(d\phi_1) + \psi_1 \rho(\psi_1)) f_1
\end{split}
\end{equation}
Since Clifford multiplication is norm-preserving and $\phi_1, \psi_1, |\nabla \phi_1|, |\nabla \psi_1| \in [0, 1]$, we conclude 
\[
\| K_{11} f_1\|_{L^2} \leq C_{11} \epsilon(T) \|f_1\|_{L^2_1}
\]
with $C_{11}$ independent of $T$. A similar calculation holds for other entries of $K$. 

Since $K$ is end-periodic, we write $K'$ for the corresponding operator on the folded manifold $X$. Consider the path of operators $Q_s=D+sK$, $s \in [0, 1]$. From Lemma \ref{fh5.3} we know that $Q_s$ is Fredholm if and only if the family of operators 
\[
Q'_{s, z}:=D' +sK'- \ln z \rho(\tau') :L^2_1(E'_1) \oplus L^2_1(E'_2)  \to  L^2(F'_1)\oplus L^2(F'_2), 
\]
is invertible for all $z$ of unit length, where 
\[
\rho(\tau') = 
\begin{pmatrix}
\rho(\tau'_1) & 0 \\
0 & \rho(\tau'_2)
\end{pmatrix}
\]
From \cite[Proposition 7.3]{LRS} we know that there exists $T_1$ such that for all $T>T_1$ one has $D' - \ln z \tau$ is invertible for all $z$ of unit length $|z|=1$. Note that the operator norm $\| K'\| \leq C_5 \epsilon(T)$ for some constant $C_5$ independent of $T$. Compactness of the unit circle guarantees us that there is $T_o >0$ such that for all $T > T_o$ the family of operators $Q'_{s, z}$ are all invertible for $s \in [0, 1]$, and $|z|=1$. Thus $Q_s$ is a path of Fredholm operators. In particular 
\[
\ind D= \ind Q_0 =\ind Q_1 = \ind \tilde{D}.
\]
\end{proof}

\begin{rem}
One can see that the proof goes through for other elliptic operators under the assumption (\ref{EPhypo}). The proof of Proposition 7.3 in \cite{LRS} is completely analytical, makes no use of anything special about Dirac operators. 
\end{rem}

\section{Surgery Formula}\label{SurFor}  

\subsection{Surgery along Torus}\label{SAT}
In this section we will prove a surgery formula for the Casson-Seiberg-Witten invariant for a homology $S^1 \times S^3$ with respect to surgeries along an embedded torus using the techniques developed in previous sections. We recall the set-up from Section \ref{setup}.

Let $(X, \s, g)$ be a smooth oriented Riemannian spin$^c$ $4$-manifold satisfying 
\[
H_*(X; \Z) \cong H_*(S^1 \times S^3; \Z),
\]
$\iota: T^2 \hookrightarrow X$ be an embedded torus with $\im \iota=\mathcal{T}$. Let $1_X \in H^1(X; \Z)$ be a fixed generator whose Poincaré dual is represented by an embedded hypersurface $Y_X \subset X$. We also require that $\mathcal
{T}$ intersects $Y_X$ transversely in a knot $K = \mathcal{T} \cap Y_X$. Then we decompose 
\[
X= M \cup N,
\]
where $N$ is a regular neighborhood of $\mathcal{T}$ identified with $D^2 \times T^2$, $M$ the closure of the complement of $N$. We write $\mathcal{Y}_{\mathcal{T}}:=\partial M  = -\partial N$ which is a copy of $T^3$. A collar neighborhood of $\Y$ is identified with $(-1, 1) \times \Y$. We choose the metric $g$ on $X$ so that it restricts as 
\[
g|_{(-1, 1) \times Y} = dt^2 + h,
\]
where $h$ is a fixed flat metric on $\Y$. Since $N=D^2 \times T^2$, we can choose the metric $g$ so that $g|_N$ has nonnegative and somewhere nonvanishing scalar curvature. As what we have been doing for the former sections, we can stretch the collar of $\Y$ to get $(X_T, g)$ for any $T >0$. This is equivalent to deforming the metric to get $(X, g_T)$. Recall that we write 
\[
M_{\infty} = M \cup [0, \infty) \times \Y, N_{\infty}=(-\infty, 0] \times \Y, \text{ and } X_{\infty}=M_{\infty} \cup N_{\infty}.
\]
We fix $T_o>0$ so that, for any $T \geq T_o$, the results of  Theorem \ref{sf4.14} and Theorem \ref{EP} hold. Note that we only assumed the metric is of product form in a neighborhood of $\Y \subset X$, but not a neighborhood of the generating hypersurface $Y_X \subset X$. Moreover the entire stretching process is only applied on $\Y$. Once we keep this in mind, we can fix $T$, and take the metric $g$ at the beginning to be $g_T$.  Recall that one can perform $(p, q)$-surgery along $\mathcal{T}$ to get $X_{p, q}$. We write $X_q=X_{1, q}$ for $q \neq 0$, and $X_{0, 1} = X_0$. Their homology is given by 
\[
H_*(X_q; \Z) =\left 
\{ \begin{array}{ll}
H_*(S^1 \times S^3; \Z) & \mbox{if $q \neq 0$} \\
H_*(T^2 \times S^2; \Z) & \mbox{if $q=0$}
\end{array}
\right.
\]

Recall that
\begin{equation}
\mathcal{SW}(X_0)=\sum_{\s_0 \in \Spinc(X_0)} \SW(X_0, \s_0),
\end{equation}
where for each $\s_0 \in \Spinc(X_0)$ $\SW(X_0, \s_0)$ is computed using the chamber specified by small perturbations.

\begin{thm}\label{SurT}
After fixing appropriate homology orientations, one has 
\[
\lambda_{SW}(X_1) - \lambda_{SW}(X) = \mathcal{SW}(X_0).
\]
\end{thm}

As a corollary of Theorem \ref{SurT} we can deduce the main result Theorem \ref{Main1}.

\begin{proof}[Proof of Theorem \ref{Main1}]
When $q =0$, this is the result in Theorem \ref{SurT}. For any $q \in \Z$, $X_{1, q}=M \cup_{\phi_{1, q}} D^2 \times T^2$. Let $T:= \subset \{0 \} \times T^2$ be the core of $D^2 \times T^2$. We denote by $\mathcal{T}_q \subset X_{1, q}$ the image of $T$ in $X_{1, q}$. Then with respect to the preferred framing of $\mathcal{T}_q$, the $(1, 1)$-surgery along $\mathcal{T}_q \subset X_q$ is $X_{q+1}$, the $(0, 1)$-surgery along $\mathcal{T}_q$ is $X_0$. Thus replacing $X$ by $X_q$ and $\mathcal{T}$ by $\mathcal{T}_q$ in Theorem \ref{SurT} gives us the result. 
\end{proof}

The proof of Theorem \ref{SurT} is divided into two parts: Proposition \ref{IM} and Proposition \ref{IC}. The first part is to use the neck-stretching and gluing argument to compare the counting of irreducible monopoles. The second part is to use the techniques of periodic spectral flow and periodic excision principle to compare the index correction terms from $\lambda_{SW}(X_1)$ and $\lambda_{SW}(X)$. There are some extra terms coming out from those two steps. Those terms will eventually cancel, thus giving us the desired result. The same kind of cancellantion principle takes place in Lim's proof for the surgery formula of Casson invariant defined using Seiberg-Witten theory in \cite{Lim1}.

Before diving into the proof, we would like to comment on the admissibility of the metric $g$ on $X$. As it turns out, when we apply Theorem \ref{sf4.14} and Theorem \ref{EP} in the proof the corresponding assumptions will concretize as the admissibility of the metric $g$. Following Remark \ref{read}, it suffices to assure that the $4$-manifold is spin cobordant to the empty set, and the spin Dirac operator $D(Y, h)$ is invertible. In the proof below we will encounter several different $Y$'s, which are $T^3, S^3$, and $S^1 \times S^2$. We put metrics on $S^3$ and $S^2 \times S^1$ with positive scalar curvature so that the invertibility of $D_B(Y, h)$ is satisfied. In this case the $4$-manifolds are taken to be products $S^1 \times Y$, which are clearly spin cobordant to the empty set due to vanishing of signature. When $Y=T^3$, Lemma \ref{2.5} tells us the spin Dirac operators $D_B(Y, h)$ have trivial kernel. In this case $X$ is an integral homology $S^1 \times S^3$, which is also spin cobordant to zero. Once the metric $g$ is chosen, the regularity of the pair $(g, \beta)$ in (\ref{1.6.1}) can be achieved by choosing a generic choice of the perturbation solely. Therefore we won't bother to repeat this point in the proof when applying our theorems.

\subsubsection{Counting Irreducible Monopoles}

Recall that in Section \ref{setup} we have chosen an oriented basis for $H_1(\Y; \Z)$ as $(\mu, \lambda, \gamma)$ called the meridian, longitude, and latitude respectively. Now we form
\[
X_1=M \cup_{\varphi_1} N \text{ and } X_0=M \cup_{\varphi_0} N, 
\]
where under the choice of basis above the diffeomorphisms $\phi_1$ and $\phi_0$ are given by the matrices 
\[
\varphi_1=
\begin{pmatrix}
1 & 0 & 0 \\
1 & 1 & 0 \\
0 & 0 & 1
\end{pmatrix} \text{ and }
\varphi_0=
\begin{pmatrix}
0 & 1 & 0 \\
-1 & 0 & 0 \\
0 & 0 & 1
\end{pmatrix}. 
\]
We put a metric $g_1$ on $X_1$ so that $g_1|_M=g|_M$, $g_1|_N$ is extended from $\varphi_1^* h$ with nonnegative and somewhere positive scalar curvature. We can similarly get a metric $g_0$ on $X_0$. Note that $h$, $h_1:=\phi_1^*h$, and $h_0:=\phi_0^*h$ are all flat on $\Y \cong T^3$. We choose a generic small perturbation $\beta$ with $\supp \beta \subset M$ so that all three pairs $(g, \beta)$, $(g_1, \beta)$ and $(g_0, \beta)$ are regular in the sense of Definition \ref{R1}. Now we deal with the spin$^c$ structure on each piece. As before we write $\s$ for the spin$^c$ structure on $X$ induced from a fixed spin structure. We write $\s_M=\s|_M$, $\s_N=\s|_N$. Let's write 
\[
\M_1:=\M^*_{g, \beta}(M_{\infty}, \s_M) \text{ and } \M_2:=\M^{\Red}_g(N_{\infty}, \s_N).
\]
Since we already assumed the stretching parameter $T$ is large, the standard gluing theorem  (see, for example, \cite{MM}, \cite{N} etc. and Appendix \ref{TGT} for the modification in the case of homology $S^1 \times S^3$) implies that  that the counting of irreducible monopoles on $X$ and $X_1$ are given by  
\begin{equation}\label{count1}
\begin{split}
& \# \M^*_{g, \beta}(X, \s) =\# \bar{\partial}_+ (\M_1) \cap \bar{\partial}_-(\M_2)\\
& \# \M^*_{g_1, \beta}(X_1, \s_1) =\# \bar{\partial}_+ (\M_1) \cap \bar{\partial}_{-, 1}(\M_2)\\
\end{split}
\end{equation}
The core of $N$ embeds in $X_0$ as a torus of self-intersection $0$, the adjunction inequality implies that the spin$^c$ structures $\s_0 \in \Spinc(X_0)$ consist of those restricting to $M$ and $N$ as $\s_M$ and $\s_N$ respectively. The standard gluing theorem implies that the counting of irreducibles over $X_0$ with respect to all spin$^c$ structures are given by 
\begin{equation}\label{count2}
\sum_{\s_0 \in \Spinc(X_0)} \# \M^*_{g_0, \beta}(X_0, \s_0) =\# \bar{\partial}_+ (\M_1) \cap  \bar{\partial}_{-, 0}(\M_2).
\end{equation}
Note that the intersections above are transverse due to Proposition \ref{2.12}. Now we proceed to put  coordinates on $\chi(\Y)$ and represent the above pieces as oriented manifolds inside.

Let's fix the spin structures $\s$, $\s_1$ on $X$, $X_1$ respectively. We write $A_{\s}$ and $A_{\s_1}$ for the corresponding spin connections. Now think of a spin structure as a homotopy class of a trivialization of tangent bundle over $1$-skeleton that extends to $2$-skeleton. Since the gluing maps preserve the images of the generator $1_X \in H^1(X; \Z)$ under the restriction maps to $M$ and $N$, the spin structures $\s$ and $\s_1$ agree on the $1$-cell in the $1$-skeleton dual to the images of $1_X$. Moreover $\s|_N$ differs from $\s_1|_N$ by a half twist on the other $1$-cell in $N$. We identify the image of $M$ in all surgered manifolds together with its boundary $\Y=\partial M$. Let $A_{\s}|_{\Y}=B_{\s}$ be the spin connection given by that on $X$. The coordinates on $\chi(\Y)$ for any spin$^c$ flat connection $\B$ is given by 
\[
\B \longmapsto ({1 \over 2\pi i} \int_{\mu} 2b, {1 \over 2\pi i} \int_{\lambda} 2b, {1 \over 2\pi i} \int_{\gamma} 2b),
\]
where $2b=\B - \B_{\s} \in \Omega^1(\Y; i\R)$. We write $(x, y, z)$ for the coordinate map. Note that the gauge action is even: $u \cdot \B=\B- 2u^{-1}du$ and $[u^{-1}du] \in 2\pi i H^1(\Y;\Z)$, thus a fundamental domain of $\chi(\Y)$ is identified with the cube
\begin{equation}
C(\Y):=\{ (x, y, z) : x, y, z \in [-1, 1]\}. 
\end{equation}
Thus $\chi(\Y)=C(\Y) / \sim$ with opposite faces of the cube identified. We write $[x, y, z] \in \chi(\Y)$ for the class represented by $(x, y, z) \in C(\Y)$. Note that the singular connection $\Theta$ is characterized as the spin connection with respect to the product spin structure on $S^1$ which does not extend over $D^2$, thus 
\[
[\Theta^{\text{t}}]=[1, 1, 1].
\]
Since $N=D^2 \times T^2$, we conclude that the restriction of any flat connection on $N$ on $\Y$ has first coordinate $x=0$. Thus the image $\bar{\partial}_{-}(\M_2) \subset \chi(\Y)$ is given by the plane
\[
P=\{ x=0\} \cap C(\Y)
\]
From the description for spin structures above, we know that 
\[
[\B_{\s_1}]=[1, 0, 0]=[-1, 0, 0].
\]
With respect to the dual basis of $(\mu, \lambda, \gamma)$ the maps $\varphi^*_1$ and $\varphi^*_0$ on $H^*(\Y; \R)$ are given by 
\[
\varphi^*_1=
\begin{pmatrix}
1 & 1 & 0 \\
0 & 1 & 0 \\
0 & 0 & 1
\end{pmatrix} \text{ and }
\varphi^*_0=
\begin{pmatrix}
0 & -1 & 0 \\
1 & 0 & 0 \\
0 & 0 & 1
\end{pmatrix}.
\]
Thus the images $\bar{\partial}_{-, 1}(\M_2)$ and $\bar{\partial}_{-, 0}(\M_2)$ are given respectively as 
\[
P_1=\{x+y=\pm 1\} \cap C(\Y)  \text{ and } P_0=\{ y=0\} \cap C(\Y). 
\]
Let $A_{\beta}$ be a spin$^c$ connection on $M_{\infty}$ with $F^+_{\A_{\beta}}=2d^+\beta$. Let $[B_{\beta}]=\bar{\partial}^+([A_{\beta}])$ with coordinate $c_{\beta}=(x_{\beta}, y_{\beta}, z_{\beta})$. Then from Lemma \ref{2.4} the image of the reducible locus $\M^{\Red}_{g, \beta}(M_{\infty}, \s)$ is 
\[
P_M= \{ [c_{\beta}]+[x, 0, z] \}=\{ [x, y_{\beta}, z] : x, z \in [- 1, 1]\}. 
\]
Note that when $\beta=0$, $P_M=P_0$. We write the closure of $\bar{\partial}^{+}(\M_1)$ as $R$ consisting of a finite number of curves transverse to $P \cup P_0 \cup P_1$ with end points in $P_M$ or $\{ \theta\}$. Without loss of generality we may assume $y_{\beta} >0$. We let $\B_s$ be the path of connections on $P_M$ given by 
\[
(x(\B_s), y(\B_s), z(\B_s)) = (s, y_{\beta}, 0), \; \; s \in [0, 1].
\]
Let $\A_s$ be the path in $\M^{\Red}_{g, \beta}(M_{\infty}, \s)$ specified by $\bar{\partial}_+([\A_s])=[\B_s]$. 
\begin{prop}\label{IM}
Under the notations as above, we have
\begin{equation}
\# \M^*_{g_1, \beta}(X_1, \s_1) -\# \M^*_{g, \beta}(X, \s) =\sum_{\s_0} \# \M^*_{g_0, \beta}(X_0, \s_0)+ \Sf D^+_{A_s} (M_{\infty}, \beta).
\end{equation}
\end{prop}

\begin{proof}
From the description above we know the count of irreducible monopoles on $X_1, X_0$, and $X$ are given respectively by the signed count of the intersections
\[
R \cap P_1, \; \; R \cap P_0, \text{ and } R \cap P. 
\]
The union of two planes $P_0 \cup P$ divides $C(\Y)$ into four quadrants $Q_i$ as the four quadrants in the $(x, y)$-plane multiplied by the $z$-axis. The two pieces of $P_1$ lie in $Q_1$ and $Q_3$ respectively as shown below, which we denoted by $P_{1, +}$ and $P_{1, -}$. We can similarly decompose 
\[
P=P_{-} \cup P_{+}, \; \; P_0=P_{0, -} \cup P_{0, +}, \;\; P_M=P_{M, -} \cup P_{M, +}.
\]
Here we write the part of $P_0$ and $P$ lying on the boundary of $Q_1$ as "$+$". The Kuranishi picture in Theorem \ref{KUR} implies that $R$ is transverse to $P_M$ as well. In particular the triple intersections $R \cap P_i \cap P_j = \emptyset$. Denote by $\Sigma=R \cap P_M \subset Q_1 \cup Q_2$ the ends of curves, which splits into $\Sigma=\Sigma_- \cup \Sigma_+$ with $\Sigma_- \subset Q_2$, $\Sigma_+ \subset Q_1$. Since $P_0 \cap Q_1$ lies below $P_{1, +}$, by choosing $\beta$ small one can assure that $\Sigma_+$ lies below $P_{1, +}$ in the first quadrant. One can deform the planes $P_{1, \pm}$ into the union $P \cup P_0$ continuously avoiding the singular point. The deforming process is easily seen in the picture, so we don't bother to write down a formula. The time-$s$ deformed plane is denoted by $P_1(s)$.  We need to keep track of the intersection 
\[
P_{1, +}(s) \cap P_M = \{ (1- \max(y_{\beta}, s), y_{\beta}, z) : z \in [-1, 1]\}, \;\; s \in [0, 1]. 
\]
Since away from $\theta$ $R \cap Q_3$ is a regular curve, i.e. any of its components is either an oriented circle or any oriented arc ends on $\partial Q_3$, and $\theta \notin P_{1, -}(s)$, thus in the third quadrant 
\[
\# R \cap P_{1, -} = \#R \cap P_{-} +  \# R \cap P_{0, -}.
\]
For the intersection in the first quadrant we write
\begin{equation}
I(s)= \# R \cap P_{1,+}(s) - \# R \cap P_+ - \# R \cap P_{0, +}. 
\end{equation}
Everytime $P_{1, +}(s)$ passes through a point in $\Sigma_+$, $I(s)$ will change by either $-1$ or $1$ according to whether the intersection $R \cap P_{1, +}(s-\epsilon)$ is positive or negative for small $\epsilon >0$. This assigns a sign to each point in $\Sigma_+$. Thus 
\[
I(0) - I(1) = I(0) = \# \Sigma_+. 
\]
Recall that the function $f: X \to S^1$ is chosen to represent the dual of the latitude $\gamma$, which restricts to $C(\Y)$ as the $z$-coordinate, thus the preimage $\bar{\partial}_+^{-1}(P_{M, +}) \subset \M_1$ is parametrized by the cylinder 
\begin{equation}
\{ A_s - \ln w \cdot df : s \in [0, 1], |w|=1\}.
\end{equation}
Note that from Theorem \ref{KUR} the points $A \in \bar{\partial}_{+}^{ -1}(\Sigma_+)$ are characterized by the property 
\begin{equation}
F^+_{\A}=2d^+\beta \text{ and } \ker D^+_A = \C, 
\end{equation}
which are exactly the spectral points for the path $D^+_{A_s}(M_{\infty}, \beta)$. Thus once the signs of each point in $\bar{\partial}_{+}^{ -1}(\Sigma_+)$ are identified, we can conclude that 
\[
\# \Sigma_+ = \Sf D^+_{A_s}(M_{\infty}, \beta). 
\]

\begin{figure}[h]
\centering
\begin{picture}(200,200)
\put(-100, -150){\includegraphics[width=1.00\textwidth]{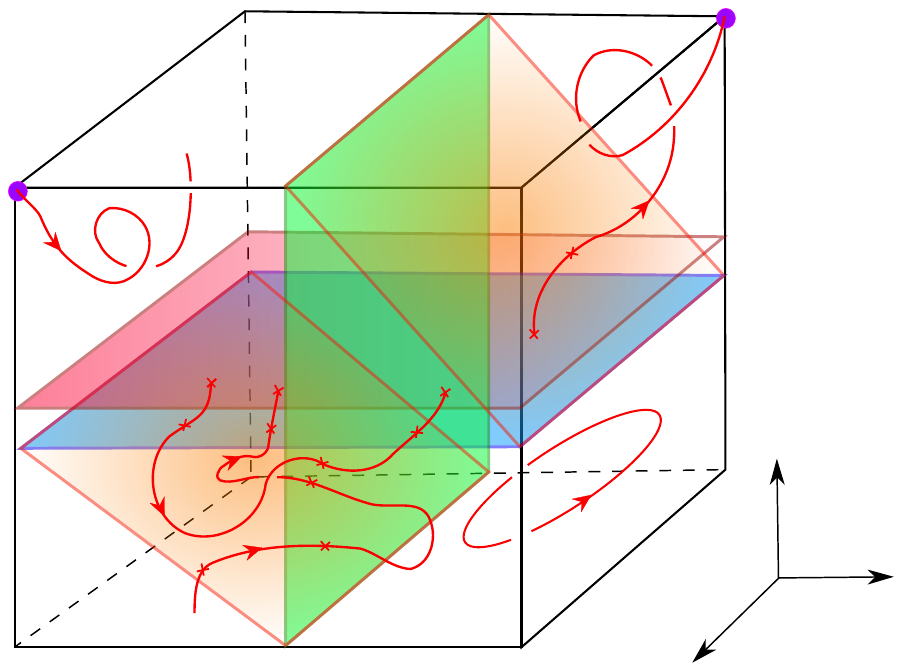}}
\put(-35,150){$\theta$}
\put(205,200){$\theta$}
\put(-40,85){$P_M$}
\put(-40,70){$P_0$}
\put(10, 20){$P_{1, -}$}
\put(130, 170){$P_{1,+}$}
\put(220, 60){$y$}
\put(240, 20){$x$}
\put(200, 0){$z$}
\put(60, -5){$P$}
\end{picture}
	\caption{Surgery Fundamental Cube}
\end{figure}

Let $(s_t, w_t)$ be a local spectral curve through a spectral point $(s_0, w_0)$ with $t \in (-\epsilon, \epsilon)$. We write $\ln w_t = u(t) + iv(t) \in \C$. Then the sign for counting periodic spectral flow at $A(0)=A_{s_0}- \ln w_0 \cdot df$ is the sign of $\dot{u}(0)$. On the other hand, inside $C(\Y)$ the plane $P$ is oriented by the ordered basis $(\partial_y, \partial_z)$. By requiring $\phi_1^*$ and $\phi_0^*$ be orientation-preserving, we conclude that $P_M$ is oriented by the ordered basis $(-\partial_x, \partial_z)$ and $P_{1, +}$ oriented by $(-\partial_x+ \partial y, \partial_z)$. Let $\bar{\partial}_+(A(0))=p(0)=(x(0), y_{\beta}, z(0)) \in \Sigma_+$ oriented by the boundary orientation of $R$, which is $-1$ if $R$ leaves $p(0)$ and $+1$ if $R$ enters $p(0)$. Both the orientations of $p(0)$ and $R$ are obtained from their preimage in $\M_1$ under the asymptotic map $\bar{\partial}_+$. The argument in \cite[Corollary 8.5]{MRS} asserts that the sign coincides with the sign of $\dot{u}(0)$. Thus as $P_{1,+}(s)$ passes through $p(0)$, the counting 
\[
\# R \cap P_{1, +}(s_0-\epsilon) - \# R \cap P_{1, +}(s_0+ \epsilon) =\sign \dot{u}(0).
\]
Thus we conclude that 
\begin{equation}
\begin{split}
\# \M^*_{g_1, \beta}(X_1, \s_1) & - \sum_{\s_0} \# \M^*_{g_0, \beta}(X_0, \s_0) - \# \M^*_{g, \beta}(X, \s) \\
&=\# R \cap P_1 -\# R \cap P_0 - \# R \cap P \\
&=(\# R \cap P_{1,-} -\# R \cap P_{0, -} - \# R \cap P_-) \\
&+(\# R \cap P_{1,+} -\# R \cap P_{0, +} - \# R \cap P_+) \\
&=I(0) = \#\Sigma_+=\Sf D^+_{A_s}(M_{\infty}, \beta). 
\end{split}
\end{equation}
\end{proof}

\subsubsection{Comparing Index Correction}
Now we complete the second half of the argument, which consists of a series of applications of the excision principle. Since $X$ and $X_1$ have the same homology as $S^1 \times S^3$, we need to compare the difference of their index correction terms. 

\begin{prop}\label{IC}
Let $D^+_{A_s}(M_{\infty}, \beta)$ be the path of Dirac operators as in Proposition \ref{IM}. Then 
\begin{equation}
\omega(X_1, g_1, \beta_1) - \omega(X, g, \beta) = \Sf D^+_{A_s}(M_{\infty}, \beta).
\end{equation}
\end{prop}

\begin{proof}
Recall we have chosen a hypersurface $Y_X \subset X$ representing Poincaré dual of the generator $1_X \in H^1(X; \Z)$ so that $Y_X \cap \mathcal{T}=K$ is an embedded knot in $Y_X$. The preferred framing of $\mathcal{T}$ induces a framing on $K \subset Y$. We take a spin $4$-manifold $(Z, \s)$ with boundary $(Y_X, \mathfrak{t})$ compatible with the spin structure on the periodic end $(W_+, \s)$. We can require that $K$ bounds a disk $D_K \subset Z$. For example one first attaches a $2$-handle $H_K$ with $0$-framing along $K$ on $I \times Y_X$, then close the resulting spin boundary with a spin $4$-manifold.

Let's form the spin manifold $(Z_+, \s)$, and denote by $A_{\s}$ the spin connection of the spin structure $\s$ on $Z_+$. Now we remove a small $4$-ball $D^4 \subset Z$ centered at the center of $D_K$. Let's consider two sets of excisable data
\[
(Z_+=D^4 \cup Z_+ \backslash D^4, D^+_{A_{\s}}) \text{ and } (\bar{D}^4_-=(-\infty, 0] \times S^3 \cup \bar{D}^4, D^+_{A_{\mathfrak{d}}}),
\]
where $\bar{D}^4$ is the oriention-reversed $4$-ball, and $\mathfrak{d}$ is the spin connection on $\bar{D}^4$. Writing $Z^c_+=(-\infty, 0] \times S^3 \cup Z_+ \backslash D^4$ with induced spin structure $\s^c$ we get 
\begin{equation}
\ind D^+_{A_{\s},\beta}(Z_+)+ \ind D^+_{A_{\mathfrak{d}}}(\overline{D}^4_-) = \ind D^+_{A_{\s^c}, \beta}(Z^c_+) + \ind D^+ (S^4).
\end{equation}
Due to positivity of scalar curvature on $S^4$ and $\bar{D}^4_-$ we conclude 
\begin{equation}\label{71}
\ind D^+_{A_{\s},\beta}(Z_+)= \ind D^+_{A_{\s^c}, \beta}(Z^c_+).
\end{equation}

Let $\pi: W_+ \to X$ be the covering projection. We write $D^c_K=D_K \backslash D^4$ for the punctured disk. Let $N_K$ be a tubular neighborhood of $(-\infty, 0] \times S^1 \cup D^c_K$, and $V_K=N_K \cup \pi^{-1}(N) \subset Z^c_+$. In this way $Z^c_+$ is decomposed into two pieces
\[
Z^c_+ := U_K \cup V_K.
\]
Note that the diffeomorphism $\phi_0 : \partial N \to \partial M$ lifts and extends to a diffeomorphism $\tilde{\phi}_0: \partial V_K \to \partial U_K$. Then we can form another spin$^c$ manifold 
\[
Z^c_{0, +}= U_K \cup_{\tilde{\phi}_0} V_K
\]
with one periodic end $W_{0, +}$ lifting $X_0$ and one cylindrical end $S'=(-\infty, 0] \times S^1 \times S^2$. We denote the spin$^c$ structure on $Z^c_{0, +}$ by $\s^c_0$ that extends the one induced from $\s_0$ on the periodic end $W_{0, +} \to X_0$. 
Now form another end-periodic manifold 
\[
\tilde{V}_K = -V_K \cup_{\tilde{\phi}_0} V_K,
\]
which inherits a spin$^c$ structure $\tilde{\s}$ induced from the spin structure on $V_K$. We denote by $\tilde{A}_{\s}$ the spin connection $(\tilde{V}_K, \tilde{\s})$. Then we apply the excision principle Theorem \ref{EP} over the two sets of periodic excisable data 
\[
(Z^c_+= U_K \cup V_K, D^+_{A_{\s}}) \text{ and } (\tilde{V}_K=-V_K \cup_{\tilde{\phi}_0} V_K, D^+_{\tilde{A}_{\s}})
\]
to get 
\begin{equation}
\ind D^+_{A_{\s^c}, \beta}(Z^c_+) + \ind D^+_{\tilde{A}_{\s}}(\tilde{V}_K) = \ind D^+_{A_{\s^c_0}, \beta}(Z^c_{0, +}) + \ind D^+_{\tilde{A}'_{\s}}(\tilde{V}'_K),
\end{equation}
where $\tilde{A}'_{\s}$ the induced connection on $\tilde{V}'_K= -V_{K} \cup_{\id} V_K$. Note that both $\tilde{V}_K$ and $\tilde{V}'_K$ admit orientation-reversing diffeomorphisms, thus the indices of corresponding Dirac operators vanish. Thus we get 
\begin{equation}\label{72}
\ind D^+_{A_{\s^c}, \beta}(Z^c_+) = \ind D^+_{A_{\s^c_0}, \beta}(Z^c_{0, +}). 
\end{equation}

Let's write  
\[
S=S^1 \times S^3=S^1 \times D^3 \cup \overline{S^1 \times D^3}
\]
with the unique spin$^c$ structure induced from either of the two spin structures. Let $Z^c_{0,+}$ be the Applying the excision principle on two sets of excisable data
\[
(Z^c_{0, +}=S' \cup Z^c_{0, +} \backslash S', D^+_{A_{\s^c_0}} ) \text{ and } (S= S^1 \times D^3 \cup \overline{S^1 \times D^3}, D^+_{A_{\mathfrak{d}_S}})
\]
gives us 
\begin{equation}
\ind D^+_{A_{\s^c_0}, \beta}(Z^c_{0, +}) +\ind D^+(S) = \ind D^+_{A_{\s_0}, \beta}(Z_{0, +})+ \ind D^+_{A_{\mathfrak{d}_S}}(S_-),
\end{equation}
where $Z_{0, +}=S^1 \times D^3 \cup Z^c_{0, +}\backslash S'$, $\s_0$ is induced from the spin structure that extends $\s^c_0|_{Z^c_{0, +}\backslash S'}$, and $S_-=(-\infty, 0] \times S^1 \times S^2 \cup \overline{S^1 \times D^3}$. Again the positivity of scalar curvature on $S$ and $S_-$ implies that 
\begin{equation}\label{73}
\ind D^+_{A_{\s^c_0}, \beta}(Z^c_{0, +})= \ind D^+_{A_{\s_0}, \beta}(Z_{0, +}). 
\end{equation}
It follows from (\ref{71}), (\ref{72}), and (\ref{73}) that 
\begin{equation}\label{74}
\ind D^+_{A_{\s}, \beta}(Z_+)= \ind D^+_{A_{\s_0}, \beta}(Z_{0, +}).
\end{equation}

Note that performing $1$-surgery on $Y_X$ along $K$ gives us the hypersurface representing Poincaré dual of the generator $1_{X_1} \in H^1(X_1 ;\Z)$. Let $K_1 \subset Y_{X, 1}$ be the image of a linking circle of $K$ under the gluing map $\phi_1$. To distinguish with $N$, we denote by $N_1$ the image of $N$ in $X_1$ after the gluing process. The complements of $N \subset X$ and $N_1 \subset X_1$ are still identified as $M$. Then we run the argument in the paragraphs above again to $(X_1, \s_1)$ instead of $(X, \s)$, i.e. take a spin manifold $Z^1$ with boundary $Y_{X, 1}$ inside which there exists a disk bounded by $K_1$, then we form $Z^1_+=Z^1 \cup W_{1, +}$, $Z^1_{0, +}=Z^1_0 \cup W_{0, +}$ etc. together with the spin structures to apply the excision argument to get 
\begin{equation}
\ind D^+_{A_{\s_1}, \beta}(Z^1_+) = \ind D^+_{A_{\s_1, 0}, \beta} (Z^1_{0, +}). 
\end{equation}

Now we apply the usual excision principle to the two sets of excisable date 
\[
(Z_{0, +} =  Z_0 \cup W_{0, +}, D^+_{A_{\s_0}, \beta}(Z_{0, +})) \text{ and } (\tilde{Z}^1_0 = Z^1_0 \cup -Z^1_0, D^+_{\tilde{A}_{\s, 1}, \beta}(\tilde{Z}^1_0)), 
\]
where $\tilde{A}_{\s, 1}$ is a connection on $\tilde{Z}^1_0$ obtained by doubling the connection on $Z^1_0$. The resulting formula is 
\begin{equation}
 \ind D^+_{A_{\s_0}, \beta}(Z_{0, +}) + \ind D^+_{\tilde{A}_{\s, 1}, \beta}(\tilde{Z}^1_0) = \ind D^+_{\tilde{A}'_{\s,1}, \beta}(\tilde{Z}_0^{01})+ \ind D^+_{A_{\s, 1}, \beta}(Z^1_{0, +}).
\end{equation}
where $\tilde{Z}^{01}_0=Z_0 \cup -Z^1_0$, and $\tilde{A}'_{\s,1}$ is the corresponding glued connection. Since $\tilde{Z}^1_0$ admits an orientation reversing diffeomorphism and $\beta$ is small, we conclude that 
\begin{equation}
\ind D^+_{\tilde{A}_{\s, 1}, \beta}(\tilde{Z}^{1}_0) =0.
\end{equation}
By the Atiyah-Singer index theorem and another application of the excision principle (for example \cite[Proposition 3.2]{MRS}) we know that 
\begin{equation}
\ind D^+_{\tilde{A}’_{\s, 1}, \beta}(\tilde{Z}^{01}_0)) = {1 \over 8}(\sigma(Z) - \sigma(Z^1)).
\end{equation}
\cite[Theorem 7.3]{MRS} implies that 
\begin{equation}
\ind D^+_{A_{\s_1, 0}, \beta} (Z^1_{0, +}) - \ind D^+_{A_{\s, 1}, \beta}(Z^1_{0, +})= \Sf D^+_{A^0_s, \beta}(X_0),
\end{equation}
where $A^0_s$ is a path of flat connections from $A_{\s, 1}|_{X_0}$ to $A_{\s_1, 0}|_{X_0}$. By construction $A_{\s, 1}|_M$ comes from the spin connection on $X$, $A_{\s_1, 0}|_M$ comes from the spin connection on $X_1$. Since they are both flat and cylindrical in the neighborhood of $\Y=\partial M \subset X_0$, their restrictions extend to flat connections over $M_{\infty}$. Moreover they are the two end points of the path $A_s$ as in Proposition \ref{IM}. Thus we may take $A^0_s|_M=A_s$. Since $N = D^2 \times T^2$ admits metric of nonnegative and somewhere positive scalar curvature, Theorem \ref{4.14} tells us that 
\begin{equation}\label{75}
\Sf D^+_{A^0_s, \beta}(X_0) = \Sf D^+_{A_s}(M_{\infty}, \beta). 
\end{equation}
Combining (\ref{74}) and (\ref{75}) we conclude that 
\begin{equation}
(\ind D^+_{A_{\s_1}, \beta}(Z^1_+) + {1 \over 8}\sigma(Z^1)) - (\ind D^+_{A_{\s}, \beta}(Z_+) +{1 \over 8} \sigma(Z)) = \Sf D^+_{A_s}(M_{\infty}, \beta).
\end{equation}
\end{proof}

\section{Applications}\label{app}
In this section we prove Proposition \ref{lmp} and Corollary \ref{6bc} in the introduction. Recall that $K \subset Y$ is a knot in an integral homology sphere, and $Y^n$ is the $n$-fold branched cover of $Y$ along $K$. 

\begin{lem}\label{letk}
Let $K^n$ be the preimage of $K$ in $Y^n$ under the covering projection $Y^n \to Y$. Then the restricted covering transformation $\tau_n: Y^n \backslash \nu(K^n) \to Y^n \backslash \nu(K^n)$ extends to a free element in the mapping class group of order $n$
\[
\tau'_n: Y^n_1(K^n) \longrightarrow Y^n_1(K^n),
\]
where $Y^n_1(K^n)$ is obtained by performing $1$-surgery along $K^n$. Moreover the quotient of $Y^n_1(K^n)$ under $\tau'_n$ is identified with $Y_n(K)$ the result of performing $n$-sugery of $Y$ along $K$. 
\end{lem}

\begin{proof}
Let's parametrize the boundary $-\partial Y^n \backslash \nu(K^n)=\partial \nu(K^n)$ with $(e^{i \theta}, e^{i\varphi})$, $\theta, \varphi \in [0, 2\pi)$, so that the covering transformation $\tau$ restricted on the boundary is given by 
\[
\tau_n: (e^{i \theta}, e^{i\varphi}) \longmapsto (e^{i (\theta +{ 2\pi \over n})}, e^{i\varphi})
\]
Denote by $\phi_1: \partial D^2 \times S^1 \to \partial Y^n \backslash \nu(K^n)$ for the map defining $1$-surgery along $K^n$. Then 
\[
\phi^{-1}_1 \circ \tau_n  \circ \phi_1: (e^{i \theta}, e^{i\varphi}) \longmapsto (e^{i (\theta +{ 2\pi \over n})}, e^{i(\varphi-{2\pi \over n})}),
\]
which extends as a map $\tau'_n$ on $Y^n_1(K^n)=Y^n \backslash \nu(K^n) \cup_{\phi_1} D^2 \times S^1$ so that 
\[
\tau'_n|_{D^2 \times S^1}: (re^{i \theta}, e^{i\varphi}) \longmapsto (re^{i (\theta +{ 2\pi \over n})}, e^{i(\varphi-{2\pi \over n})}), \; \; r\in [0, 1].
\]
Thus $\tau'_n$ is free of order $n$. 

To prove the second statement, we identify the quotient of $D^2 \times S^1$ under $\tau_n'$ with $D^2 \times S^1$ by 
\[
[re^{i\theta}, e^{i \varphi}] \longmapsto (re^{i(\theta + \varphi)}, e^{in\varphi}),
\]
and the quotient of $\partial Y^n \backslash \nu(K^n)$ under $\tau_n$ with $S^1 \times S^1$ by 
\[
[e^{i\theta}, e^{i \varphi}] \longmapsto (e^{in\theta}, e^{i\varphi}). 
\]
Then $\partial D^2 \times S^1 / \sim_{\tau'_n}$ is attached to $\partial Y^n\backslash \nu(K^n)$ by the map
\[
(e^{i\theta}, e^{i \varphi}) \longmapsto (e^{in \theta}, e^{i (\varphi+\theta)}),
\]
which is the gluing map for performing $n$-surgery. 
\end{proof}

\begin{proof}[Proof of Proposition \ref{lmp}]
Let $X^n$ be the mapping torus of $Y^n$ under the map $\tau_n$. Denote by $\mathcal{T}$ the torus embedded in $X^n$ given by the knot $K^n$ fixed by $\tau_n$. Then performing $(1,1)$-surgery of $X^n$ along $\mathcal{T}$ gives us the mapping torus $X'$ of $Y^n_1(K^n)$ under the free self-diffeomorphism $\tau'_n$. By the computation in \cite[Section 7]{RS1}, we know that 
\begin{equation}\label{fmt}
\lambda_{SW}(X')=-n \lambda(Y) - {1 \over 8}\sum_{m=0}^{n-1} \sign^{m/n}(K)- {1 \over 2}\Delta''_{K \subset Y}(1). 
\end{equation}
We remark that in \cite{RS1}, (\ref{fmt}) is originally derived under the assumption that $Y^n_1(K^n)$ is a rational homology sphere. However in the derivation, they only used the fact that the quotient $Y_n(K)$ of $Y^n_1(K^n)$ under $\tau_n'$ is a homology lens space. Thus the surgery formula (\ref{LAT}) tells us that 
\[
\lambda(X^n)+\mathcal{SW}(X^n_{0,1}) =- n \lambda(Y) - {1 \over 8}\sum_{m=0}^{n-1} \sign^{m/n}(K)- {1 \over 2}\Delta''_{K \subset Y}(1),
\]
where $X^n_{0,1}$ is obtained by performing $(0, 1)$-surgery of $X^n$ along $\mathcal{T}$. The same argument as in Lemma \ref{letk} shows that $X^n_{0,1}$ is the mapping torus 
\[
[0, 1] \times Y^n_0(K) / (0, \tau^0_n(y) ) \sim (1, y)
\]
of $Y^n_0(K^n)$ under a free self-diffeomorphism $\tau^0_n$ of order $n$ induced by $\tau_n$. Moreover there is a free circle action on $X^n_{0, 1}$ given by 
\[
e^{i2\pi s} \cdot [t, y] = [t+ns, y], \; \; s \in [0, 1],
\]
with the quotient identified with $Y_0(K)$. We claim the circle bundle $X^n_{0, 1} \to Y_0(K)$ is trivial. Indeed consider the Gysin sequence associated this bundle:
\[
H^0(Y_0(K);\Z) \xrightarrow{\smile e} H^2(Y_0(K); \Z) \xrightarrow{\pi^*} H^2(X^n_{0, 1}; \Z) \rightarrow H^1(Y_0(K); \Z),
\]
where $e \in H^2(Y_0(K) ;\Z)$ is the Euler class, $\pi: X^n_{0, 1} \to Y_0(K)$ is the projection map. This sequences reads as 
\[
\Z \longrightarrow \Z \longrightarrow \Z \oplus \Z \longrightarrow \Z. 
\]
Exactness forces the first map to be zero, thus the Euler class to be zero. Now $X^n_{0,1}$ is the product $S^1 \times Y_0(K)$. According the result of Meng-Taubes \cite{MT} and Baldridge \cite{B} that 
\[
\mathcal{SW}(S^1 \times Y_0(K)) = - \mathcal{SW}(Y_0(K)) = -{1 \over 2}\Delta''_{K \subset Y}(1),
\]
we conclude that 
\[
\lambda(X^n) = -n \lambda(Y) - {1 \over 8}\sum_{m=0}^{n-1} \sign^{m/n}(K).
\]
\end{proof}

\begin{proof}[Proof of Corollary \ref{6bc}]
Let $K$ be either of the trefoil or the figure-eight knot. 
(\ref{lmp1}) tells us that 
\begin{equation}\label{6bc2}
\lambda_{SW}(X^6(K)) = -{1 \over 8} \sum_{m=0}^5 \sign^{m/6}(K).
\end{equation}
The Tristram-Levine signature $\sign^{m/6}(K)$ is computed as follows. Let $V$ be the Seifert matrix of $K$, $B(t)=(1-t)V +(1-t^{-1})V^t$. Then 
\[
\sign^{m/6}(K) = \sigma (B(e^{i\pi m \over 3})), 
\]
i.e. the signature of the matrix $B(e^{i\pi m \over 3})$. The Seifert matrices of the right-handed trefoil, left-handed trefoil, and figure-eight knot are given respectively by 
\[
V_1=
\begin{pmatrix}
-1 & 1 \\
0 & -1 
\end{pmatrix}, \;
V_2=
\begin{pmatrix}
1 & -1 \\
0 & 1 
\end{pmatrix}, \;
V_3=
\begin{pmatrix}
-1 & 0 \\
1 & 1 
\end{pmatrix}.
\]
Then the corresponding Tristram-Levine signature is given by 
\[
\sigma (B_1(e^{i\pi m \over 3})) =\left \{ \begin{array}{ll}
0 & \mbox{$m=0$} \\
-1 & \mbox{$m=1, 5$} \\
-2 & \mbox{$m=2, 3, 4$}
  	\end{array}
\right.
\]
\[
\sigma (B_2(e^{i\pi m \over 3})) =\left \{ \begin{array}{ll}
0 & \mbox{$m=0$} \\
1 & \mbox{$m=1,5$} \\
2 & \mbox{$m=2,3,4$}
  	\end{array}
\right.
\]
\[
\sigma (B_3(e^{i\pi m \over 3})) = 0, \; m=0, 1, 2, 3, 4, 5.
\]
Then we get the desired result from (\ref{6bc2}).
\end{proof}

\appendix
\section{The Gluing Theorem}\label{TGT}

In the proof of Proposition \ref{IM} we made use of the gluing theorem to identify the moduli space $\# \M^*_{g, \beta}(X, \s)$ of irreducible monopoles on $(X, \s)$ with the fiber product of the  moduli spaces $\M^*_{g, \beta}(M_{\infty}, \s_M)$ and $\M^{\Red}_g(N_{\infty}, \s_N)$, which gives us the count (\ref{count1}) and (\ref{count2}). Two issues might arise in our set-up, which are absent in the standard gluing results (e.g. \cite{KM1}, \cite{MM}, \cite{T2} etc.). One is that as we are running the neck stretching argument over an integral homology $S^1 \times S^3$ there might exist a sequence of irreducible monopoles converging to a reducible one in the limit. The other issue is that over $T^3$ there is a singular point $\theta$ in the critical sets of the Chern-Simons-Dirac functional. Despite the intersection of images of the asymptotic maps on both sides $M_{\infty}$ and $N_{\infty}$ misses $\theta$, we still need to make sure there are no broken flowlines which could potentially flow to the singular point with no energy lost so that we cannnot exlude simply by an energy argument. Since the local gluing at a Morse-Bott critical points is well documented in the literature (for instance see \cite{MM} for the Yang-Mills case and \cite[Chapter 2.5]{Lin} for the Seiberg-Witten case), the identification of $\# \M^*_{g, \beta}(X, \s)$ with the fiber product is standard once those two issues are resolved. The purpose of this appendix is to supply such an argument. 

Let $X=M \cup N$ be a decomposition of an integral homology $S^1 \times S^3$  with $\partial M=-\partial N=Y$. We assume $N=D^2 \times T^2$ is given by a tubular neighborhood of an embedded torus as in the surgery construction, thus $Y=T^3$. Let $h$ be a flat metric on $Y$, and $g|_N$ a metric with nonnegative (positive somewhere in its interior) scalar curvature. Recall the neck-stretching set-up in Section \ref{nssu}, we consider the following notion of convergence for monopoles.
\begin{dfn} \label{3.1}
Let $\{T_n\}$ be an increasing sequence of positive numbers with $T_n \to \infty$. We say a sequence of monopoles $\{[\Gamma_n]\} $ in $\M(X_{T_n})$ converges to $([\Gamma_o], [\Gamma'_o]) \in \M(M_{\infty}, [\tilde{\mathfrak{b}}]) \times \M([\tilde{\mathfrak{b}}], N_{\infty})$ if the following two conditions holds.
\begin{enumerate}
\item[\upshape (i)] There exist gauge transformations $u_n: X_{T_n} \to S^1$ such that 
\[
u_n \cdot \Gamma_n|_{M_{T_n}} \xrightarrow{L^2_{k, loc}} \Gamma_o, \; u_n \cdot \Gamma_n|_{N_{T_n}} \xrightarrow{L^2_{k, loc}} \Gamma'_o. 
\]
\item[\upshape (ii)] For any $\epsilon >0$, there exist $T_o>0$ and $N>0$ such that for all $n>N$, one has $T_n >T_o$. Moreover one can find gauge transformations $u_n: I_{T_n-T_o}  \to S^1$ satisfying 
\[
\| u_n \cdot \Gamma_n|_{I_{T_n-T_o}} -\Gamma_{\mathfrak{b}} \|_{L^2_{k}(I_{T_n-T_o})} < \epsilon, 
\]
where $\Gamma_{\mathfrak{b}}$ is the constant trajectory on $I_{T_n-T_o}$. 
\end{enumerate}
\end{dfn}

The neck-stretching theorem we are going to prove in this section states as follows.

\begin{thm} \label{G1}
With the notations above and given the following two assumptions: 
\begin{enumerate}
\item[\upshape (i)] The metric $g$ on $X$ is admissible in the sense of Definition \ref{dfad}. 
\item[\upshape (ii)] The singular points $\theta \notin \bar{\partial}_+(\M^*(M_{\infty})) \cap \bar{\partial}_-(\M^{\Red}(N_{\infty}))$. 
\end{enumerate}
Then for any sequence of irreducible monopoles $[\Gamma_n] \in \M^*(X_{T_n})$ there exists a monopole $\mathfrak{b}$ with $[\mathfrak{b}] \neq \theta$, and $([\Gamma_o], [\Gamma'_o]) \in \M^*(M_{\infty}, [\tilde{\mathfrak{b}}]) \times \M^{\Red}([\tilde{\mathfrak{b}}], N_{\infty})$ such that, after possibly passing to a subsequence, $[\Gamma_n]$ converges to $([\Gamma_o], [\Gamma'_o])$ in the sense of Definition \ref{3.1}.
\end{thm}

We recall some well-known compactness results in Seiberg-Witten theory (see \cite[Theorem 5.1.1]{KM1}).

\begin{lem} \label{3.3}
Let $(M, \s, g)$ be a Riemannian spin$^c$ manifold with compatible boundary $(Y, \mathfrak{t}, h)$. Then the following compactness properties hold. 
\begin{enumerate}
\item[\upshape (i)] Let $\Gamma_n \in \mathcal{C}_k(M, \s)$ be a sequence of monopoles satisfying $\mathfrak{F}_{\beta}(\Gamma_n)=0$ with a uniform bound on energy $\mathcal{E}_{\beta}(\Gamma_n) \leq C$. Then there exist gauge transformations $u_n: M \to S^1$ and a smooth monopole $\Gamma_o \in \mathcal{C}(M, \s)$ such that, after passing to a subsequence, $u_n \cdot \Gamma_n$ converges to $\Gamma_o$ in $L^2_k(M')$ for any interior domain $M' \Subset M$. Moreover $\mathcal{E}_{\beta}(\Gamma_o) \leq C$. 
\item[\upshape (ii)] Let $\{T_n\}$ be an increasing sequence of positive numbers with $T_n \to \infty$, and $\Gamma_n \in \mathcal{C}_k(M_{T_n}, \s)$ a sequence of monopoles satisfying $\mathfrak{F}_{\beta_n}(\Gamma_n)=0$, where $\supp \beta_n \in M$, $\beta_n \to \beta$ in $L^2_k(M)$. Suppose there is a uniform bound $\mathcal{E}_{\beta_n}(\Gamma_n) \leq C$. Then there exist gauge transformations $u_n: M_{T_n} \to S^1$ and a smooth monopole $\Gamma_o \in \mathcal{C}(M_{\infty}, \s)$ such that, after passing to a subsequence, $u_n \cdot \Gamma_n$ converges to $\Gamma_o$ in $L^2_{k, loc}(M_{\infty})$. Moreover $\mathcal{E}_{\beta}(\Gamma_o) \leq C$.
\end{enumerate}
\end{lem}

Since the spin$^c$ structure  $\mathfrak{t}_0$ is torsion, the Chern-Simons-Dirac functional descends to a real valued function on the quotient configuration space $\mathcal{L}: \mathcal{B}(Y) \to \R$ with a connected critical submanifold $\chi(Y)$. By making a good choice of the reference flat connection $B_0$, we normalize the functional so that $\mathcal{L}([\mathfrak{b}])=0$ for all $\mathfrak{b} \in \chi(Y)$. For each constant flowline $\Gamma_{\mathfrak{b}} \in \mathcal{C}_k([-1, 1] \times Y)$ we take a gauge invariant neighborhood  $U_{\Gamma_{\mathfrak{b}}}$ of $\Gamma_{\mathfrak{b}}$ such that one can find a contractible open neighborhood $V_{[\mathfrak{b}]}$ of $[\mathfrak{b}]$ in $\mathcal{B}_l(Y)$ for each $[\mathfrak{b}] \in \chi(Y)$ so that for any configuration $\Gamma \in U_{\Gamma_{\mathfrak{b}}}$ one has $[\gamma(t)] \in V_{[\mathfrak{b}]}$, where $\gamma(t)=\Gamma|_{\{t \} \times Y}$. Note that $U_{\Gamma_{\mathfrak{b}}}$ is gauge invariant, thus only depends on the class $[\mathfrak{b}]$. We write $U_{\chi(Y)} =\bigcup U_{\Gamma_{\mathfrak{b}}}$ for a neighborhood of the critical manifold. We note a useful result for later use:

\begin{lem}\label{3.4} (\cite[Lemma 16.2.2]{KM1})
Let $U_{\Gamma_{\mathfrak{b}}}$ be a family of neighborhoods as above, $C$ any fixed constant. Then there exists $\epsilon_o >0$ such that for any monopole $\Gamma \in \mathcal{C}_k([-1,1] \times Y)$ satisfying 
\[
\mathcal{E}(\Gamma) \leq C \text{  and  } \mathcal{E}(\Gamma|_{[-{1 \over 2}, {1 \over 2}] \times Y} )\leq \epsilon_o,
\]
one has $\Gamma \in U_{\Gamma_{\mathfrak{b}}}$ for some $[\mathfrak{b}] \in \chi(Y)$.
\end{lem}

\begin{lem} \label{3.5}
The unperturbed finite energy moduli space $\M_{g}( (-\infty, \infty) \times Y, \s)$ consists of constant flowlines $[\Gamma_{\mathfrak{b}}]$, $[\mathfrak{b}] \in \chi(Y)$. 
\end{lem}

\begin{proof}
Let $[\Gamma] \in \M_g((-\infty, \infty) \times Y, \s)$ be a monopole. There are asymptotic maps $\partial_+ :\M_g((-\infty, \infty) \times Y, \s) \to \chi(Y)$ and $\partial_-: \M_g((-\infty, \infty) \times Y, \s) \to \chi(Y)$ corresponding to the positive and negative ends respectively (in this case $\chi_M(Y)=\chi(Y)$ since every gauge transformation on $Y$ extends to one on $(-\infty , \infty) \times Y$). We write $[\mathfrak{b}_{\pm}]=\partial_{\pm} [\Gamma]$. Since the critical manifold $\chi(Y)$ is connected, we know $\mathcal{L}([\mathfrak{b}_-])=\mathcal{L}([\mathfrak{b}_+])$. Thus 
\[
\mathcal{E}(\Gamma)=2(\mathcal{L}(\mathfrak{b}_-)-\mathcal{L}(\mathfrak{b}_+))=0.
\]
Since the scalar curvature $s$ over $(-\infty, \infty) \times Y$ is nonnegative, $\mathcal{E}(\Gamma)=0$ implies that $\Gamma=(B, 0)$ with $B$ a flat spin$^c$ connection on $S$. Thus $\mathfrak{b}_-=\mathfrak{b}_+$ and $\Gamma$ is the constant flowline.  
\end{proof}

In order to prove the Convergence (ii) in the Definition \ref{3.1} we recall a result concerning the drop of the Chern-Simons-Dirac functional near a non-singular point $[\mathfrak{b}] \in \chi(Y)$. The notion being non-singular we are using here is also referred to as being Morse-Bott in the literature, i.e. the Hessian $\Hess \mathcal{L}|_{[\mathfrak{b}]}$ is nondegenerate in the normal direction of $T_{[\mathfrak{b}]}\chi(Y)$ in $T_{[\mathfrak{b}]}\mathcal{B}_l(Y)$. Putting $[\mathfrak{b}]$ into Coulomb gauge, from (\ref{2.4}) we see that $[\mathfrak{b}]$ is Morse-Bott if and only if $\ker D_B=0$ with $\mathfrak{b}=(B, 0)$. 

\begin{lem}\label{3.6.6}
(\cite[Lemma 13.5.2]{KM1}) Let $[\mathfrak{b}] \neq \theta \in \chi(Y)$ be a non-singular critical point. There exists a neighborhood $V_{[\mathfrak{b}]}$ of $[\mathfrak{b}] $ in $\chi(Y)$ and $\delta_{\mathfrak{b}} >0$ such that for any monopole $\Gamma \in \M([t_1, t_2] \times Y)$ with corresponding flowline $\gamma(t) \in V_{[\mathfrak{b}]}$ for all $t \in [t_1, t_2]$, one has  
\[
-|\mathcal{L}(\gamma(t_1))| \cdot e^{\delta_{\mathfrak{b}}(t-t_2)} \leq \mathcal{L}(\gamma(t)) \leq |\mathcal{L}(\gamma(t_2))| \cdot e^{-\delta_{\mathfrak{b}} (t-t_1)}.
\]
\end{lem}

\begin{rem}
The result is originally stated for $[\mathfrak{b}]$ being a Morse critical point. The exponential decay for Morse-Bott critical points is deduced in \cite[Lemma 2.13]{Lin}. Since we have normalized $\mathcal{L}([\mathfrak{b}])$ to be $0$, there is no such term involved. If we choose a neighborhood $V_{\theta}$ of $\theta$ in $\chi(Y)$, the exponent $\delta_{\mathfrak{b}}$ can be chosen to be uniform for all $[\mathfrak{b}] \in \chi(Y) \backslash V_{\theta}$.
\end{rem}

\begin{proof}[Proof of Theorem \ref{G1}]
Let $\Gamma_n=(A_n, \Phi_n)$ be a sequence of irreducible monopoles over $X_{T_n}$, $T_n$ an increasing sequence of real numbers with $T_n \to \infty$. We may assume $T_n \in \Z$ for the sake of simplifying narration. Each $\Gamma_n$ decomposes into 3 parts:
\[
\Gamma^1_n:=\Gamma_n|_{M}, \Gamma^2_n:=\Gamma_n|_{I_{T_n}}, \Gamma^3_n:=\Gamma_n|_N.
\]
We also write $\gamma_n(t)=\Gamma_n|_{\{t \} \times Y}$ for $t \in [-T_n, T_n]$ the associated flowline. Note that over a closed $4$-manifold the energy for a monopole can be written as 
\[
\mathcal{E}_{\beta_n}(\Gamma_n)={1 \over 4} \int_{X_{T_n}} (F_{\A} -4d\beta_n) \wedge (F_{\A}- 4d\beta_n) = -\pi^2 c_1^2(\s)[X]=:C. 
\] 
Thus the energy of $\Gamma_n^i, i=1, 2, 3$, are uniformly bounded by the constant $C$. Take families of neighborhoods $U_{\Gamma_{\mathfrak{b}}}$ and $V_{[\mathfrak{b}]}$ for $\Gamma_{\mathfrak{b}}$ and $[\mathfrak{b}]$ respectively as above. Let $\epsilon_o$ be the constant as in Lemma \ref{3.4}. Given $t \in \R$ write $\tau_t \Gamma^2_n (s, y)=\Gamma^2_n (s+t, y)$ the translation with $s \in [-T_n-t, T_n-t]$. Let 
\[
S_n=\{ p \in [-T_n, T_n-1] \cap \Z : \tau_{p+1}\Gamma_n^2|_{[p, p+2]} \notin U_{\chi(Y)} \}.
\]
From Lemma \ref{3.4} we know that each $S_n$ contains at most ${C \over \epsilon_o}$ elements. After passing to a subsequence we may assume $S_n$ has exactly $m$ elements, $p^n_1 < p^n_2 < ... <p^n_m$, for each $n$. Since $\Gamma_n$ is irreducible, unique continuation of Dirac operators implies that $\Gamma_n|_{Y_{-T_n}} \notin \chi(Y)$ and $\Gamma_n|_{Y_{T_n}} \notin \chi(Y)$ for all $n$. For each $j$ in the rang $1 \leq j \leq m-1$, the set of differences $\{p^n_{j+1} - p^n_j\}_{n \geq 1}$ is either bounded or unbounded. In the first case since it's integer valued we can pass to a subsequence so that $p^n_{j+1} - p^n_{j}$ is independent of $n$. In the second case we can pass to a subsequence so that $p^n_{j+1}-p^n_j$ increases to $\infty$. Thus we can define an equivalence relation "$\sim$" on the finite set $\{1, 2, ..., m\}$ by 
\[
j \sim j' \Longleftrightarrow \lim_{n \to \infty} p^n_j - p^n_{j'} < \infty. 
\]
For each equivalence class we pick a representative $j_i$, $i=1, ..., d$, ordered so that $j_i < j_{i+1}$. We let 
\[
a_i^n=\min \{p^n_j : j \sim j_i\} \text{ and } b_i^n=\max \{p^n_j : j \sim j_i\}. 
\]
Then the length of the intervals $I_i^n=[a_i^n, b_i^n+2]$ is independent of $n$ (possibly 0), and the length of the intervals $J_i^n=[b_i^n+2, a^n_{i+1}]$ approaches $\infty$ as $n \to \infty$. The choice of $J_i^n$ is made so that $\Gamma_n|_{J_i^n} \in U_{\chi(Y)}$. We also write $J_0^n=[-T_n, b^n_1+2]$ and $J^n_{d}=[a^n_d, T_n]$. From the construction of $a_1^n$, after passing to a subsequence, either $a_n^1$ is bounded or $a_n^1 \to \infty$. In the first case we take $J_0^n=[-T_n, b^n_2+2]$. Thus we may assume the length of $J_0^n$ goes to $\infty$ as well. The same remark applies to that of $J_d^n$. 
For $1 \leq i \leq d-1$ let $c_i^n={a^n_{i+1}+ b^n_i+2 \over 2}$ be the middle point of the interval $J_i^n$, $l_i^n=a^n_{i+1}-b^n_i-2$ be the length of the interval $J_i^n$. Then $\tau_{c_i^n} \Gamma_n^2|_{J_i^n}$ is a sequence of monopoles over $[-l^n_i, l^n_i ] \times Y$ with uniformly bounded energy and the length $l^n_i \to \infty$. 

Combining Lemma \ref{3.3} and Lemma \ref{3.5} we conclude that there exists $[\mathfrak{b}_i] \in \chi(Y)$ for each $i$ in the range $0 \leq i \leq d$ and $[\Gamma_o] \in \M(M_{\infty}, [\tilde{\mathfrak{b}}_0]), [\Gamma_o'] \in \M([\tilde{\mathfrak{b}}_d], N_{\infty})$ such that up to gauge
\begin{enumerate}
\item[\upshape (i)] $\Gamma_n|_{M_{b_1^n+2+T_n}} \to \Gamma_o$, $\Gamma_n|_{N_{T_n-a^n_d}} \to \Gamma'_o$ in $L^2_{k, loc}$-topology as $n \to \infty$, 
\item[\upshape (ii)]
$ \tau_{c_{i-1}^n} \Gamma_n^2|_{J_i^n} \longrightarrow \Gamma_{\mathfrak{b}_i} \text{ in $L^2_{k, loc}$-topology}, 1 \leq i \leq d-1,$ as $n \to \infty$. 
\end{enumerate}

Now we claim that $[\mathfrak{b}_i]=[\mathfrak{b}_{i+1}]$ for all $i$ in the range $0 \leq i \leq d-1$. Suppose this is not true. There exists some $i$ with $0 \leq i \leq d-1$ such that $[\mathfrak{b}_i] \neq [\mathfrak{b}_{i+1}]$. Note that $\Gamma_n|_{I_i^n \times Y} \notin U_{\chi(Y)}$, from Lemma \ref{3.4} we know that 
\[
\mathcal{L}([\gamma_n(t)]) - \mathcal{L}([\gamma_n(s)])={1 \over 2} \mathcal{E}(\Gamma_n|_{I_i^n\times Y}) \geq {l_i  \over 2} \epsilon_0, \forall t \in [-T_n, a_i^n] \text{ and } s \in [b_i^n+2, T_n],
\]
where $l_i= b_i^n+2-a_i^n$ is the length of $I_i^n$. Since $L^2_k \hookrightarrow C^0$ is continuous, we conclude that $\gamma(c_{i-1}^n) \to \mathfrak{b}_i$ and $\gamma(c_i^n) \to \mathfrak{b}_{i+1}$ in $C^0$-topology. The continuity of the Chern-Simons-Dirac functional implies that 
\[
\mathcal{L}(\gamma_n(c_{i-1}^n)) - \mathcal{L}(\gamma_n(c_i^n)) \to 0 \text{ as } n \to \infty.
\]
However $\mathcal{L}(\gamma_n(c_{i-1}^n) )- \mathcal{L}(\gamma_n(c_i^n))  \geq {l_i  \over 2}  \epsilon_0>0$ for all $n$, which is a contradiction. From the argument we also conclude that $l_i=0$ for all $i$, which is impossible unless $d=1$. Thus we conclude 
\[
\Gamma_n|_{M_{T_n}} \to \Gamma_o, \Gamma_n|_{N_{T_n}} \to \Gamma'_o \text{ in $L^2_{k, loc}$-topology as $n \to \infty$}, 
\]
which is the convergence condition (i) in Definition \ref{3.1}. 

Write $\mathfrak{b}_0=\mathfrak{b}$. Now we know that $[\Gamma_o] \in \M(M_{\infty}, [\tilde{\mathfrak{b}}])$, $[\Gamma'_o] \in \M([\tilde{\mathfrak{b}}], N_{\infty})$. From the second assumption in the statement we know that $[\mathfrak{b}] \neq \theta$. Write $\Gamma_{\mathfrak{b}}=(A_{\mathfrak{b}}, 0)$ for the constant monopole given by $\mathfrak{b}$ and $\gamma_{\mathfrak{b}}$ be the corresponding flowline. We want to show the convergence condition (ii) in Definition \ref{3.1} holds for $\{\Gamma_n\}$ after possibly choosing gauge transformations and passing to a subsequence. 

Now for any neighborhood $U_{\Gamma_{\mathfrak{b}}}$ and $V_{[\mathfrak{b}]}$, apply Lemma \ref{3.4} and the argument above, we get $T_o>0$ independent of $n$ such that, after passing to a subsequence, 
\[
[\Gamma_n|_{I_{T_n-T_o}}] \in U_{\Gamma_{\mathfrak{b}}} \text{ for all $T_n > T_o$ }.
\]
In particular we choose $V_{[\mathfrak{b}]}$ to satisfy the requirement of Lemma \ref{3.6.6}. \cite[Equation (2.10)]{Lin} shows that after fixing some gauge,
\begin{equation}
\| \Gamma_n|_{[i-1, i+1] \times Y} -\Gamma_{\mathfrak{b}} \|^2_{L^2_k([i-1, i+1] \times Y)}  \leq K_1 (\mathcal{L}(-T_o+T_n)+\mathcal{L}(T_n-T_o) )
\end{equation}
Invoke Lemma \ref{3.6.6} and sum over all integers $i$ in the range $[T_o-T_n+1, T_n-T-1]$ to get 
\begin{equation}\label{3.1.1}
\|\Gamma_n|_{I_{T_n-T_o}} - \Gamma_{\mathfrak{b}}\|^2_{L^2_k(I_{T_n-T_o})} \leq K_2 (T_n-T_o)e^{-\delta_{\mathfrak{b}} ({T_n -T_o \over2})}.
\end{equation}
Thus for any $\epsilon >0$, one can choose $n$ large enough so that the left hand side of (\ref{3.1.1}) is bounded by $\epsilon^2$. This completes the proof of convergence.
 
Lastly we show that $[\Gamma_o]=[(A_o, \Phi_o)] \in \M(M_{\infty}, [\tilde{\mathfrak{b}}])$ is irreducible. We note that since by assumption $N_{\infty}$ admits nonnegative scalar curvature, $[\Gamma'_o]=[(A'_o, 0)]$ has to be reducible. Suppose $[\Gamma_o]$ is reducible as well, i.e. $\Phi_o \equiv 0$. Then after possibly some gauge transformation the convergence result proved above gives us a sequence of non-zero spinors $\Phi_n \in \ker D^+_{A_n}$ such that for any given $\epsilon>0$ for all sufficiently large $n$ and $T_n >T_o$
\[
\|\Phi_n\|_{L^2_k(X_{T_n})} =\|\Phi_n|_{M_{T_o}}\|_{L^2_k} +\|\Phi_n|_{I_{T_n-{T_o}}}\|_{L^2_k} +\|\Phi_n|_{N_{T_o}}\|_{L^2_k} < \epsilon. 
\]
Note that $H^1(X_T; \Z) \to H^1(M_T;\Z)$ is injective, we denote the image of $1_{X_T}$ in $H^1(M_T;\Z)$ by $1_{M_T}$. We choose smooth functions $f_T: X_T \to S^1$ with $[df_T] = 1_{X_T}$ such that $f_T|_{M_T}$ converges in $C^{\infty}_{loc}$-topology to a smooth function $f_{M_{\infty}}: M_{\infty} \to S^1$ with $[df_{M_{\infty}}]$ representing the element $1_{M_{\infty}}$ which restricts to $1_{M_T}$ for every $T$. Note that $A_o$ solves the equation 
\[
F^+_{\A_o} = 2d^+ \beta_M.
\]
Since $A_o$ is the limit of connections on $X_{T_n}$, it has the form 
\begin{equation}
A_o= A_M + \beta_M- \ln z_M \cdot df_{M_{\infty}}, 
\end{equation}
where $A_M$ is the flat spin connection on $(M_{\infty}, \s)$. Similarly on $(N_{\infty}, \s)$ we get 
\begin{equation}
A_o'=A_N+\beta_N- \ln z_N \cdot df_{M_{\infty}}.
\end{equation}
Since $\bar{\partial}_{+}[A_o] = \bar{\partial}_{-}[A'_o] = [B]$, and $\beta_M$, $\beta_N$ are all compactly supported, we conclude that $z_N=z_M$, which we write as $z_o$. The convergence of $[\Gamma_n]$ enables us to write 
\[
A_n=A_{X_T} +\beta_T -\ln z_o \cdot df_{T_n} + a_n.
\]
The equation that $A_n$ satisfies 
\[
{1 \over 2}F^+_{\A_n} - 2d^+ \beta = \rho^{-1}(\Phi_n \Phi_n^*)_0
\]
gives us $\|a_n\|_{L^2_k} \leq K_3 \epsilon$. This shows that $D^+_{A_n}=D^+_{z_o, \beta}(X_{T_n}) + \rho(a_n)$ has nontrivial kernel given by $\Phi_n$. By choosing $\epsilon$ small, this would violate Lemma \ref{3.8.8}. Thus $\Gamma_o$ is irreducible and the theorem is proved. 
\end{proof}



\bibliographystyle{plain}
\bibliography{References}

\end{document}